\documentclass[12pt]{amsart}
\usepackage{amssymb,setspace,mathrsfs,bm}
\usepackage[author-year]{amsrefs}

\title{Harmonic Analysis of Additive
    L\'evy Processes}
   \thanks{Research supported in part by a grant from the National Science
   Foundation}
\author[Khoshnevisan]{Davar Khoshnevisan}
   \address{Davar Khoshnevisan:\ Department of Mathematics,
       The University of Utah,
      155 S. 1400 E. Salt Lake City, UT 84112--0090}
   \email{davar@math.utah.edu}
   \urladdr{http://www.math.utah.edu/\~{}davar}
\author[Xiao]{Yimin Xiao}
    \address{Department of Statistics and Probability, A-413 Wells
        Hall, Michigan State University,
        East Lansing, MI 48824}
    \email{xiao@stt.msu.edu}
    \urladdr{http://www.stt.msu.edu/\~{}xiaoyimi}
\date{%
    June 24, 2007}

\theoremstyle{plain}{
\newtheorem{theorem}{Theorem}[section]}
\theoremstyle{plain}{
\newtheorem{proposition}[theorem]{Proposition}}
\theoremstyle{plain}{
   \newtheorem{lemma}[theorem]{Lemma}}
\theoremstyle{plain}{
   \newtheorem{corollary}[theorem]{Corollary}}
\theoremstyle{definition}{
   \newtheorem{definition}[theorem]{Definition}}
\theoremstyle{definition}{
   }
\theoremstyle{remark}{
   \newtheorem{remark}[theorem]{Remark}}
\theoremstyle{definition}{
   \newtheorem{problem}{Open problem}}
\numberwithin{equation}{section}
\newcommand{\dimh}{\dim_{_{\rm H}}}

\newcommand{\1}{\mathbf{1}}
\renewcommand{\P}{\mathrm{P}}

\newcommand{\E}{\mathrm{E}}
\newcommand{\R}{\mathbf{R}}
\newcommand{\e}{\epsilon}
\renewcommand{\Re}{\mathop{\mathrm{Re}}}
\renewcommand{\Im}{\mathop{\mathrm{Im}}}
\renewcommand{\mathcal}{\mathscr}
\newcommand{\X}{\mathfrak{X}}
\begin{document}
\onehalfspacing

\begin{abstract}
    Let $X_1,\ldots,X_N$ denote
    $N$ independent $d$-dimensional
    L\'evy processes, and consider
    the $N$-parameter random field
    \[
        \X(\bm{t}):= X_1(t_1)+\cdots+X_N(t_N).
    \]
    First we demonstrate that for all nonrandom Borel sets
    $F\subseteq\R^d$, the Minkowski sum $\X(\R^N_+)\oplus F$,
    of the range $\X(\R^N_+)$ of $\X$
    with $F$, can have positive $d$-dimensional
    Lebesgue measure if and only if a
    certain capacity of $F$ is positive.
    This improves our earlier joint effort
    with Yuquan Zhong
    \ycite{KXZ:03} by removing a symmetry-type
    condition there. Moreover, we show that under
    mild regularity conditions, our necessary and
    sufficient condition can be recast in terms of
    one-potential densities. This rests on developing
    results in classical [non-probabilistic]
    harmonic analysis that might
    be of independent interest.
    As was shown in \fullocite{KXZ:03},
    the potential theory of the type studied here has a large number of
    consequences in the theory of L\'evy processes.
    We present a few new consequences here.
\end{abstract}
\maketitle
\section{\bf Introduction}
\subsection{Background}
It is known that, for all integers $d\ge 2$, the range of
$d$-dimensional Brownian motion has zero Lebesgue
measure. See
\ocite{Levy} for $d=2$, \ocite{Ville}
for $d=3$, and \ocite{Kakutani:44a} for the remaining
assertions. There is a perhaps better-known, but
equivalent, formulation of this theorem: When $d\ge 2$,
the range of $d$-dimensional Brownian motion does not
hit points.
\ocite{Kakutani:44b} has generalized this by
proving that, for all integers $d\ge 1$,
the range of $d$-dimensional Brownian motion
can hit a nonrandom Borel set $F\subseteq\R^d$
if and only if $\mathrm{cap}(F)>0$, where
$\mathrm{cap}$ denotes, temporarily, the logarithmic
capacity if $d=2$ and the Riesz capacity of
index $d-2$ if $d\ge 3$; the case $d=1$ is
elementary. [Actually,
Kakutani's paper discusses only
the planar case. The theorem, for $d\ge 3$, is pointed out in
\fullocite{DET:50}.]

Kakutani's theorem is the starting point of
a deep \emph{probabilistic potential theory} initiated by
Hunt \ycites{hunt:III,hunt:II,hunt:I,hunt:56}.
The literature on this topic
is rich and quite large; see, for example, the books by
\ocite{BG}, \ocite{Doob}, \fullocite{FOT}, \ocite{Getoor:book},
and \ocite{rockner}, together with
their combined bibliography.

One of the central assertions of probabilistic
potential theory is that a nice Markov process will hit
a nonrandom measurable set $F$ if and only if $\mathrm{cap}(F)>0$,
where $\mathrm{cap}$ is
a certain natural capacity in the sense of G. Choquet \cite{DM}*{Chapter III,
pp.\ 51--55}. Moreover, that capacity
is defined solely, and fairly explicitly,
in terms of the Markov process itself.

There are interesting examples where $F$ is itself
random. For instance, suppose
$X$ is $d$-dimensional standard
Brownian motion, and $F=Y((0\,,\infty))$
is the range---minus the starting point---of
an independent standard Brownian motion $Y$ on $\R^d$.
In this particular case, it is well known that
\begin{equation}\label{eq:DEK}
    \P\left\{ X(s)=Y(t)\text{ for some
    $s,t>0$}\right\}>0\quad
    \text{if and only if}\quad
    d\le 3.
\end{equation}
This result was proved by \ocite{Levy}
for $d=2$, \ocite{Kakutani:44a}
for $d\ge 5$, and \fullocite{DET:50} for $d=3,4$.
Peres \ycites{Peres:96b,Peres:96a} and \ocite{Kh:2003} contain
different elementary proofs of this fact.

There are many generalizations of \eqref{eq:DEK} in the
literature. For example, \fullocite{DET:54} proved that the paths
of an arbitrary number of independent planar Brownian motions can
intersect. While LeGall (1987) proved that the trajectories of a planar
Brownian motion can intersect itself countably many times. And
\fullocite{DvoretzkyErdosKakutaniTaylor} showed that three
independent Brownian-motion trajectories in $\R^d$ can intersect
if and only if $d\le 2$. For other results along these lines, see
Hawkes \ycites{Hawkes:78b,Hawkes:77,Hawkes:76:77}, Hendricks
\ycites{Hendricks:99,Hendricks:74}, \ocite{Kahane:SRSF}*{Chapter
16, Section 6}, Lawler \ycites{Lawler:89,Lawler:85,Lawler:82},
\fullocite{PPS}, \ocite{LeGall92},
Peres \ycites{Peres:99,Peres:96b,Peres:96a},
\ocite{rogers}, \ocite{Tongring}, and their combined bibliography.

For a long time, a good deal of effort was concentrated
on generalizing \eqref{eq:DEK}
to other concrete Markov processes than Brownian motion.
But the problem of deciding when the paths of
$N$ independent (general but nice) Markov processes can intersect
remained elusive.
It was finally settled, more than three
decades later, by
\ocite{FitzsimmonsSalisbury}, whose approach was to consider
the said problem as one about a certain multiparameter Markov
process.

To be concrete, let us consider the case $N=2$,
and let $X:=\{X(t)\}_{t\ge 0}$ and $Y:=\{Y(t)\}_{t\ge 0}$ denote
two independent (nice) Markov processes on a (nice) state
space $S$. The starting point of the work of Fitzsimmons and
Salisbury is the observation
that $\P \{ X(s)=Y(t)\text{ for some $s,t>0$} \}>0$
if and only if the two-parameter Markov process
$X\otimes Y$ hits the diagonal
$\text{diag }S:=\{x\otimes x:\,x\in S\}$ of $S\times S$,
where
\begin{equation}\label{eq:2par:PLP}
    (X\otimes Y)(s\,,t) := \left( \begin{matrix}
        X(s)\\
       Y(t)
    \end{matrix}\right)\qquad\text{for all
    $s,t\ge 0$}.
\end{equation}
In the special case that $X$ and $Y$ are L\'evy processes,
the Fitzsimmons--Salisbury theory was used to solve
the then-long-standing Hendricks--Taylor conjecture
\cite{HendricksTaylor}.

The said connection to multiparameter processes is of
paramount importance in the Fitzsimmons--Salisbury theory, and appears earlier
in the works of Evans \ycites{Evans:87a,Evans:87b}.
See also \fullocite{LeGallShiehRosen}
and
Salisbury \ycites{Salisbury:96,Salisbury:92,Salisbury:87}.
\ocite{Walsh}*{pp.\ 364--368} discusses a connection between
$X\otimes Y$ and the Dirichlet
problem for the biLaplacian $\Delta\otimes\Delta$
on the bi-disc of $\R^d\times\R^d$.

The Fitzsimmons--Salisbury theory was refined and generalized in different
directions by \ocite{Hirsch},
Hirsch and Song \ycites{HS:99,HS:96,HS:95a,HS:95b,HS:95c,HS:95d,HS:94},
and Khoshnevisan
\ycites{Kh:book,Kh:1999}. See also \ocite{Ren}
who derives an implicit-function theorem in classical Wiener space
by studying a very
closely-related problem.

The two-parameter process $X\otimes Y$ itself was introduced earlier
in the works of \ocite{Wolpert}, who used $X\otimes Y$ to build a
$(\phi^\kappa)_2$ model of the Euclidean field theory. This too
initiated a very large body of works. For some of the earlier
examples, see the works by \ocite{Aizenman}, \ocite{AlbeverioZhou},
Dynkin \ycites{dyn1,dyn2,dyn3,dyn4,dyn5,dyn6,dyn7,dyn8,dyn9,dyn10},
\ocite{FelderFrohlich}, Rosen \ycites{rosen:1983,rosen:1984}, and
Westwater \ycites{Westwater:I,Westwater:II,Westwater:III}. [This is
by no means an exhaustive list.]

In the case that $X$ and $Y$ are L\'evy processes on $\R^d$ [i.e.,
have stationary independent increments], $X\otimes Y$ is an example
of the so-called additive L\'evy processes. But as it turns out, it
is important to maintain a broader perspective and consider more
than two L\'evy processes. With this in mind, let $X_1,\ldots,X_N$
denote $N$ independent L\'evy processes on $\R^d$ such that each
$X_j$ is normalized via the L\'evy--Khintchine formula
\cite{bertoin:book,sato}:
\begin{equation}
    \E \exp\left(i\xi\cdot X_j(t)\right) =
    \exp\left(-t\Psi_j(\xi)\right)\qquad
    \text{for all $t\ge 0$ and $\xi\in\R^d$.}
\end{equation}
The function $\Psi_j$ is called the
\emph{characteristic exponent}---or \emph{L\'evy exponent}---of
$X_j$, and its defining property is that
$\Psi_j$ is a negative-definite function
\cites{Schoenberg:38,bergforst}.

\subsection{The main results}
The main object of this paper is
to develop some basic probabilistic potential theory
for the following $N$-parameter random field $\X$
with values in $\R^d$:
\begin{equation}
    \X(\bm{t}) :=X_1(t_1) + \cdots + X_N(t_N)
    \qquad\text{for all $\bm{t}:=(t_1\,,\ldots,t_N)\in\R^N_+$.}
\end{equation}
On a few occasions we might write $( \oplus_{j=1}^N X_j) (\bm{t})$
in place of $\X(\bm{t})$, as well.

The random field $\X$ is a so-called
\emph{additive L\'evy process}, and is characterized
by its multiparameter L\'evy--Khintchine formula:
\begin{equation}
    \E \exp\left(i\xi\cdot\X(\bm{t})\right)
    =\exp\left(-\bm{t}\cdot\bm{\Psi}(\xi)\right)
    \qquad\text{for all $\bm{t}\in\R^N_+$ and
    $\xi\in\R^d$};
\end{equation}
where $\bm{\Psi}(\xi):=(\Psi_1(\xi)\,,\ldots,\Psi_N(\xi))$ is the
\emph{characteristic exponent} of $\X$. Our goal is to describe the
potential-theoretic properties of $\X$ solely in terms of its
characteristic exponent $\bm{\Psi}$. Thus, it is likely that our
harmonic-analytic viewpoint can be extended to study the potential
theory of more general multiparameter Markov processes that are
based on the Feller processes of Jacob
\ycites{jacob3,jacob2,jacob1}.

In order to describe our main results let us first
consider the kernel
\begin{equation}
    K_{\bm{\Psi}}(\xi) := \prod_{j=1}^N
    \Re\left(\frac{1}{1+\Psi_j(\xi)}\right)
    \qquad\text{for all $\xi\in\R^d$}.
\end{equation}
When $N=1$, this kernel plays a central role in the works
of \ocite{Orey} and \ocite{Kesten}. The kernel for
general $N$ was introduced first by \ocite{Evans:87b}; see also
\fullocite{KXZ:03}.

Based on the kernel $K_{\bm{\Psi}}$, we define, for all Schwartz
distributions $\mu$ on $\R^d$,
\begin{equation}
    I_{\bm{\Psi}}(\mu) := \frac{1}{(2\pi)^d}
    \int_{\R^d} |\hat{\mu}(\xi)|^2
    K_{\bm{\Psi}}(\xi) \, d\xi.
\end{equation}
We are primarily interested in the case where $\mu$ is a real-valued
locally integrable function, or a $\sigma$-finite Borel measure on
$\R^d$. In either case, we refer to $I_{\bm{\Psi}}(\mu)$ as the
\emph{energy} of $\mu$. Our notion of energy corresponds to a
\emph{capacity} $\mathrm{cap}_{\bm{\Psi}}$, which is the following
set function: For all Borel sets $F\subseteq\R^d$,
\begin{equation}
    \mathrm{cap}_{\bm{\Psi}}(F) := \left[\inf_{%
    \mu\in\mathcal{P}_c(F)}
    I_{\bm{\Psi}}(\mu)\right]^{-1},
\end{equation}
where $\mathcal{P}_c(F)$ denotes the collection of all
compactly-supported Borel probability measures on $F$,
$\inf\varnothing:=\infty$, and $1/\infty:=0$.

The following is the first central result of this paper.
Here and throughout,
$\lambda_k$ denotes $k$-dimensional Lebesgue measure
on $\R^k$ for all integers $k\ge 1$.

\begin{theorem}\label{th:main}
    Let $\mathfrak{X}$ be an $N$-parameter
    additive L\'evy process on $\R^d$
    with exponent $\bm{\Psi}$. Then, for
    all Borel sets $F\subseteq\R^d$,
    \begin{equation}
        \E\left[ \lambda_d\left(
        \X(\R^N_+)\oplus F\right) \right]>0\quad
        \text{if and only if}\quad
        \mathrm{cap}_{\bm{\Psi}}(F)>0.
    \end{equation}
\end{theorem}

\begin{remark}
    \begin{enumerate}
        \item
            Theorem \ref{th:main} in the one-parameter
            setting is still very interesting, but much
            easier to derive. See \ocite{Kesten} for the case that $F:=\{0\}$ and
            \ocite{Hawkes:84} for general $F$. For a scholarly
            pedagogic account see \ocite{bertoin:book}*{p.\ 60}.
        \item
            One can view
            Theorem \ref{th:main} as a contribution to the theory
            of Dirichlet forms for a class of infinite-dimensional
            L\'evy processes. These L\'evy processes are
            in general non-symmetric. \ocite{rockner} describes a
            general theory of Dirichlet forms for nice
            infinite-dimensional Markov processes that are
            not necessarily symmetric. It would be
            interesting to know if the processes of
            the present paper lend themselves to the analysis of
            the general theory of Dirichlet forms.
            We have no conjectures along these lines.\qed
    \end{enumerate}
\end{remark}

Our earlier collaborative effort with
Yuquan Zhong \ycite{KXZ:03} yielded the conclusion of Theorem \ref{th:main}
under an exogenous technical condition on
$X_1,\ldots,X_N$. A first aim of this paper is to establish the fact that
Theorem \ref{th:main} holds in complete generality.
Also, we showed in our earlier
works 
\citelist{\cite{KXZ:05}\cite{KX04a},
\fullocite{KXZ:03}}
that such a theorem has a large number
of consequences, many of them in the classical theory
of L\'evy processes itself. Next we describe a
few such consequences that are nontrivial due to their
intimate connections to
harmonic analysis.

Our next result provides a criterion for a Borel set $F\subseteq\R^d$
to contain intersection points of $N$ independent L\'evy processes.
It completes and complements the well-known results of
\ocite{FitzsimmonsSalisbury}. See also Corollary \ref{cor:FS--1}
and Remark \ref{rem:FS} below.

\begin{theorem}\label{th:FS}
    Let $X_1,\ldots,X_N$ be independent L\'evy
    processes on $\R^d$, and assume that each $X_j$
    has a one-potential density $u_j:\R^d\to\bar\R_+$
    such that $u_j(0)>0$.
    Then, for all nonempty Borel sets $F\subseteq\R^d$,
    \begin{equation}
        \P\left\{ X_1(t_1)=\cdots= X_N(t_N)
        \in F\
        \text{ for some $t_1,\ldots,t_N>0$}\right\}>0
    \end{equation}
    if and only if there exists a compact-support
    Borel probability measure $\mu$ on $F$ such that
    \begin{equation}
        \int_{\R^d}\cdots\int_{\R^d}\left|
        \hat{\mu}\left(\xi^1+\cdots+\xi^N\right)\right|^2
        \prod_{j=1}^N \Re\left( \frac{1}{1+\Psi_j(\xi^j)}\right)
        \, d\xi^1\cdots\, d\xi^N < \infty.
    \end{equation}

    Suppose, in addition, that every $u_j$ is continuous
    on $\R^d$, and finite on $\R^d\setminus\{0\}$. Then, another
    equivalent condition is that there
    exists a compact-support probability measure
    $\mu$ on $F$ such that
    \begin{equation}
        \iint \prod_{j=1}^N \left(
        \frac{u_j\left(x-y\right)
        +u_j\left(y-x\right)}{2}\right)
        \mu (dx)\,  \mu (dy) <\infty.
    \end{equation}
\end{theorem}

In order to describe our next contribution,
let us recall that the \emph{one-potential measure}
$U$ of a L\'evy process $X:=\{X(t)\}_{t\ge 0}$
on $\R^d$ is defined
as
\begin{equation}
    U(A) := \int_0^\infty \P\left\{
    X(t)\in A\right\} e^{-t}\, dt,
\end{equation}
for all Borel sets $A\subseteq\R^d$. Next
we offer a two-parameter ``additive variant'' which requires
fewer technical conditions than Theorem \ref{th:FS}.

\begin{theorem}\label{th:bertoin}
    Suppose $X_1$ and $X_2$ are independent L\'evy processes
    on $\R^d$ with respective one-potential measures
    $U_1$ and $U_2$. Suppose $U_1(dx)=
    u_1(x)\, dx$, where $u_1:\R^d\to\bar\R_+$, and
    $u_1*U_2>0$ almost everywhere.
    Then, for all
    Borel sets $F\subseteq\R^d$,
    \begin{equation}\label{Eq:Inverse}
        \P\left\{  X_1(t_1)+
        X_2(t_2)\in F\text{ for some }
        t_1,t_2>0\right\}>0
    \end{equation}
    if and only if there exists a compact-support
    Borel probability measure $\mu$ on $F$ such that
    \begin{equation}\label{Eq:Inverse2}
        \int_{\R^d} |\hat\mu(\xi)|^2\Re\left( \frac{1}{1+\Psi_1(\xi)}\right)
        \Re\left( \frac{1}{1+\Psi_2(\xi)}\right)\, d\xi<\infty.
    \end{equation}

    Suppose, in addition, that $u_1$ is continuous on $\R^d$,
    and finite on $\R^d\setminus\{0\}$. Then, \eqref{Eq:Inverse}
    holds
    if and only if there exists a compact-support
    probability measure $\mu$ on $F$ such that
    \begin{equation}\label{Eq:Inverse3}
        \iint Q(x-y)\,\mu(dx)\,\mu(dy)<\infty,
    \end{equation}
    where
    \begin{equation}\label{Def:Q}
        Q(x):=\int_{\R^d} \left[
        \frac{u_1(x+y) + u_1(x-y) + u_1(-x + y) +u_1(-x-y)}{4}\right]\,  U_2(dy)
    \end{equation}
    for all $x\in\R^d$.
\end{theorem}
Among other things, Theorem \ref{th:bertoin} confirms a
conjecture of \ocite{bertoin:1999} and \ocite{bertoin:st-flour}*{p.\
49}; see Remark \ref{rem:bertoin} for details.

Finally we mention a result on the Hausdorff dimension of the set of
intersections of the sample paths of L\'evy processes.

\begin{theorem}\label{th:dimH}
    Let $X_1,\ldots,X_N$ be independent L\'evy
    processes on $\R^d$, and assume that each $X_j$
    has a one-potential density $u_j:\R^d\to\bar\R_+$
    such that $u_j(0)>0$. Then, almost surely on $\big\{\cap_{k=1}^N X_k(\R_+)
    \ne \varnothing\big\}$,
    \begin{equation}\label{co:dimH:1}\begin{split}
        &\dimh \bigcap_{k=1}^N X_k(\R_+)\\
        &\quad =\sup\left\{
            s\in(0\,,d):\ \int_{(\R^d)^N} \prod_{j=1}^N
                \Re\left( \frac{1}{1+\Psi_j(\xi^j)}\right)
                \, \frac{ d\xi}{1+\|\xi^1+\cdots+\xi^N\|^{d-s}}<\infty\right\},
    \end{split}\end{equation}
    where $\sup\varnothing:=0$.
    Suppose, in addition, that the $u_j$'s are continuous on $\R^d$,
    and finite on $\R^d\setminus\{0\}$. Then, almost surely on $\big\{\cap_{k=1}^N X_k(\R_+)
    \ne \varnothing\big\}$,
    \begin{equation}\label{co:dimH:2}
        \dimh \bigcap_{k=1}^N X_k(\R_+)
        \quad = \sup\left\{s\in(0\,,d):\ \int_{\R^d}
        \prod_{j=1}^N \left( \frac{u_j(z)+u_j(-z)}{2}\right)
        \, \frac{dz}{\|z\|^s}<\infty
        \right\}.
    \end{equation}
\end{theorem}

In the remainder of the paper we prove Theorem \ref{th:main}
and its stated corollaries in the order in which they are presented.
Finally, we conclude by two zero-one laws for the Lebesgue measure
and capacity of the range of an additive L\'evy process,
that, we believe, might have independent interest.

We end this section with four problems and conjectures.

\begin{problem}\label{pbm1}
    Throughout this paper, we impose continuity
    conditions on various one-potential densities. This is
    mainly because we are able to develop general harmonic-analytic results
    only for kernels that satisfy some regularity properties. Can
    the continuity conditions be dropped? We believe the answer
    is ``yes.'' This is motivated, in part, by
    the following fact, which follows from inspecting the
    proofs: The condition
   ``$u$ is continuous on $\R^d$ and finite on $\R^d\setminus\{0\}$'' is
    used only for proving the ``if" portions in the second parts of
    Theorems \ref{th:FS} and \ref{th:bertoin}.
\end{problem}

\begin{problem}
    Jacob \citelist{\ycite{jacob3} \ycite{jacob2} \ycite{jacob1}}
    has constructed a very large
    class of Feller processes that behave locally like L\'evy
    processes.
    Moreover, his construction is deeply connected to harmonic analysis.
    Because the results of the present paper involve mainly the
    local structure of L\'evy processes, and are inextricably
    harmonic analytic, we ask:
    Is it possible to study the harmonic-analytic
    potential theory of several Jacob processes
    by somehow extending the methods of the present
    paper?
\end{problem}

\begin{problem}
	We ask:
    Is there a ``useful'' theory of excessive functions and/or measures for
    additive L\'evy processes (or more general multiparameter
    Markov processes)? This question is intimately connected to Open Problem
    \ref{pbm1}, but deserves to be asked on its own.
    In the one-parameter
    case, the answer is a decisive ``yes'' \cite{Getoor:book}.
    But the one-dimensional theory does not appear to readily have a
    suitable extension to the multiparameter setting.
\end{problem}

\begin{problem}\label{pbm:4}
    We conjecture that, under the conditions of Theorem \ref{th:dimH},
    the following holds almost surely on $\big\{\cap_{k=1}^N X_k(\R_+)
    \ne \varnothing\big\}$:
    \begin{equation}\label{co:dimH:3}
        \dimh \bigcap_{k=1}^N X_k(\R_+)=
        \sup\left\{s\in(0\,,d):\ \int_{(-1,1)^d}
        \prod_{j=1}^N \left( \frac{u_j(z)+u_j(-z)}{2}\right)
        \, \frac{dz}{\|z\|^s}<\infty
        \right\}.
   \end{equation}
   [The difference between this and
   \eqref{co:dimH:2} is in the range of the integrals.]
    But in all but one case we have
    no proof; see Remark \ref{rem:allude} below for the mentioned
    case. As we shall see in that remark, what we actually
    prove is the following harmonic-analytic fact:
    {\it Suppose $u$ is the one-potential density of a L\'evy process,
    $u(0)>0$, $u$ is continuous on $\R^d$, and $u$ is
    finite on $\R^d\setminus\{0\}$.
    Then the local square-integrability
    of $u$ implies the [global] square-integrability of $u$.}
    We believe that the following more general result holds:
    If $u_1,\ldots,u_N$ are one-potential
    densities that share the stated properties for $u$, then
    \begin{equation}
        \prod_{j=1}^N \left(\frac{u_j(\bullet)+
        u_j(-\bullet)}{2}\right)
        \in L^1_{\text{\it loc}}(\R^d)\quad
        \Rightarrow\quad
        \prod_{j=1}^N \left(\frac{u_j(\bullet)+
        u_j(-\bullet)}{2}\right)\in L^1(\R^d).
    \end{equation}
    If this is so, then the results of this paper
    imply Conjecture \eqref{co:dimH:3}.
\end{problem}

\section{\bf The stationary additive L\'evy random field}

Consider a classical L\'evy process
$X:=\{X(t)\}_{t\ge 0}$ on $\R^d$ with characteristic exponent $\Psi$.
Let us introduce an independent copy $X'$ of $X$, and extend the
definition of $X$ to a process indexed by $\R$ as follows:
\begin{equation}\label{eq:X:tilde}
    \widetilde{X}(t) :=
    \begin{cases}
        X(t)&\text{if $t\ge 0$},\\
        -X'(-t)&\text{if $t<0$}.
    \end{cases}
\end{equation}
This is the two-sided L\'evy process with exponent $\Psi$ in the
sense that $\widetilde{X} := \{\widetilde{X}(t)\}_{t \in \R}$ has
stationary and independent increments. Moreover,
$\{\widetilde{X}(t+s)-\widetilde{X}(s)\}_{t\ge 0}$ is a copy of $X$
for all $s\in\R$.

We also define $\widetilde{\X}:= \{\widetilde{\X}(\bm{t})\}_{\bm{t}\in\R^N}$
as the corresponding $N$-parameter process, indexed by all of
$\R^N$, whose values are in $\R^d$ and are defined as
\begin{equation}
    \widetilde{\X}(\bm{t}) := \widetilde{X}_1(t_1)+
    \cdots+\widetilde{X}_N(t_N)
    \qquad\text{for all $\bm{t}:=(t_1\,,\ldots,t_N)\in\R^N$.}
\end{equation}
We are assuming, of course, that $\widetilde{X}_1,
\ldots,\widetilde{X}_N$ are independent
two-sided extensions of the processes
$X_1,\ldots,X_N$, respectively.

We intend to prove the following two-sided version of
Theorem \ref{th:main}.

\begin{theorem}\label{th:main:tilde}
    For all Borel sets $F\subseteq\R^d$,
    \begin{equation}
        \E\left[ \lambda_d\left(
        \widetilde{\X}(\R^N)\oplus F\right) \right]>0\quad
        \text{if and only if}\quad
        \mathrm{cap}_{\bm{\Psi}}(F)>0.
    \end{equation}
\end{theorem}

This implies Theorem \ref{th:main} effortlessly. Indeed,
we know already from Remark 1.2 of \fullocite{KXZ:03} that
\begin{equation}
    \mathrm{cap}_{\bm{\Psi}}(F)>0\quad
    \Longrightarrow\quad \E\left[
    \lambda_d \left( \X(\R^N_+)\oplus F\right) \right]>0.
\end{equation}
Thus, we seek only to derive the converse implication. But
that follows from Theorem \ref{th:main:tilde},
because $\X(\R^N_+)\subseteq\widetilde{\X}(\R^N)$.

Henceforth, we assume that the underlying probability space
is the collection of all paths $\omega:\R^N\to\R^d$ that
have the form
$\omega(\bm{t}) = \sum_{j=1}^N \omega_j(t_j)$
for all $\bm{t}\in\R^N$,
where each $\omega_j$ maps $\R^N$ to $\R^d$ such that
$\omega_j(\bm{0})=0$; and $\omega_j\in D_{\R^d}(\R)$,
the Skorohod space of cadlag functions from
$\R$---not $[0\,,\infty)$---to $\R^d$.

We can then assume that
the stationary additive L\'evy fields, described
earlier in this section, are in \emph{canonical form}.
That is, $\widetilde{\X}(\bm{t})(\omega) :=
\omega(\bm{t})$ for all
$t\in\R^N$ and $\omega\in\Omega$.
Because we are interested only in distributional results,
this is a harmless assumption.

Define $\P_x$ to be the law of $x+\widetilde{\X}$,
and $\E_x$ the
expectation operation with respect to $\P_x$, for every
$x\in\R^d$. Thus, we are identify
$\P$ with $\P_0$, and $\E$ with $\E_0$.

We are interested primarily in the $\sigma$-finite measure
\begin{equation}
    \P_{\lambda_d} := \int_{\R^d} \P_x\, dx,
\end{equation}
and the corresponding expectation operator $\E_{\lambda_d}$,
defined by
\begin{equation}
    \E_{\lambda_d} f := \int_\Omega f(\omega)\, \P_{\lambda_d}(d\omega)
    \qquad\text{for all $f\in L^1(\P_{\lambda_d})$}.
\end{equation}

If  $A\ominus B:=\{ a-b:\, a\in A,\, b\in B\}$, then by the
Fubini-Tonelli theorem,
\begin{equation}\label{eq:fubini}\begin{split}
    \E\left[\lambda_d\left( \widetilde{\X}(\R^N)\ominus F\right)\right]
        &= \E\left[ \int_{\R^d} \mathbf{1}_{\widetilde{\X}(\R^N)\ominus F}
        (x)\,dx \right]\\
    &=\int_{\R^d}\P_{-x}\left\{\widetilde{\X}(\bm{t})
        \in F\text{ for some $\bm{t}\in\R^N$}\right\}\,dx\\
    &=\P_{\lambda_d}\left\{ \widetilde{\X}(\R^N)\cap F\neq\varnothing
        \right\}.
\end{split}\end{equation}
Thus, Theorem \ref{th:main:tilde} is a potential-theoretic
characterization of all polar sets for $\widetilde{\X}$
under the $\sigma$-finite measure $\P_{\lambda_d}$.
With this viewpoint in mind, we proceed to introduce
some of the fundamental objects that are related
to the process $\widetilde{\X}$.

Define, for all $\bm{t}\in\R^N$, the linear operator
$P_{\bm t}$ as follows:
\begin{equation}
    (P_{\bm t} f)(x) := \E_x \left[ f\left( \widetilde{\X}(\bm{t})
    \right) \right]\qquad\text{for all $x\in\R^d$}.
\end{equation}
This is well-defined, for example, if $f:\R^d\to\R_+$ is Borel-measurable,
or if $f:\R^d\to\R$ is Borel-measurable and $P_{\bm t}(|f|)$ is finite
at $x$. Also define the linear operator $R$ by
\begin{equation}\label{Def:Rf}
    (Rf)(x) := \frac{1}{2^N} \int_{\R^N}
    (P_{\bm t}f)(x)e^{-[\bm{t}]} \, d\bm{t}
    \qquad\text{for all $x\in\R^d$},
\end{equation}
where
\begin{equation}\label{eq:[t]}
    [\bm{t}]:=|t_1|+\cdots+|t_N|
\end{equation}
denotes the
$\ell^1$-norm of $\bm{t}\in\R^N$. [We will use this notation
throughout.] The $\ell^2$-norm of $\bm{t}\in\R^N$ will be denoted by
$\|\bm{t}\|$.

Again, $(Rf)(x)$ is well
defined if $f:\R^d\to\R_+$ is Borel-measurable, or if
$f:\R^d\to\R$ is Borel-measurable and $R(|f|)$ is finite at $x$.

Our next result is a basic regularity lemma for $R$. It should
be recognized as a multiparameter version of a very well-known
property of Markovian semigroups and their resolvents.

\begin{lemma}\label{lem:contraction}
    Each $P_{\bm t}$ and $R$ are contractions
    on $L^p(\R^d)$, as long as $1\le p\le\infty$.
\end{lemma}

\begin{proof}
    Choose and fix $j\in\{1\,,\ldots,N\}$ and $t\in \R$,
    and define $\mu_{j,t}$ to be the distribution of
    the random variable $-\widetilde{X}_j(t)$. If
    $f:\R^d\to\R_+$ is Borel-measurable, then
    \begin{equation}\label{eq:P_t:mu}
        P_{\bm t}f = f*\mu_{1,t_1}*\cdots*
        \mu_{N,t_N},
    \end{equation}
    where $*$ denotes convolution. This implies readily
    that $P_{\bm t}$ is a contraction on $L^p(\R^d)$ for all $p\in[1\,,\infty]$,
    and hence,
    \begin{equation}
        \|Rf\|_{L^p(\R^d)} \le \frac{1}{2^N}\int_{\R^N}
        \|P_{\bm t}f\|_{L^p(\R^d)} e^{-[\bm{t}]} \, d\bm{t}.
    \end{equation}
    Since $P_{\bm{t}}$ is a contraction on $L^p(\R^d)$,
    the preceding is bounded above by $\|f\|_{L^p(\R^d)}$.
\end{proof}
Henceforth, let ``$\widehat{\hskip3mm}$'' denote the [Schwartz]
Fourier transform on any and every Euclidean space $\R^k$. Our
Fourier transform is normalized such that
\begin{equation}
    \hat{f}(\xi) := \int_{\R^k} e^{i\xi\cdot x}
    f(x)\,dx,
\end{equation}
for all $f\in L^1(\R^k)$, say.
We then have the following.

\begin{lemma}\label{lem:P:U}
    If $f\in L^1(\R^d)$, then $\hat{R}=K_{\bm{\Psi}}$.
    That is,
    $\widehat{(Rf)}(\xi) = K_{\bm{\Psi}}(\xi)\hat{f}(\xi)$
    for all $\xi\in\R^d$.
\end{lemma}

\begin{proof}
    Recall $\mu_{j,t}$ from the proof of Lemma \ref{lem:contraction}. Its Fourier
    transform is given by
    \begin{equation}\label{eq:mu:hat}
        \widehat{\mu_{j,t}}(\xi) = \exp\left\{ -|t|\Psi_j\left(-
        \text{\rm sgn}(t) \xi \right) \right\}
        \qquad\text{for all $\xi\in\R^d$}.
    \end{equation}
    Equations \eqref{eq:P_t:mu} and
    \eqref{eq:mu:hat}, and the Plancherel theorem together
    imply that
    \begin{equation}\label{eq:P_tf:Fourier}
        \widehat{(P_{\bm t}f)}(\xi) =
        {\hat{f}(\xi)}
        \exp\left(-\sum_{j=1}^N |t_j| \Psi_j \left( -
        \text{\rm sgn}(t_j)\xi \right) \right).
    \end{equation}
    Consequently,
    \begin{equation}\begin{split}
        \widehat{(Rf)}(\xi) &= \frac{1}{2^N}\hat{f}(\xi)
            \int_{\R^N}
            \exp\left(-\sum_{j=1}^N |t_j| \big[ 1+\Psi_j \left( -
            \text{\rm sgn}(t_j)\xi \right) \big] \right)\, d\bm{t}\\
        &=\frac{1}{2^N}\hat{f}(\xi) \prod_{j=1}^N\int_0^\infty
            \left( e^{-t[1+\Psi_j(\xi)]} + e^{-t[1+\Psi_j(-\xi)]}\right)\, dt.
    \end{split}\end{equation}
    A direct computation reveals that $\widehat{Rf}=K_{\bm{\Psi}}\hat{f}$,
    as asserted.
\end{proof}

The following is a functional-analytic consequence.

\begin{corollary}\label{cor:R:adjoint}
    The operator $R$ maps $L^2(\R^d)$ into $L^2(\R^d)$,
    and is self-adjoint.
\end{corollary}

\begin{proof}
    By Lemma \ref{lem:P:U}, if $f\in L^1(\R^d)\cap L^2(\R^d)$
    and $g\in L^2(\R^d)$, then
    \begin{equation}
        \int_{\R^d} (Rf)(x) g(x)\,dx =\frac{1}{(2\pi)^d}
        \int_{\R^d} K_{\bm{\Psi}}(\xi)\, \hat{f}(\xi)\,
        \overline{\hat{g}(\xi)}\,d\xi.
    \end{equation}
    Thanks to Lemma \ref{lem:contraction}, the preceding
    holds for all $f\in L^2(\R^d)$. Duality then implies
    that $R: L^2(\R^d)\to L^2(\R^d)$. Moreover, since
    $K_{\bm{\Psi}}$ is real, $R$ is self-adjoint.
\end{proof}

The following lemma shows that for every ${\bm t} \in \R^N$, the
distribution of $\widetilde{\X}(\bm{t})$ under $\P_{\lambda}$ is
$\lambda_d$. This is the reason why we call $\widetilde{\X}$ a
\emph{stationary} additive L\'evy process.
\begin{lemma}\label{lem:Ef}
    If $f:\R^d\to\R_+$ is Borel-measurable, then
    \begin{equation}
        \E_{\lambda_d} \left[ f\left( \widetilde{\X}(\bm{t})\right)\right]
        =\int_{\R^d} ( P_{\bm{t}}f)(x)\, dx
        = \int_{\R^d} f(y)\, dy\qquad\text{for all $\bm{t}\in\R^N$}.
    \end{equation}
\end{lemma}

\begin{proof}
    We apply the Fubini-Tonelli theorem to find that
    \begin{equation}
        \E_{\lambda_d} \left[ f\left( \widetilde{\X}(\bm{t})\right)\right]
        = \E \int_{\R^d} f\left( x+\widetilde{\X}(\bm{t})\right)\, dx.
    \end{equation}
    A change of variables [$y:=x+\widetilde{\X}(\bm{t})$]
    proves that the preceding expression is equal to the integral
    of $f$. This implies half of the lemma.
    Another application of the Fubini--Tonelli theorem
    implies the remaining half as well.
\end{proof}

Let us choose and fix $\pi\subseteq\{1\,,\ldots,N\}$
and identify $\pi$ with the partial order
$\prec_\pi$, on $\R^N$, which is defined as follows:
For all $\bm{s},\bm{t}\in\R^N$,
\begin{equation}
    \bm{s}\prec_\pi\bm{t}\quad\text{iff}\quad \begin{cases}
        s_i\le t_i&\text{for all $i\in\pi$, and}\\
        s_i>t_i&\text{for all $i\not\in\pi$}.
    \end{cases}
\end{equation}
The collection of all $\pi\subseteq\{1\,,\ldots,N\}$ forms
a collective total order on $\R^N$ in the sense that
\begin{equation}\label{eq:total:order}
    \text{for all $\bm{s},\bm{t}\in\R^N$ there exists
    $\pi\subseteq\{1\,,\ldots,N\}$ such that
    $\bm{s}\prec_\pi\bm{t}$}.
\end{equation}

For all $\pi\subseteq\{1\,,\ldots,N\}$,
we define the $\pi$-\emph{history} of the random field $\widetilde{\X}$
as the collection
\begin{equation}
    \mathscr{H}_\pi(\bm{t}) := \sigma\left(\left\{
    \widetilde{\X}(\bm{s})\right\}_{\bm{s}\prec_\pi\bm{t}}\right)
    \qquad\text{for all $\bm{t}\in\R^N$},
\end{equation}
where $\sigma(\,\cdots)$ denotes the $\sigma$-algebra generated
by whatever is in the parentheses. Without loss of generality, we
assume that each $\mathscr{H}_\pi(\bm{t})$ is complete
with respect to $\P_x$ for all $x\in\R^d$; else, we replace
it with the said completion. Also, we assume without loss
of generality that $\bm{t}\mapsto\mathscr{H}_\pi(\bm{t})$
is $\pi$-right-continuous.
More precisely, we assume
that
\begin{equation}
    \mathscr{H}_\pi(\bm{t}) =
    \bigcap_{\bm{s}\in\R^N:\
    \bm{t} \prec_\pi\bm{s}} \mathscr{H}_\pi(\bm{s})
    \quad\text{for all $\bm{t}\in\R^N$ and $\pi\subseteq\{1\,,\ldots,N\}$}.
\end{equation}
If not, then we replace the left-hand side by
the right-hand side everywhere.

\begin{proposition}[A Markov-random-field property]\label{pr:MP}
    Suppose $\pi\subseteq\{1\,,\ldots,N\}$
    and $\bm{s}\prec_\pi\bm{t}$, both in $\R^N$.
    Then, for all measurable functions $f:\R^d\to\R_+$,
    \begin{equation}
        \E_{\lambda_d}\left[ \left.
        f\left(\widetilde{\X}(\bm{t})\right)\ \right|\,
        \mathscr{H}_\pi(\bm{s})\right]
        = \left( P_{\bm{t}-\bm{s}}f\right) \left( \widetilde{\X}(\bm{s})
        \right)\qquad \P_{\lambda_d}\text{-a.s.}
    \end{equation}
\end{proposition}

\begin{proof}
    Choose and fix Borel measurable
    functions $g,\phi_1,\ldots,\phi_m:\R^d\to\R_+$,
    and ``$N$-parameter time
    points'' $\bm{s}^1,\ldots,\bm{s}^m\in\R^N$ such that
    \begin{equation}
        \bm{s}^j\prec_\pi \bm{s}\prec_\pi\bm{t}\qquad
        \text{for all  $j=1,\ldots,m$.}
    \end{equation}

    According to the Fubini-Tonelli theorem,
    \begin{equation}\begin{split}
        &\E_{\lambda_d} \left[ f\left(\widetilde{\X}(\bm{t})\right)
            g\left(\widetilde{\X}(\bm{s})\right) \prod_{j=1}^m
            \phi_j\left(\widetilde{\X}(\bm{s}^j)\right) \right]\\
        &\hskip1.5in = \int_{\R^d} \E
            \left[ f\left(x+\widetilde{\X}(\bm{t})\right)
            g\left(x+\widetilde{\X}(\bm{s})\right) \prod_{j=1}^m
            \phi_j\left(x+\widetilde{\X}(\bm{s}^j)\right) \right]\, dx\\
        &\hskip1.5in =\int_{\R^d} \E\left[ f(A+y) \prod_{j=1}^m\phi_j(A_j+y)
            \right] g(y)\, dy,
    \end{split}\end{equation}
    where
    $A := \widetilde{\X}(\bm{t}) - \widetilde{\X}(\bm{s})$
    and $A_j := \widetilde{\X}(\bm{s}^j) - \widetilde{\X}(\bm{s})$
    for all $j=1,\ldots,m$.

    The independent-increments property of each
    of the L\'evy processes $\widetilde{X}_j$
    implies that $A$ is independent of $\{A_j\}_{j=1}^m$.
    This and the stationary-increments property
    of $\widetilde{X}_1,\ldots,\widetilde{X}_N$ together imply that
    \begin{equation}\begin{split}
        &\E_{\lambda_d} \left[ f\left(\widetilde{\X}(\bm{t})\right)
            g\left(\widetilde{\X}(\bm{s})\right) \prod_{j=1}^m
            \phi_j\left(\widetilde{\X}(\bm{s}^j)\right) \right]\\
        &\hskip2in =\int_{\R^d} \E\left[ f(A+y) \right]
            \E\left[ \prod_{j=1}^m\phi_j(A_j+y)
            \right] g(y)\, dy.
    \end{split}\end{equation}
    After a change of variables and an appeal to the stationary-independent
    property of the increments of $X_1,\ldots,X_N$ and
    $X_1',\ldots,X_N'$, we arrive at the following:
    \begin{equation}\begin{split}
        &\E_{\lambda_d} \left[ f\left(\widetilde{\X}(\bm{t})\right)
            g\left(\widetilde{\X}(\bm{s})\right) \prod_{j=1}^m
            \phi_j\left(\widetilde{\X}(\bm{s}^j)\right) \right]\\
        &\hskip2in =\int_{\R^d} \E\left[ f\left(
            \widetilde{\X}(\bm{t}-\bm{s})+y
            \right) \right]
            \E_y\left[ \prod_{j=1}^m\phi_j(A_j)
            \right] g(y)\, dy\\
        &\hskip2in =\int_{\R^d} (P_{\bm{t}-\bm{s}} f)(y)
            \E_y\left[ \prod_{j=1}^m\phi_j(A_j)
            \right] g(y)\, dy.
    \end{split}\end{equation}
    By the monotone class theorem, for all nonnegative
    $\mathscr{H}_\pi(\bm{s})$-measurable random variables
    $Y$,
    \begin{equation}
        \E_{\lambda_d} \left[ f\left(\widetilde{\X}(\bm{t})\right)
        g\left(\widetilde{\X}(\bm{s})\right) Y \right]
        =\int_{\R^d} (P_{\bm{t}-\bm{s}} f)(y)
        \, \psi(y)\, g(y)\, dy,
    \end{equation}
    where $\psi:\R^d\to\R_+$ is a measurable function. This proves
    the proposition.
\end{proof}

\begin{lemma}\label{lem:Cov}
    If $f,g\in L^2(\R^d)$
    and $\bm{t},\bm{s}\in\R^N$, then
    \begin{equation}
        \E_{\lambda_d} \left[ f\left(\widetilde{\X}(\bm{s})\right)
        g\left(\widetilde{\X}(\bm{t})\right)\right]
        =\int_{\R^d} f(y) (P_{\bm{t-s}}g)(y)\, dy.
    \end{equation}
\end{lemma}

\begin{proof}
    We may consider, without loss of generality, measurable
    and nonnegative functions $f,g\in L^2(\R^d)$. Let
    $\pi$ denote the collection of all $i\in\{1\,,\ldots,N\}$
    such that $s_i\le t_i$. Then $\bm{s}\prec_\pi\bm{t}$, and
    Proposition \ref{pr:MP} implies that $\P_{\lambda_d}$-a.s.,
    \begin{equation}
        \E_{\lambda_d}\left[ \left. f\left(\widetilde{\X}(\bm{t})\right)
        \ \right|\, \mathscr{H}_\pi(\bm{s}) \right]
        =\left( P_{\bm{t}-\bm{s}} f\right)\left(
        \widetilde{\X}(\bm{s}) \right).
    \end{equation}
    This and Lemma \ref{lem:Ef} together conclude the proof.
\end{proof}

\section{\bf The sojourn operator}

Recall \eqref{eq:[t]},
and consider the ``sojourn operator,''
\begin{equation}\label{eq:Sf}
    Sf := \frac{1}{2^N} \int_{\R^N} f\left(
    \widetilde{\X}(\bm{t}) \right)  e^{-[\bm{t}]}\, d\bm{t}.
\end{equation}

Our first lemma records the fact that $S$ maps density
functions to mean-one random variables $[\P_{\lambda_d}]$.
\begin{lemma}\label{lem:ESf}
    If $f$ is a probability density function on $\R^d$,
    then $\E_{\lambda_d} [Sf] =1$.
\end{lemma}
This follows readily from Lemma \ref{lem:Ef}. Our next result shows
that, under a mild condition on $\Psi_j$, $S$ embeds  functions in
$L^2(\R^d)$, quasi-isometrically, into the subcollection of all
functions in $L^2(\R^d)$ that have finite energy. Namely,

\begin{proposition}\label{pr:EA_f2}
    If $f\in L^2(\R^d)$, then
    \begin{equation} \label{Eq:SoB1}
        \left\| Sf\right\|_{L^2(\P_{\lambda_d})}
        \le \sqrt{I_{\bm{\Psi}}(f)}.
    \end{equation}
    Suppose, in addition, that there exists a constant
    $c\in(0\,,\sqrt 2)$, such that the following sector
    condition holds for all $j=1,\ldots, N$:
    \begin{equation}\label{eq:sector}
        \left| \Im \Psi_j(\xi)\right| \le c\left( 1+
        \Re \Psi_j(\xi)\right) \qquad\text{for all
        $\xi\in\R^d$}.
    \end{equation}
    Then, there exists a constant $A\in(0\,,1)$ such that
    \begin{equation} \label{Eq:SoB2}
        A\, \sqrt{I_{\bm{\Psi}}(f)}\le
        \left\| Sf\right\|_{L^2(\P_{\lambda_d})}
        \le \sqrt{I_{\bm{\Psi}}(f)}.
    \end{equation}
\end{proposition}

\begin{remark}[Generalized Sobolev spaces]
    When $N=1$ and the L\'evy process
    in question is symmetric, the following
    problem arises in the theory
    of Dirichlet forms: For what $f$ in the class $\mathscr{D}(\R^d)$,
    of Schwartz distributions on $\R^d$, can we define $Sf$
    as an element of $L^2(\P_{\lambda_d})$ (say)? This
    problem continues to make sense in the more general context
    of additive L\'evy processes. And the answer
    is given by (\ref{Eq:SoB2}) in Proposition \ref{pr:EA_f2} as follows:
    Assume that the \emph{sector condition} \eqref{eq:sector} holds
    for all $j=1,\ldots,N$ and some $c\in(0\,,\sqrt 2)$.
    Let $\mathscr{S}_{\bm{\Psi}}(\R^d)$ denote the completion
    of the collection of all members of
    $L^2(\R^d)$ that have finite energy $I_{\bm{\Psi}}$,
    where the completion is made in the norm $\|f\|_{\bm{\Psi}}:=
    I_{\bm{\Psi}}^{1/2}(f)+\|f\|_{L^2(\R^d)}$. Then,
    there exists an a.s.-unique maximal extension $\bar{S}$ of $S$ such that
    $\bar{S}:\mathscr{S}_{\bm{\Psi}}(\R^d)\to\bar{S}
    (\mathscr{S}_{\bm{\Psi}}(\R^d))$ is a quasi-isometry. The space
    $\mathscr{S}_{\bm{\Psi}}(\R^d)$ generalizes further some of
    the $\psi$-Bessel potential spaces of
    \fullocite{farkasjacobschilling} and
    \ocite{farkasleopold}; see also
    \ocite{jacobschilling}, \ocite{masjanagel},
    and \ocite{slobodecki}.\qed
\end{remark}

The proof requires a technical lemma, which
we develop first.

\begin{lemma}\label{lem:technical}
    For all $z\in\mathbf{C}$ define
    \begin{equation}
        \Lambda (z) :=
        \int_{-\infty}^\infty \int_{-\infty}^\infty
        e^{-|t|-|s|-|t-s|\sigma(z;\, t-s)}\, dt\, ds,
    \end{equation}
    where $\sigma(z;\, r):=z$ if $r\ge 0$ and
    $\sigma(z;\, r):=\bar z$ otherwise.
    Then for all $z\in\mathbf{C}$ with $\Re z \ge 0$,
    \begin{equation} \label{Eq:Lambda1}
        \Lambda (z) \le 4\Re\left(\frac{1}{1+z}\right).
    \end{equation}
    If, in addition, $|\Im z| \le c(1+\Re z)$ for some
    $c\in (0\,,\sqrt{2} )$, then
    \begin{equation} \label{Eq:Lambda2}
      \Lambda (z)\ge  2\left( 2-c^2\right)
        \Re\left(\frac{1}{1+z}\right).
    \end{equation}
\end{lemma}

\begin{proof}
    The double integral is computed by
    dividing the region of integration into four
    natural parts: (i) $s,t\ge 0$; (ii) $s,t\le 0$;
    (iii) $t\ge 0\ge s$; and (iv) $s\ge 0\ge t$. Direct
    computation reveals that for all $z\in\mathbf{C}$ with $\Re z \ge 0$
    \begin{equation}\begin{split}
        &\int_0^\infty\int_0^\infty
            e^{-|t|-|s|-|t-s|\sigma(z;\, t-s)}\, dt\, ds
            + \int_{-\infty}^0 \int_{-\infty}^0
            e^{-|t|-|s|-|t-s|\sigma(z;\, t-s)}\, dt\, ds\\
        &\hskip4.4in =
            2\Re\left(\frac{1}{1+z}\right).
    \end{split}\end{equation}
    Similarly, one can compute
    \begin{equation}\begin{split}
        &\int_0^\infty\int_{-\infty}^0
            e^{-|t|-|s|-|t-s|\sigma(z;\, t-s)}\, dt\, ds
            + \int_{-\infty}^0 \int_0^\infty
            e^{-|t|-|s|-|t-s|\sigma(z;\, t-s)}\, dt\, ds\\
        &\hskip4in = \frac 1 {(1 + z)^2} +  \frac 1 {(1 + \bar z)^2}.
    \end{split}\end{equation}
    Consequently,
    \begin{equation}\label{Eq:Lambda3}
        \Lambda (z) = 2\Re\left(\frac{1}{1+z}\right) +  \frac{2
        \big((1+\Re z)^2 - (\Im z)^2\big)} {|1 + z|^4},
    \end{equation}
    for all $z\in\mathbf{C}$ with $\Re z \ge 0$. It follows that
    \begin{equation}\label{Eq:Lambda4}
        \Lambda (z)\le  2\Re\left(\frac{1}{1+z}\right)  \left[
        1+\Re\left(\frac{1}{1+z}\right)\right]
    \end{equation}
    for all $z\in\mathbf{C}$ with $\Re z \ge 0$. Whenever $\Re z \ge 0$,
    we have $0\le \Re(1+z)^{-1}\le 1$, and hence (\ref{Eq:Lambda1})
    follows from \eqref{Eq:Lambda4}. On the other hand, if $|\Im z | \le
    c(1+\Re z)$, then \eqref{Eq:Lambda3} yields
    \begin{equation}
        \Lambda (z) \ge 2\Re\left(\frac{1}{1+z}\right)
        + 2\left(1-c^2\right) \left[ \Re
        \left( \frac{1}{1+z}\right)\right]^2,
    \end{equation}
    from which the result follows readily, because
    $0\le \Re[(1+z)^{-1}]\le 1$ when $\Re z\ge 0$.
\end{proof}

\begin{proof}[Proof of Proposition \ref{pr:EA_f2}]
    We apply Lemma \ref{lem:Cov} to deduce that
    \begin{equation}
        \E_{\lambda_d}
        \left( \left| Sf\right|^2\right)
        = \frac{1}{4^N}\int_{\R^d} \iint_{\R^N\times\R^N}
        e^{-[\bm{t}]-[\bm{s}]} f(y) (P_{\bm{t}-\bm{s}}f)(y)\, d\bm{t}
        \, d\bm{s}\, dy.
    \end{equation}
    In accord with \eqref{eq:P_tf:Fourier}
    and Parseval's identity,
    for all $\bm{u}\in\R^N$,
    \begin{equation}\begin{split}
        \int_{\R^d} f(y) (P_{\bm{u}}f)(y)\, dy
            &=\frac{1}{(2\pi)^d}\int_{\R^d}
            \overline{\hat{f}(\xi)}\,
            \widehat{(P_{\bm u}f)}(\xi)\, d\xi\\
        &= \frac{1}{(2\pi)^d}\int_{\R^d} \left| \hat{f}(\xi) \right|^2
            \exp\left(-\sum_{j=1}^N |u_j| \Psi_j \left( -
            \text{\rm sgn}(u_j)\xi \right) \right)\, d\xi.
    \end{split}\end{equation}
    This and the Fubini-Tonelli theorem together reveal that
    \begin{equation}
        \E_{\lambda_d}
        \left(\left| Sf\right|^2\right) = \frac{1}{4^N(2\pi)^d}
        \int_{\R^d} \left| \hat{f}(\xi)\right|^2    \prod_{j=1}^N
        \Lambda (\Psi_j(\xi)) \,d\xi.
    \end{equation}
Since $\Re \Psi_j(\xi) \ge 0$ for all $\xi \in \R^d$ and $j = 1,
\ldots, N$, we apply Lemma \ref{lem:technical} to this formula, and
conclude the proof of the proposition.
\end{proof}

\section{\bf Proof of Theorem \ref{th:main:tilde}}

Thanks to the definition of $\mathrm{cap}_{\bm{\Psi}}$, and to the
countable additivity of $\P$, it suffices to consider only the case
that
\begin{equation}
    \text{$F$ is a compact set}.
\end{equation}
This condition is tacitly assumed throughout this section.
We note, in particular, that $\mathcal{P}_c(F)$ denotes merely the
collection of all Borel probability measures that are supported
on $F$.

Proposition 5.7 of \fullocite{KXZ:03} proves that
for every compact set $F \subseteq \R^d$,
\begin{equation}\label{asserted:LB}
    \mathrm{cap}_{\bm{\Psi}}(F) > 0 \ \Longrightarrow\
    \E\left[\lambda_d\left( \X([0, r]^N) \oplus F\right)\right]
    > 0
\end{equation}
for all $r > 0$. It is clear the latter implies that $
\E\big[\lambda_d\big(\widetilde{\X}(\R^N) \oplus F\big)\big]
> 0$, and one obtains half of the theorem.

Since $\mathrm{cap}_{\bm{\Psi}}(-F)  = \mathrm{cap}_{\bm{\Psi}}(F)$,
we may and will replace $F$ by $-F$ throughout. In light of
\eqref{eq:fubini}, we then assume that
\begin{equation}\label{assume}
    \P_{\lambda_d}\left\{ \widetilde{\X}(\R^N)\cap
    F\neq\varnothing\right\}>0,
\end{equation}
and seek to deduce the existence of a probability measure $\mu$ on
$F$ such that $I_{\bm{\Psi}}(\mu)<\infty$.

We will prove a little more. Namely, that for all $k>0$
there exists a constant $A=A(k\,,N)\in(0\,,\infty)$ such that
\begin{equation}
    \P_{\lambda_d}\left\{ \widetilde{\X}\left(
    [-k\,,k]^N\right) \cap F \neq \varnothing\right\} \le
    A\, \mathrm{cap}_{\bm{\Psi}}(F).
\end{equation}
In fact,
we will prove that for all
sufficiently large $k>0$,
\begin{equation}\label{asserted:UB}
    \P_{\lambda_d}\left\{ \widetilde{\X}\left(
    [-k\,,k]^N\right) \cap F \neq \varnothing\right\} \le
    e^{2Nk} 4^N\, \mathrm{cap}_{\bm{\Psi}}(F).
\end{equation}
This would conclude our proof of the second half of the theorem.

It is a standard fact that there exists a probability density
function $\phi_1$ in $C^\infty(\R^d)$ with the following properties:
\begin{enumerate}
    \item[\bf P1.] $\phi_1(x)=0$ if $\|x\|>1$;
    \item[\bf P2.] $\phi_1(x)=\phi_1(-x)$ for all $x\in\R^d$;
    \item[\bf P3.] $\hat{\phi}_1(x)\ge 0$ for all $x\in\R^d$.
\end{enumerate}
This can be obtained, for example, readily from Plancherel's
[duality] theorem of Fourier analysis.

We recall also the following standard fact: $\phi_1 \in
L^1(\R^d)$ and {\bf P3} together imply that $\hat{\phi}_1 \in L^1(\R^d)$
\cite{Hawkes:84}*{Lemma 1}.

Now we define an approximation to the identity $\{\phi_\e\}_{\e>0}$
by setting
\begin{equation} \label{Def:phie}
    \phi_\e(x) := \frac{1}{\e^d}\phi_1\left( \frac{x}{\e}\right)\qquad
    \text{for all $x\in\R^d$ and $\e>0$}.
\end{equation}
It follows readily from this that for every $\mu \in \mathcal{P}_c(F)$:
\begin{enumerate}
    \item[(i)] $\mu*\phi_\e$ is a uniformly continuous probability density;
    \item[(ii)] $\mu*\phi_\e$ is supported on the closed
        $\e$-enlargement of $F$, which we denote by $F_\e$, for all $\e>0$;
    \item[(iii)] $\phi_\e$ is symmetric; and
    \item[(iv)] $\lim_{\e\to 0^+}\hat{\phi_\e}(\xi)=1$
        for all $\xi\in\R^d$.
\end{enumerate}

As was done in \fullocite{KXZ:03}, we can find a random
variable $\bm{T}$ with values in $\R^N\cup\{\infty\}$ such that:
\begin{enumerate}
    \item $\{\bm{T}=\infty\}$ is
        equal to the event that $\widetilde{\X}(\bm{t})
        \not\in F$ for all $\bm{t}\in\R^N$;
    \item $\widetilde{\X}(\bm{T})\in F$ on $\{\bm{T}\neq\infty\}$.
\end{enumerate}
This can be accomplished pathwise.
Consequently, \eqref{assume} is equivalent to the condition that
\begin{equation}\label{assume2}
    \P_{\lambda_d}\left\{\bm{T}\neq\infty\right\}>0.
\end{equation}

Define for all Borel sets $A\subseteq\R^d$
and all integers $k\ge 1$,
\begin{equation}
    \mu_k (A) := \P_{\lambda_d}\left( \widetilde{\X}(\bm{T})\in A\
    \Big|\ \bm{T}\in[-k\,,k]^N
    \right).
\end{equation}
We claim that $\mu_k$ is a probability measure on $F$
for all $k$ sufficiently large.
In order to prove this claim we choose and fix
$l>0$, and consider
\begin{equation}
    \mu_{k,l}(A) := \frac{\P_{\lambda_d}\left\{
    \widetilde{\X}(\bm{T})\in A ~,~ \bm{T}\in[-k\,,k]^N ~,~
    \left| \widetilde{\X}(\bm{0})\right| \le l \right\}}{\P_{\lambda_d}
    \left\{
    \bm{T}\in[-k\,,k]^N~,~\left| \widetilde{\X}(\bm{0})\right| \le l\right\}}.
\end{equation}
Because the $\P_{\lambda_d}$-distribution of
$\widetilde{\X}(\bm{0})$ is $\lambda_d$, $\mu_{k,l}$ is a probability
measure on $F$ for all $k$ and $l$ sufficiently large; see \eqref{assume2}.
And $\mu_{k,l}(A)$ converges to $\mu_k(A)$ for all
Borel sets $A$ as $l\uparrow\infty$
[monotone convergence theorem]. This proves the
assertion that $\mu_k\in\mathcal{P}_c(F)$ for all $k$ large.
Choose and fix such a large integer $k$.

Now we define $f_\e:=\mu_k*\phi_\e$, and observe that according to
Proposition \ref{pr:MP}, for all nonrandom times $\bm{\tau}\in\R^N$,
\begin{equation}\begin{split}
    &\sum_{\pi\subseteq\{1,\dots,N\}}\E_{\lambda_d}\left[\left.
        \int_{\bm{t}\succ_\pi\bm{\tau}}
        f_\e\left(\widetilde{\X}(\bm{t})\right)
        e^{-[\bm{t}]}\, d\bm{t} \ \right|\,
        \mathscr{H}_\pi(\bm{\tau})\right]\\
    &\hskip2in=\sum_{\pi\subseteq\{1,\dots,N\}}
        \int_{\bm{t}\succ_\pi\bm{\tau}} \left( P_{\bm{t}-\bm{\tau}}f_\e
        \right)\left(\widetilde{\X}(\bm{\tau})\right) e^{-[\bm{t}]}\,
        d\bm{t}\\
    &\hskip2in\ge e^{-[\bm{\tau}]}\sum_{\pi\subseteq\{1,\ldots,N\}}
        \int_{\bm{s}\succ_\pi\bm{0}} \left( P_{\bm{s}}f_\e
        \right)\left(\widetilde{\X}(\bm{\tau})\right) e^{-[\bm{s}]}\,
        d\bm{s}.
\end{split}\end{equation}
We can apply \eqref{eq:total:order} to deduce then that
for all nonrandom times $\bm{\tau}\in\R^N$,
\begin{equation}\label{eq:Sfeps:LB1}\begin{split}
    \E_{\lambda_d}\left[\left.Sf_\e \ \right|\,
        \mathscr{H}_\pi(\bm{\tau})\right] &\ge \frac{e^{-[\bm{\tau}]}}{2^N}
        \int_{\R^N} \left( P_{\bm{s}}f_\e
        \right)\left(\widetilde{\X}(\bm{\tau})\right) e^{-[\bm{s}]}\,
        d\bm{s}\\
    &  = e^{-[\bm{\tau}]}\, (Rf_\e)\left(\widetilde{\X}(\bm{\tau})\right).
\end{split}\end{equation}
Because $f_\e$ is continuous and compactly supported, one can verify
from (\ref{Def:Rf}) that $R f_\e$ is continuous [this can also be
shown by Lemma \ref{lem:P:U}, the fact that $\hat{f_\e} \in
L^1(\R^d)$ and the Fourier inversion formula]. Also, the
$L^2(\P_{\lambda_d})$-norm of the left-most term in
\eqref{eq:Sfeps:LB1} is bounded above by
$\| Sf_\e\|_{L^2(\P_{\lambda_d})}$,
and this is at most $\sqrt{I_{\bm{\Psi}}(f_\e)}$, in turn;
see Proposition \ref{pr:EA_f2}. It is easy to see that
\begin{equation} \label{Eq:IPsi2}\begin{split}
    I_{\bm{\Psi}}(f_\e) &= I_{\bm{\Psi}}(\mu_k*\phi_\e)\\
    &\le \frac{1}{(2\pi)^{2d}}
        \int_{\R^d}|\hat{\phi_\e}(\xi)|^2\, d\xi\\
    &<\infty.
\end{split}\end{equation}
Since
$|\hat{\mu}_k(\xi)|^2K_{\bm{\Psi}}(\xi)\le 1$,
we conclude that $\|S f_\e\|_{L^2(\P_{\lambda_d})}$ is finite for
all $\e>0$.

A simple adaptation of the proof of Proposition \ref{pr:MP} proves
that $\mathscr{H}_\pi$ is a \emph{commuting filtration} for all
$\pi\subseteq\{1\,,\ldots,N\}$. By this we mean that for all
$\bm{t}\in\R$,
\begin{equation}
    \mathscr{H}^1_\pi(t_1),\ldots,\mathscr{H}^N_\pi(t_N)
    \text{ are conditionally independent $[\P_{\lambda_d}]$, given }
    \mathscr{H}_\pi(\bm{t}),
\end{equation}
where $\mathscr{H}^j_\pi(t_j)$ is defined as
the following $\sigma$-algebra:
\begin{equation}
    \mathscr{H}^j_\pi(t_j) := \begin{cases}
        \sigma\left( \widetilde{X}_j(s);\, -\infty< s\le t_j\right)
            &\text{if $j\in\pi$},\\
        \sigma\left( \widetilde{X}_j(s);\, \infty> s\ge t_j\right)
            &\text{if $j\not\in\pi$}.
    \end{cases}
\end{equation}
Here, $\sigma(\,\cdots)$ denotes the $\sigma$-algebra generated by
the random variables in the parenthesis.

The stated commutation property readily implies that for all
random variables
$Y\in L^2(\P_{\lambda_d})$ and
partial orders $\pi\subseteq\{1\,,\ldots,N\}$:
\begin{enumerate}
    \item $\bm{\tau}\mapsto \E_{\lambda_d}[Y\,|\,
        \mathscr{H}_\pi(\bm{\tau})]$ has a version that
        is cadlag in each of its $N$ variables, uniformly in all other
        $N-1$ variables; and
    \item The second moment of
        $\sup_{\bm{\tau}\in\R^N}\E_{\lambda_d}[Y\,|\,
        \mathscr{H}_\pi(\bm{\tau})]$ is at most $4^N$ times
        the second moment of $Y$ $[\P_{\lambda_d}]$.
\end{enumerate}
See \ocite{Kh:book}*{Theorem 2.3.2, p.\ 235}
for the case where $\P_{\lambda_d}$ is replaced
by a probability measure. The details of the
remaining changes are explained
in a slightly different setting in Lemma 4.2
of \fullocite{KXZ:03}. In summary,
\eqref{eq:Sfeps:LB1} holds for all $\bm{\tau}\in\R^N$,
$\P_{\lambda_d}$-almost surely [note the order of
the quantifiers]. It follows immediately from this that
for all integers $k\ge 1$,
\begin{equation}
    \sup_{\bm{\tau}\in\R^N}\E_{\lambda_d}\left[\left.Sf_\e \ \right|\,
    \mathscr{H}_\pi(\bm{\tau})\right] \ge
    e^{-Nk}\, (Rf_\e)\left( \widetilde{\X}(\bm{T})\right)
    \cdot \1_{\{\bm{T}\in[-k,k]^N\}}\qquad
    \P_{\lambda_d}\text{-a.s.}
\end{equation}
According to Item (2) above, the second moment of the left-hand side
is at most $4^N$ times the second moment of $Sf_\e$. As we noticed
earlier, the latter is at most $I_{\bm{\Psi}}(f_\e)$. Therefore,
\begin{equation}\begin{split}
    e^{2Nk}4^N I_{\bm{\Psi}}(f_\e) &\ge \E_{\lambda_d}
        \left[ \left|
        (Rf_\e)\left(\widetilde{\X}(\bm{T})\right)\right|^2
        ~;~ \bm{T}\in[-k\,,k]^N\right]\\
    &=\E_{\lambda_d}
        \left[ \left. \left|
        (Rf_\e)\left(\widetilde{\X}(\bm{T})\right)\right|^2
        \ \right|\ \bm{T}\in[-k\,,k]^N\right]\cdot
        \P_{\lambda_d}\left\{\bm{T}\in[-k\,,k]^N\right\}.
\end{split}\end{equation}
It follows from this and the Cauchy--Schwarz inequality that
\begin{equation}
    e^{2Nk}4^N I_{\bm{\Psi}}(f_\e)
    \ge\left|\E_{\lambda_d}
    \left[ \left.
    (Rf_\e)\left(\widetilde{\X}(\bm{T})\right)
    \ \right|\ \bm{T}\in[-k\,,k]^N\right]\right|^2\cdot
    \P_{\lambda_d}\left\{\bm{T}\in[-k\,,k]^N\right\}.
\end{equation}
Using the definition of
$\mu_k$, we can write the above as
\begin{equation}\label{Eq:428}\begin{split}
    e^{2Nk}4^N I_{\bm{\Psi}}(f_\e)
        &\ge\left| \int_{\R^d}
        (Rf_\e)\, d\mu_k\right|^2\cdot
        \P_{\lambda_d}\left\{\bm{T}\in[-k\,,k]^N\right\}\\
    &= \frac{1}{(2\pi)^{2d}}\left| \int_{\R^d} \overline{\hat{\mu}_k(\xi)}\
        \widehat{(Rf_\e)}(\xi)\, d\xi\right|^2 \cdot
        \P_{\lambda_d}\left\{\bm{T}\in[-k\,,k]^N\right\},
\end{split}\end{equation}
where the equality follows from the Parseval identity.
According to
Lemma \ref{lem:P:U}, the Fourier transform of $Rf_\e$ is
$K_{\bm{\Psi}}$ times the Fourier transform of $f_\e$, and the
latter is $\hat{\mu}_k \hat{\phi}_\e$. Hence (\ref{Eq:428}) implies
that
\begin{equation}\begin{split}
    e^{2Nk}4^N I_{\bm{\Psi}}(f_\e)
    \ge \frac{1}{(2\pi)^{2d}}\left| \int_{\R^d} \left|\hat{\mu}_k(\xi)\right|^2
        \hat{\phi}_\e(\xi)\ K_{\bm{\Psi}}(\xi)\, d\xi\right|^2 \cdot
        \P_{\lambda_d}\left\{\bm{T}\in[-k\,,k]^N\right\}.
\end{split}\end{equation}
Now we apply Property {\bf P3} to deduce that $\hat{\phi}_\e\ge 0$.
Because $\phi_\e$ is also a probability density, it follows that
$\hat{\phi}_\e\ge|\hat{\phi}_\e|^2$. Consequently,
\begin{equation}\begin{split}
    e^{2Nk} 4^N I_{\bm{\Psi}}(f_\e)
        &\ge \frac{1}{(2\pi)^{2d}}\left| \int_{\R^d} \left|\hat{f}_\e(\xi)\right|^2
        \ K_{\bm{\Psi}}(\xi)\, d\xi\right|^2 \cdot
        \P_{\lambda_d}\left\{\bm{T}\in[-k\,,k]^N\right\}\\
    &= \left| I_{\bm{\Psi}}(f_\e)\right|^2 \cdot
        \P_{\lambda_d}\left\{\bm{T}\in[-k\,,k]^N\right\}.
\end{split}\end{equation}
We have seen in (\ref{Eq:IPsi2}) that $I_{\bm{\Psi}}(f_\e)$ is
finite for each $\e>0$. If it were zero for arbitrary small $\e>0$,
then we apply Fatou's lemma to deduce that $I_{\bm{\Psi}}(\mu_k)
\le\liminf_{\e\to 0}I_{\bm{\Psi}}(f_\e)=0$. This and the fact that
$K_{\bm{\Psi}}(\xi) > 0$ for all $\xi \in \R^d$ would imply
$\hat{\mu}_k (\xi) \equiv 0$ for all $\xi \in \R^d$, which is a
contradiction. Hence we can deduce that for all $\e>0$ small enough,
\begin{equation}\label{eq:eee}
    \frac{e^{2Nk} 4^N}{I_{\bm{\Psi}}(f_\e)} \ge
    \P_{\lambda_d}\left\{\bm{T}\in[-k\,,k]^N\right\},
\end{equation}
and this is positive for $k$ large; see
\eqref{assume2}.
The right-hand side of \eqref{eq:eee} is independent of $\e>0$.
Hence, we can let $\e\downarrow 0$
and appeal to Fatou's lemma to deduce that $I_{\bm{\Psi}}(\mu_k)<\infty$.
Thus, in any event, we have produced a probability measure
$\mu_k$ on $F$ whose energy $I_{\bm{\Psi}}(\mu_k)$
is finite. This concludes the proof, and also implies
\eqref{asserted:UB}, thanks to the defining properties of
the function $\bm{T}$.\qed

\section{\bf On kernels of positive type}

In this section
we study kernels of positive type; they are recalled next.
Here and throughout, $\bar\R_+:=[0\,,\infty]$ is defined
to be the usual one-point
compactification of $\R_+:=[0\,,\infty)$, and is endowed with
the corresponding Borel sigma-algebra.

\begin{definition}
    A \emph{kernel} [on $\R^d$]
    is a Borel measurable function
    $\kappa:\R^d\to\bar\R_+$ such that $\kappa\in L^1_{\text{\it loc}}(\R^d)$.
    If, in addition, $\hat{\kappa}(\xi)\ge 0$
    for all $\xi\in\R^d$, then we say that
    $\kappa$ is a kernel of \emph{positive type}.
\end{definition}

Clearly, every kernel $\kappa$ can be redefined on
a Lebesgue-null set so that the resulting modification
$\tilde\kappa$ maps $\R^d$ into $\R_+$. However, we might
lose some of the nice properties of $\kappa$ by doing this.
A notable property is that $\kappa$ might be continuous;
that is, $\kappa(x)\to\kappa(y)$---in $\bar{\R}_+$---as $x$
converges to $y$ in $\R^d$. In this case, $\tilde\kappa$ might not
be continuous. From this perspective, it is sometimes
advantageous to work with the
$\bar{\R}_+$-valued function $\kappa$.
For examples, we have in
mind \emph{Riesz kernels}. They are defined
as follows: Choose and fix some number $\alpha\in(0\,,d)$, and then
let the Riesz kernel $\kappa_\alpha$ of index $\alpha$ be
\begin{equation}\label{eq:riesz}
    \kappa_\alpha(x) := \begin{cases}
        \| x\|^{\alpha-d}&\text{if $x\neq 0$},\\
        \infty&\text{if $x=0$}.
    \end{cases}
\end{equation}
It is easy to check that $\kappa_\alpha$ is a continuous kernel
for each $\alpha\in(0\,,d)$.
In fact, every $\kappa_\alpha$ is a kernel of positive type,
as can be seen via the following standard fact:
\begin{equation}
    \hat\kappa_\alpha(\xi)= c_{d,\alpha}\kappa_{d-\alpha}(\xi)
    \qquad\text{for all $\xi\in\R^d$}.
\end{equation}
Here $c_{d,\alpha}$ is a universal constant that depends only
on $d$ and $\alpha$;
see \ocite{Kahane:SRSF}*{p.\ 134} or \ocite{Mattila}*{eq.\ (12.10),
p.\ 161}, for example.

It is true---but still harder
to prove---that for all Borel probability measures $\mu$ on $\R^d$,
\begin{equation}\label{eq:ac:riesz}
    \iint \frac{\mu(dx)\,\mu(dy)}{\|x-y\|^{d-\alpha}}
    = \frac{c_{d,\alpha}}{(2\pi)^d} \int_{\R^d} |\hat{\mu}(\xi)|^2
    \,\frac{d\xi}{\|\xi\|^\alpha}.
\end{equation}
See \ocite{Mattila}*{p.\ Lemma 12.12, p.\ 162}.

The utility of \eqref{eq:ac:riesz} is in the fact that
it shows that Riesz-type energies of the left-hand side
are equal to P\'olya--Szeg\H{o} energies of the right-hand side.
This is a probabilistically significant fact. For example,
consider the case that $\alpha\in(0\,,2]$. Then,
$\mu\mapsto \iint \|x-y\|^{-d+\alpha}\,\mu(dx)\,\mu(dy)$ is the
``energy functional'' associated to continuous additive functionals
of various stable processes of index $\alpha$. At the same
time, $\mu\mapsto
(2\pi)^{-d}\int_{\R^d}|\hat\mu(\xi)|^2\, \|\xi\|^{-\alpha}
\,d\xi$ is a Fourier-analytic energy form of the type that
appears more generally in the earlier parts of the present paper.
Roughly speaking, $(2\pi)^{-d}\int_{\R^d}|\hat\mu(\xi)|^2\, \|\xi\|^{-\alpha}
\,d\xi \asymp I_{\Psi}(\mu)$, where $\Psi(\xi)=\|\xi\|^\alpha$ defines
the L\'evy exponent of an isotropic stable process of index
$\alpha$.
[Analytically speaking, this is
the Sobolev norm of $\mu$ that corresponds to
the fractional Laplacian operator $-(-\Delta)^{\alpha/2}$.]
Thus, we seek to find a useful generalization of \eqref{eq:ac:riesz}
that goes beyond one-parameter stable processes.

We define for all finite Borel measures
$\mu$ and $\nu$ on $\R^d$, and all kernels $\kappa$ on $\R^d$,
\begin{equation}
    \mathcal{E}_\kappa (\mu\,,\nu) := \iint \left(\frac{%
    \kappa(x-y)+\kappa(y-x)}{2}\right)\, \mu(dx)\,\nu(dy).
\end{equation}
This is called the \emph{mutual energy} between $\mu$ and $\nu$
in gauge $\kappa$, and defines a quadratic form with pseudo-norm
\begin{equation}
    \mathcal{E}_\kappa(\mu) := \mathcal{E}_\kappa(\mu\,,\mu).
\end{equation}
This is the ``$\kappa$-energy'' of the measure $\mu$.
There is a corresponding capacity defined as
\begin{equation}
    \mathscr{C}_\kappa(F) := \frac{1}{\inf\mathscr{E}_\kappa(\mu)},
\end{equation}
where the infimum is taken over all compactly supported
probability measures $\mu$ on $F$, $\inf\varnothing:=\infty$,
and $1/\infty:=0$.

We can recognize the left-hand side of \eqref{eq:ac:riesz}
to be $\mathcal{E}_{\kappa_{d-\alpha}}(\mu)$.
Because $\kappa_{d-\alpha}(x)<\infty$ if and only if
$x\neq 0$, the following is a nontrivial generalization of
\eqref{eq:ac:riesz}.

\begin{theorem}\label{th:ac:fourier}
    Suppose $\kappa$ is a continuous kernel of positive type on
    $\R^d$ which satisfies one of the following two conditions:
    \begin{enumerate}
        \item $\kappa(x)<\infty$ if and only if
            $x\neq 0$;
        \item $\hat\kappa\in L^\infty(\R^d)$, and $\kappa(x)<\infty$ when
            $x\neq 0$.
    \end{enumerate}
    Then for all Borel probability measures
    $\mu$ and $\nu$ on $\R^d$,
    \begin{equation}
        \mathcal{E}_{\kappa*\nu}(\mu)
        = \frac{1}{(2\pi)^d}
        \int_{\R^d} \hat{\kappa}(\xi)\,\Re \hat{\nu}(\xi)\, \left| \hat{\mu}(\xi)
        \right|^2\, d\xi.
    \end{equation}
\end{theorem}

\begin{remark}\label{rem:ac:fourier}
    If $\kappa\in L^1(\R^d)$, then $\hat{\kappa}$ can be define
    by the usual Fourier transform,
    $\hat\kappa(\xi) = \int_{\R^d} \exp(ix\cdot \xi)\kappa(x)\, dx$.
    That is, the condition $\hat\kappa\in L^\infty(\R^d)$ is automatically
    verified in this case. In fact, $\hat\kappa$ is bounded
    in this case, as can be seen from
    $\sup_{\xi\in\R^d} |\hat\kappa(\xi)|=\|\kappa\|_{L^1(\R^d)}$.\qed
\end{remark}

Our proof of Theorem \ref{th:ac:fourier} proceeds in four steps;
the first three are stated as lemmas.

The folklore of harmonic analysis contains precise versions of the loose
assertion that ``typically, kernels of positive type achieve
their supremum at the origin.'' The first step in the proof of
Theorem \ref{th:ac:fourier} is to verify a
suitable form of this statement.

\begin{lemma}\label{lem:sup:zero}
    Suppose $\kappa$ is a kernel of positive type such that
    $\hat\kappa\in L^\infty(\R^d)$. Suppose also that $\kappa$ is
    continuous on all of $\R^d$, and finite on $\R^d\setminus\{0\}$.
    Then, $\kappa(0)=\sup_{x\in\R^d}\kappa(x)$.
\end{lemma}

\begin{proof}
    Recall the functions $\phi_\e$ from the proof
    of Theorem \ref{th:main:tilde}; see \eqref{Def:phie}. Then,
    $\kappa*\phi_\e\in L^1(\R^d)$ and $\widehat{\kappa*\phi}_\e =
    \hat\kappa
    \hat{\phi}_\e \ge 0$.
    But $\text{ess sup}_{\xi\in\R^d}|\hat{\kappa}(\xi)|<\infty$
    and $\sup_{\xi\in\R^d}|\hat{\phi}_\e(\xi)|\le\|\phi_\e\|_{L^1(\R^d)}=1.$
    Moreover, the construction of $\phi_\e$ ensures that
    $\hat{\phi}_\e$ is integrable. Therefore,
    $\widehat{\kappa*\phi}_\e \in L^1(\R^d)\cap L^\infty (\R^d)$, and
    for Lebesgue-almost all $x\in\R^d$,
    \begin{equation}\label{eq:5.1}
        (\kappa*\phi_\e)(x) = \frac{1}{(2\pi)^d}
        \int_{\R^d} e^{-ix\cdot\xi}\,
        \hat{\kappa}(\xi)\, \hat{\phi}_\e (\xi)
        \, d\xi,
    \end{equation}
    thanks to the inversion formula for Fourier transforms. Since both sides of
    \eqref{eq:5.1} are continuous
    functions of $x$, that equation is valid for all $x\in\R^d$.
    Furthermore, because $\kappa$ is a kernel of positive type,
   it follows that $(\kappa*\phi_\e)(x)\le (2\pi)^{-d}
   \int_{\R^d} \hat{\kappa}(\xi)\, d\xi$.
    But if $x\neq 0$, then
    $\lim_{\e\downarrow 0} (\kappa*\phi_\e)(x)=\kappa(x)$, and hence,
    \begin{equation}
        \sup_{x\in\R^d\setminus\{0\}} \kappa (x)
        \le \frac{1}{(2\pi)^d}
        \int_{\R^d} \hat{\kappa}(\xi)\, d\xi.
    \end{equation}
    It suffices to prove that
    \begin{equation}\label{eq:kappa:goal}
        \frac{1}{(2\pi)^d}
        \int_{\R^d} \hat{\kappa}(\xi)\, d\xi
        \le \kappa(0).
    \end{equation}
    This holds trivially if $\kappa(0)$ is infinite.
    Therefore, we may assume without loss of generality
    that $\kappa(0)<\infty$. The continuity of $\kappa$
    ensures that it is uniformly continuous in
    a neighborhood of the origin, thence we have
    $\lim_{\e\downarrow 0}(\kappa*\phi_\e)(0)=\kappa(0)$ by
    the classical Fej\'er theorem.
    Also, we recall that $\lim_{\e\downarrow 0}\hat{\phi}_\e(\xi)
    =1$ for all $\xi\in\R^d$.
    We use these facts in conjunction with
    \eqref{eq:5.1} and Fatou's lemma
    to deduce \eqref{eq:kappa:goal}, and hence
    the lemma.
\end{proof}

Next we present the second step in the proof of Theorem
\ref{th:ac:fourier}. This is another folklore fact from harmonic
analysis.

We say that a kernel $\kappa$
is \emph{lower semicontinuous} if there exist a sequence of
continuous functions
$\kappa_1,\kappa_2,\ldots:\R^d\to\R_+$ such that
$\kappa_n(x)\le\kappa_{n+1}(x)$ for all $n\ge 1$ and
$x\in\R^d$, such that $\kappa_n(x)\uparrow \kappa(x)$
for all $x\in\R^d$, as $n\uparrow\infty$. Because
$\kappa_n(x)$ is assumed to be in $\R_+$ [and not
$\bar{\R}_+$], our definition of lower semicontinuity
is slightly different from the usual one. Nonetheless,
the following is a consequence of Lemma 12.11 of
\ocite{Mattila}*{p.\ 161}.

\begin{lemma}\label{lem:kernel:energy:UB}
    Suppose $\kappa$ is a lower semicontinuous
    kernel of positive type on $\R^d$. Then,
    for all Borel probability measures $\mu$ on
    $\R^d$,
    \begin{equation}
        \mathcal{E}_\kappa(\mu)
        \le \frac{1}{(2\pi)^d} \int_{\R^d} \hat{\kappa}
        (\xi)\, |\hat{\mu}(\xi)|^2\, d\xi.
    \end{equation}
\end{lemma}

Our next lemma constitutes the
third step of our proof of Theorem \ref{th:ac:fourier}.

\begin{lemma}\label{lem:ac:fourier}
    Suppose $\kappa$ is a continuous kernel of positive type on
    $\R^d$, which satisfies one of the following two conditions:
    \begin{enumerate}
        \item $\kappa(x)<\infty$ if and only if
            $x\neq 0$;
        \item $\hat\kappa\in L^\infty(\R^d)$, and $\kappa(x)<\infty$ when
            $x\neq 0$.
    \end{enumerate}
    Then, for all Borel probability measures
    $\mu$ on $\R^d$,
    \begin{equation}
        \mathcal{E}_\kappa(\mu)= \frac{1}{(2\pi)^d}
        \int_{\R^d} \hat{\kappa}(\xi)\, \left| \hat{\mu}(\xi)
        \right|^2\, d\xi.
    \end{equation}
\end{lemma}

A weaker version of this
result is stated in \ocite{Kahane:SRSF}*{p.\ 134} without proof.

According to Kahane (\emph{loc.\ cit.}), functions $\kappa$
that satisfy the conditions of Lemma \ref{lem:ac:fourier}
are called \emph{potential kernels}. They can be defined,
in equivalent terms, as kernels of positive types that
are continuous on $\R^d\setminus\{0\}$ and $\lim_{x\to 0}\kappa(x)=\infty$.
We will not use this terminology: the term ``potential kernel'' is
reserved for another object.

\begin{proof}[Proof of Lemma \ref{lem:ac:fourier}]
    Regardless of whether $\kappa$ satisfies condition
    (1) or (2), it is lower semicontinuous. Therefore,
    in light of Lemma \ref{lem:kernel:energy:UB}, it suffices
    to prove that
    \begin{equation}\label{goal:lem:ac:fourier}
        \mathcal{E}_\kappa(\mu)
        \ge \frac{1}{(2\pi)^d} \int_{\R^d} \hat{\kappa}
        (\xi)\, |\hat{\mu}(\xi)|^2\, d\xi.
    \end{equation}
    From here on, our proof considers two separate cases:\\

    \emph{Case 1.} First, let us suppose $\kappa$ satisfies
    condition (1) of the lemma.

    Without loss of generality, we may assume that
    $\mathcal{E}_\kappa(\mu) < \infty$.
    Since $\lim_{x\to 0}\kappa(x)=\infty$, this
    implies that $\mu$ does not charge singletons.

    According the Lusin's theorem, for every $\eta \in (0\,, 1)$ we can
    find a compact set $K_\eta \subseteq \R^d \backslash \{0\}$ with
    $\mu(K_\eta^c) \le \eta$ such that $\kappa*\mu$ is continuous---and hence
    uniformly continuous---on $K_\eta$. Define for all Borel sets $A
    \subseteq \R^d$,
    \begin{equation}
        \mu_\eta(A) = \frac{\mu(A \cap K_\eta)} { 1 - \eta}.
    \end{equation}
    Then $\mu_\eta$ is supported by the compact set $K_\eta$, and
    \begin{equation}
        \frac{1}{1 -\eta} \ge \mu_\eta(K_\eta) = \frac{\mu(K_\eta) }{1 -\eta} \ge
        1.
    \end{equation}

    Now $\lim_{r \to 0} \mu_\eta(B(x\,, r)) = 0$ for all $x \in \R^d$,
    where $B(x\,, r)$ denotes the $\ell^2$-ball of radius $r > 0$ about $x
    \in \R^d$. Therefore, a
    compactness argument reveals that for all $\eta \in (0\,,
    1)$,
    \begin{equation}\label{Eq:eta1}
        \lim_{r \downarrow 0} \sup_{x \in K_\eta} \mu_\eta(B(x\,, r)) = 0.
    \end{equation}
    For otherwise we can find $\delta > 0$ and $x_r \in K_\eta$ such
    that for all $r > 0$, $\mu_\eta(B(x_r\,, r)) \ge \delta$. By
    compactness we can extract a subsequence $r' \to 0$ and $x \in
    K_\eta$ such that $x_{r'} \to x$. It follows easily then
    $\mu_\eta(B(x\,, \e)) \ge \delta$ for all $\e > 0$, whence
    $\mu_\eta(\{x\}) \ge \delta$, which contradicts the fact that
    $\mu_\eta$ does not charge singletons.

    Next we choose and fix $y \in K_\eta$ and $\eta > 0$. We claim that
    \begin{equation}\label{Eq:eta2}
        \sup_{x \in K_\eta} \int_{B(x, \e)}
        \kappa (y - z)\, \mu_\eta(dz) \to 0
        \quad \hbox{ as }\, \e \downarrow 0.
    \end{equation}
    Indeed, by \eqref{Eq:eta1} and the fact that $(\kappa*\mu_\eta)(y) <
    \infty$, we see that  for all $\rho
    > 0$ there exists $\theta > 0$ such that
    \begin{equation}\label{Eq:eta3}
        \lim_{\e \downarrow 0} \, \sup_{x\in K_\eta: |x-y|\le\theta/2}
        \int_{B(x, \e)} \kappa(y - z)\, \mu_\eta(dz) \le \rho.
    \end{equation}
    On the other hand, by the continuity of $\kappa$ on $\R^d \backslash
    \{0\}$ and \eqref{Eq:eta1}, we have
    \begin{equation}\label{Eq:eta4}
        \lim_{\e \downarrow 0} \, \sup_{x \in K_\eta: |x - y| > \theta/2}\,
        \int_{B(x, \e)} \kappa(y - z)\, \mu_\eta(dz) = 0.
    \end{equation}
    By combining (\ref{Eq:eta3}) and \eqref{Eq:eta4}, we find that
    \begin{equation}\label{Eq:eta5}
        \lim_{\e \downarrow 0} \, \sup_{x \in K_\eta} \int_{B(x, \e)}
        \kappa(y - z)\, \mu_\eta(dz) \le \rho.
    \end{equation}
    Thus  \eqref{Eq:eta2} follows from \eqref{Eq:eta5}, because
    we can choose $\rho$ as small as we want.
    By $\eqref{Eq:eta2}$ and another appeal to compactness, we obtain
    \begin{equation}\label{Eq:eta6}
        \sup_{y \in K_\eta} \, \sup_{x \in K_\eta} \int_{B(x, \e)}
        \kappa(y - z)\, \mu_\eta(dz) \to 0 \quad \hbox{ as }\, \e \downarrow 0.
    \end{equation}
    Consequently, for all  $k \ge 1$,  we can find $\e_k \to 0$ such
    that
    \begin{equation}\label{Eq:eta7}
        \sup_{y \in K_\eta} \, \sup_{x \in K_\eta} \left| (\kappa
        *\mu_\eta)(y) - \int_{B(x, \e_k)^c} \kappa(y - z) \, \mu_\eta(dz)
        \right| \le \frac 1 k.
    \end{equation}

    Let $y \in K_\eta$ and let $\{y_n\}_{n=1}^\infty$ be an arbitrary sequence in
    $K_\eta$ such that $\lim_{n \to \infty} y_n = y$. Because $z \mapsto
    \kappa(y - z)$ is uniformly continuous on $B(y\,, \e_k)^c \cap K_\eta$, we
    have
    \begin{equation}
    \lim_{n \to \infty} \int_{B(y, \e_k)^c} \kappa(y_n - z) \, \mu_\eta(dz) =
        \int_{B(y, \e_k)^c} \kappa(y - z) \, \mu_\eta(dz).
    \end{equation}
    This and \eqref{Eq:eta7} together imply that for all $k \ge 1$,
    \begin{equation}\label{Eq:eta8}
        \lim_{n \to \infty}   \big| (\kappa*\mu_\eta)(y_n) - (\kappa*\mu_\eta)
        (y)\big| \le \frac 2 k.
    \end{equation}
    Let $k \uparrow \infty$ to deduce that $\kappa*\mu_\eta$ is continuous,
    and hence uniformly continuous, on $K_\eta$. On the other hand, it
    can be verified directly that $\kappa*\mu_\eta$ is continuous on
    $K_\eta^c$. Hence, we have shown that $\kappa*\mu_\eta$ is continuous on
    $\R^d$.

    If $0<\e,\eta<1$, then we can appeal to the Fubini-Tonelli theorem,
    a few times in succession, to deduce that
    \begin{equation}
        \mathcal{E}_{\kappa*\psi_\e}(\mu_\eta)
        = \int \kappa(-z) \left( a_{\e,\eta}* b_{\e,\eta}\right)(z)\, dz,
    \end{equation}
    where $a_{\e,\eta} := \phi_\e * \mu_\eta$,
    $b_{\e,\eta} := \phi_\e * \breve{\mu}_\eta$,
    and $\breve{\mu}_\eta$ is the Borel probability measure on
    $\R^d$ that is defined by
    \begin{equation}
        \breve{\mu}_\eta (A) := \mu_\eta(-A)\qquad
        \text{for all $A\subseteq\R^d$}.
    \end{equation}
    We observe the following elementary facts:
    \begin{enumerate}
        \item Both $a_{\e,\eta}$ and $b_{\e,\eta}$ are infinitely differentiable
            functions of compact support;
        \item because $\phi_\e$ is of positive type,
            the Fourier transform of $a_{\e,\eta}*b_{\e,\eta}$ is
            $|\hat\phi_\e|^2 |\hat\mu_\eta|^2$;
        \item the Fourier transform of
            $z\mapsto \kappa(-z)$ is the same as that of $\kappa$
            because $\kappa$ is of positive type;
    \end{enumerate}
    Therefore, we can combine the preceding with the Parseval
    identity, and deduce that
    \begin{equation}\label{eq:E:smoothed:twice}\begin{split}
        \mathcal{E}_{\kappa*\psi_\e}(\mu_\eta)
            &= \frac{1}{(2\pi)^d}\int_{\R^d} \hat\kappa(\xi)\,
            \overline{\hat{a}_{\e,\eta}(\xi)\,
            \hat{b}_{\e,\eta}(\xi)}\, d\xi\\
        &= \frac{1}{(2\pi)^d}\int_{\R^d} \hat{\kappa}(\xi)\,
            |\hat{\phi}_\e(\xi)|^2\,\left| \hat{\mu}_\eta(\xi)\right|^2\, d\xi.
    \end{split}\end{equation}
    We can apply the Fubini-Tonelli theorem to write the left-most term in
    another way, as well. Namely,
    \begin{equation}
        \mathcal{E}_{\kappa*\psi_\e}(\mu_\eta) =
        \int \left(\kappa*\psi_\e*\mu_\eta\right)\, d\mu_\eta.
    \end{equation}
    The continuity of $\kappa*\mu_\eta$, and Fej\'er's theorem,
    together imply
    that $\kappa*\mu_\eta*\psi_\e$ converges to
    $\kappa*\mu_\eta$ uniformly on
    $K_\eta$ as $\e\downarrow  0$, and hence
    $\mathcal{E}_{\kappa*\psi_\e}(\mu_\eta)$ converges to
    $\mathcal{E}_\kappa(\mu_\eta)$ as $\e\downarrow 0$.
    This and \eqref{eq:E:smoothed:twice} together imply that
    \begin{equation}\label{Eq:eta10}\begin{split}
        \mathcal{E}_\kappa(\mu_\eta)
            &= \frac{1}{(2\pi)^d} \lim_{\e \to 0}  \int_{\R^d} \hat{\kappa}(\xi)\,
            |\hat{\phi}_\e(\xi)|^2\,\left| \hat{\mu}_\eta(\xi)\right|^2\, d\xi\\
        &\ge \frac{1}{(2\pi)^d} \int_{\R^d}\hat{\kappa}(\xi)\,  \left|
            \hat{\mu}_\eta(\xi)\right|^2\, d\xi,
    \end{split}\end{equation}
    owing to Fatou's lemma.
    Since $\mu_\eta$ is $(1-\eta)^{-1}$ times a restriction of
    $\mu$, it follows that
    \begin{equation}\label{Eq:eta11}
        \frac{1}{(1 - \eta)^2}\mathcal{E}_\kappa(\mu)
        \ge \frac{1}{(2\pi)^d}
        \int_{\R^d}\hat{\kappa}(\xi)\,  \left| \hat{\mu}_\eta(\xi)\right|^2\,
        d\xi.
    \end{equation}
    Because for all $ \xi \in \R^d$,
    \begin{equation}\label{Eq:eta12}
        \left| \hat{\mu}_\eta (\xi) - \frac{\hat \mu (\xi) } {1 - \eta}
        \right|\le \frac{ \mu(K_\eta^c)} {1 - \eta} \le \frac{\eta}  {1-
        \eta},
    \end{equation}
    we deduce that $\lim_{\eta \to 0} \hat{\mu}_\eta = \hat{\mu}$
    pointwise. An appeal to Fatou's lemma
    justifies \eqref{goal:lem:ac:fourier}, and this completes our proof
    in the first case.\\

    \emph{Case 2.} Next, we consider the case that $\kappa$
        satisfies condition (2) of the lemma. If $\kappa(0)=\infty$,
        then condition (1) is satisfied, and therefore the
        proof is complete. Thus, we may assume that $\kappa(0)<\infty$.
        According to Lemma
        \ref{lem:sup:zero}, $\kappa$ is a bounded and continuous
        function from $\R^d$ into $\R_+$. For all Borel sets
        $A\subseteq \R^d$ define
        \begin{equation}
            \mu_n(A) := \frac{\mu_n\left( A\cap [-n\,,n]^d\right)}{
            \chi_n},\quad\text{where}\quad
            \chi_n:=\mu\left([-n\,,n]^d\right).
        \end{equation}
        If $n>0$ is sufficiently large, then $\mu_n$ is a well-defined
        Borel probability measure on $[-n\,,n]^d$, and since
        $\kappa$ is uniformly continuous on $[-n\,,n]^d$,
        \begin{equation}\begin{split}
            \frac{1}{\chi_n^2}\mathcal{E}_\kappa(\mu)
                &\ge\mathcal{E}_\kappa(\mu_n)\\
            &=\lim_{\e\downarrow 0} \mathcal{E}_{\kappa*\psi_\e}(\mu_n)\\
            &= \frac{1}{(2\pi)^d}\lim_{\e\downarrow 0}
                \int_{\R^d}\hat{\kappa}(\xi)\,
                \hat{\psi_\e}(\xi)\, |\hat{\mu}_n(\xi)|^2\, d\xi.
        \end{split}\end{equation}
        The first equality holds because $\kappa*\psi_\e$ converges
        uniformly to $\kappa$ on $[-n\,,n]^d$, as $\e\downarrow 0$.
        The second is a consequence of the Parseval identity.
        Thanks to Fatou's lemma, we have proved that for all
        $n>0$ sufficiently large,
        \begin{equation}
            \frac{1}{\chi_n^2}\mathcal{E}_\kappa(\mu)
            \ge \frac{1}{(2\pi)^d}
            \int_{\R^d}\hat{\kappa}(\xi)\,
            |\hat{\mu}_n(\xi)|^2\, d\xi.
        \end{equation}
        As $n$ tends to infinity, $\chi_n\uparrow 1$
        and $\hat{\mu}_n\to \hat{\mu}$ pointwise. Therefore, another
        appeal to Fatou's lemma implies \eqref{goal:lem:ac:fourier},
        and hence the lemma.
\end{proof}

Now we derive Theorem \ref{th:ac:fourier}.

\begin{proof}[Proof of Theorem \ref{th:ac:fourier}]
    Let us observe that by the Fubini-Tonelli theorem,
    \begin{equation}\label{eq:adjoint}
        \mathcal{E}_{\kappa*\nu}(\mu)
        = \mathcal{E}_\kappa \left(\mu\,,\breve{\nu}*\mu\right).
    \end{equation}
    In fact, both sides are equal to $\E[\kappa(X-X'-Y)]$, where
    $(X\,,X'\,,Y)$ are independent, $X$ and $X'$ are distributed
    as $\mu$, and $Y$ is distributed as $\nu$.

    Lemma \ref{lem:ac:fourier}, and polarization, together imply
    that the following holds
    for every Borel probability measure $\sigma$ on $\R^d$:
    \begin{equation}\label{eq:polarized}
        \mathcal{E}_\kappa(\mu\,,\sigma) =
        \frac{1}{(2\pi)^d} \int_{\R^d} \hat{\kappa}(\xi)\,
        \hat{\sigma}(\xi)\, \overline{\hat{\mu}(\xi)}\,
        d\xi.
    \end{equation}
    Indeed, we first notice that
    \begin{equation}
        \mathcal{E}_\kappa\left( \frac{\mu+\sigma}{2}\right)
        =\frac14 \mathcal{E}_\kappa(\mu) + \frac14\mathcal{E}_\kappa
        (\sigma) + \frac12 \mathcal{E}_\kappa(\mu\,,\sigma).
    \end{equation}
    Thus, we solve for $\mathcal{E}_\kappa(\mu\,,\sigma)$
    and apply Lemma \ref{lem:ac:fourier} to deduce that
     \begin{equation}\begin{split}
        \mathcal{E}_\kappa(\mu\,,\sigma) &=
            \frac{2}{(2\pi)^d} \int_{\R^d}\hat{\kappa}(\xi)\,
            \left| \widehat{\left(\frac{\mu+\sigma}{2}\right)}
            (\xi)\right|^2\, d\xi\\
        &\hskip1in - \frac{1}{2(2\pi)^d} \int_{\R^d}\hat{\kappa}(\xi)\,
            |\hat{\mu}(\xi)|^2\, d\xi -
            \frac{1}{2(2\pi)^d} \int_{\R^d}\hat{\kappa}(\xi)\,
            |\hat{\sigma}(\xi)|^2\, d\xi.
    \end{split}\end{equation}
    We solve to obtain \eqref{eq:polarized}.

    Thus we define $\sigma:=\breve{\nu}*\mu$, observe
    that $\sigma$ is a Borel probability measure on $\R^d$,
    and $\mathcal{E}_{\kappa*\nu}(\mu)
    = \mathcal{E}_\kappa(\mu\,,\sigma)$.
    Thus, we may apply \eqref{eq:adjoint} and \eqref{eq:polarized}---in
    this order---to find that
    \begin{equation}\label{eq:lm:rm}\begin{split}
        \mathcal{E}_{\kappa*\nu}(\mu)
            &= \frac{1}{(2\pi)^d}\int_{\R^d} \hat{\kappa}(\xi)\,
            \hat{\sigma}(\xi)\, \overline{\hat{\mu}(\xi)}\,
            d\xi\\
        &= \frac{1}{(2\pi)^d}\int_{\R^d} \hat{\kappa}(\xi)\,
            \hat{\nu}(-\xi)\, |\hat{\mu}(\xi)|^2\,
            d\xi.
    \end{split}\end{equation}
    In the last line, we have only used the fact that
    $\hat{\sigma}(\xi)=\hat{\nu}(-\xi)\hat{\mu}(\xi)$
    for all $\xi\in\R^d$. Because the left-most term in
    \eqref{eq:lm:rm} is real-valued,
    so is the right-most term. Therefore, we may consider only the real
    part of the right-most item in \eqref{eq:lm:rm}, and this proves the result.
\end{proof}

\section{\bf Absolute continuity considerations}

Suppose $E$ is a Borel measurable subset
of $\R^N$ that has positive Lebesgue
measure, and $Y:=\{Y(\bm{t})\}_{\bm{t}\in E}$ is
an $\R^d$-valued random field that is indexed
by $E$.
\begin{definition}\label{def:one:pot}
    We say that $Y$ has a \emph{one-potential
    density} $v$ if $v\in L^1(\R^d)$ is nonnegative
    and satisfies
    the following for all Borel measurable functions
    $f:\R^d\to\R_+$:
    \begin{equation}
        \frac{1}{\int_E e^{-[\bm{s}]}
        \, d\bm{s}}\,\E \left[\int_E f(Y(\bm{t})) e^{-[\bm{t}]}
        \, d\bm{t}\right] = \int_{\R^d} f(x)v(x)\,dx.
    \end{equation}
\end{definition}
In particular, $\widetilde{\X}$ has a one-potential
density $v$ if for all Borel measurable functions
$f:\R^d\to\R_+$,
\begin{equation}\label{eq:R:brevev}
    (Rf)(x) =\int_{\R^d} f(x+y) v(y)\, dy :=
    (f*\breve{v})(x)\qquad
    \text{for all $x\in\R^d$},
\end{equation}
where $\breve{g}(z):=g(-z)$ for all functions $g$. It follows from
Lemma 2.1 in \ocite{Hawkes:79} that $\widetilde{\X}$ has a one-potential
density if and only if the operator $R$ defined by \eqref{Def:Rf} is
\emph{strong Feller}. That is, if $Rf$ is continuous whenever $f$ is
Borel measurable and has compact support.

In order to be concrete, we choose a ``nice'' version of the one-potential
density $v$, when it exists. Before we proceed further,
let us observe that when $v$ exists it is a probability density.
In particular, the Lebesgue density theorem tell us that
\begin{equation}
    v(x) = \lim_{\e\downarrow 0} \frac{1}{\lambda_d(B(0\,,\e))}
    \int_{B(x,\e)} v(y)\, dy\qquad\text{for almost all $x\in\R^d$
    $[\lambda_d]$}.
\end{equation}
We can recognize the integral as $(R\1_{B(0,\e)})(-x)$.
Moreover, we can alter $v$ on a Lebesgue-null set and still obtain
a one-potential density for $\widetilde{\X}$. Therefore,
from now on, we \emph{always} choose the following version
of $v$:
\begin{equation}\label{eq:v}
    v(x) := \liminf_{\e\downarrow 0} \frac{\left(
    R\1_{(0,\e)}\right)(-x)}{\lambda_d(B(0\,,\e))}
    \qquad\text{for all $x\in\R^d$}.
\end{equation}

Lemma \ref{lem:P:U} states that
the Fourier multiplier of $R$ is $K_{\bm{\Psi}}$. Therefore,
the Fourier transform of $\breve{v}$ is also $K_{\bm{\Psi}}$;
confer with \eqref{eq:R:brevev}.
Because $K_{\bm{\Psi}}$ is real-valued, this proves
the following:

\begin{lemma}\label{lem:vhat:kappa}
    If $\widetilde{\X}$ has a one-potential density $v$,
    then $\hat{v}=K_{\bm{\Psi}}$ in the sense of
    Schwartz.
    In particular, $v$ is an integrable kernel of positive
    type on $\R^d$.
\end{lemma}

\begin{remark}\label{rem:MR}
    Let $\e_1,\ldots,\e_N$ denote $N$
    random variables, all independent of one another,
    as well as $\{X_j\}_{j=1}^N$ and $\{X'_j\}_{j=1}^N$,
    with $\P\{\e_1=\pm 1\}=1/2$.
    We can re-organize the order of integration a few times
    to find that
    \begin{equation}
        \frac{1}{2^N}\E\left[\int_{\R^N}
        f\left(\widetilde{\X}(\bm{t})\right)
        e^{-[\bm{t}]}\,d\bm{t}\right] =
        \E\left[\int_{\R^N_+}
        f\left( \sum_{j=1}^N \e_j X_j(s_j) \right)
        e^{-[\bm{s}]}\, d\bm{s}\right].
    \end{equation}
    Thus, in particular, $\widetilde{\X}$ has a one-potential
    density if and only if the $\R^N_+$-indexed random field
    $(t_1\,,\ldots,t_N)\mapsto \sum_{j=1}^N\e_j X_j(t_j)$
    does [interpreted in the obvious sense].
    Interestingly enough, the latter random field appears
    earlier---though
    for quite different reasons as ours---in the works of
    Marcus and Rosen \ycites{MarcusRosen:b,MarcusRosen:a}.\qed
\end{remark}

\begin{remark}\label{rem:pot:density}
    It follows from the definitions that
    a random field $\{Y(\bm{t})\}_{\bm{t}\in E}$ has a one-potential
    density if and only if the Borel probability measure
    \begin{equation}\label{ex:ac}
        \R^d\supseteq A\mapsto \E\left[\int_E \1_A(Y(\bm{t}))
        e^{-[\bm{t}]}\, d\bm{t}\right]
        \qquad\text{
        is absolutely continuous.}
    \end{equation}
    Now suppose $f\ge 0$ is Borel measurable. Then
    \begin{equation}
        \int_{\R^N}
        f\left(\widetilde{\X}(\bm{t})\right)
        e^{-[\bm{t}]}\,d\bm{t} \ge \int_{\R^N_+} f(\X(\bm{t})) e^{-[\bm{t}]}\,
        d\bm{t}.
    \end{equation}
    We take expectations of both sides to deduce from \eqref{ex:ac} that
    if $\widetilde{\X}$ has a one-potential density, then so does $\X$.
    When $N=1$, it can be verified that the converse it also true. However,
    the previous remark can be used to show that, when $N \ge 2$, the converse
    is not necessarily true. Thus, we can conclude
    that the existence of a one-potential density for
    $\widetilde{\X}$ is a \emph{more stringent} condition than the existence
    of a one-potential density for $\X$. Another consequence of
    the preceding is the following: If the one-potential density of
    $\widetilde{\X}$ exists and the potential
    density of $\X$ is a.e.-positive, then
    the one-potential density of $\widetilde{\X}$ is per force
    also a.e.-positive. We have, and will, encounter these
    conditions several times.\qed
\end{remark}

Our next theorem is the main result of this subsection. From a
technical point of view, it is also a key result in this paper.
In order to describe it properly, we introduce some notation
first.

If $T$ is a nonempty subset of $\{1\,,\ldots,N\}$, then we define
$|T|$ to be the cardinality of $T$, and $\X_T$ to be the
\emph{subprocess} associated to the index $T$. That is,
$\X_T$ is the following $|T|$-parameter, $\R^d$-valued random
field:
\begin{equation}\label{eq:subprocess}
    \X_T(\bm{t}) := \sum_{j\in T} X_j (t_j)\qquad
    \text{for all $\bm{t}\in\R^{|T|}_+$}.
\end{equation}
In order to obtain nice formulas, we define $\X_\varnothing$ to be
the constant $0$.
In this way, it follows that, regardless of whether or not
$|T|>0$,
each $\X_T$ is itself an additive L\'evy process,
and the L\'evy exponent of $\X_T$ is the function
$(\Psi_j)_{j\in T}:\R^d\to\mathbf{C}^{|T|}$. Thus, we can talk about
the stationary field $\widetilde{\X}_T$ for all $T\subseteq\{1\,,
\ldots,N\}$, etc. Despite this new notation, we continue to write
$\X$ and $\widetilde{\X}$ in place of the more cumbersome
$\X_{\{1,\ldots,N\}}$ and $\widetilde{\X}_{\{1,\ldots,N\}}$.

\begin{theorem}\label{th:main:v}
    Suppose there exists a nonrandom and nonempty subset
    $T\subseteq\{1\,,\ldots,N\}$, such that
    $\widetilde{\X}_T$ has a
    one-potential density $v_T:\R^d\mapsto\bar\R_+$ that is
    continuous on $\R^d$, and finite on $\R^d\setminus\{0\}$.
    Then, $\widetilde{\X}$ has a one-potential density
    $v$ on $\R^d$, and for all Borel
    sets $F\subseteq\R^d$,
    \begin{equation} \label{Eq:main--v}
        \E\left[\lambda_d\left( \X(\R^N_+) \oplus F\right) \right]
        >0\quad\Leftrightarrow\quad
        \mathscr{C}_v(F)>0.
    \end{equation}
\end{theorem}

\begin{proof}
    First, consider the case that $T=\{1\,,\ldots,N\}$. In this case,
    $\widetilde\X$ has a one-potential density $v:=v_{\{1,\ldots,N\}}$
    that is continuous
    on $\R^d$, and finite away from the origin.
    In light of Remark \ref{rem:ac:fourier},
    Lemma \ref{lem:vhat:kappa} shows us that $\kappa:=v$ satisfies
    condition (2) of Theorem \ref{th:ac:fourier}.
    Theorem \ref{th:ac:fourier}, in turn, implies that
    $I_{\bm{\Psi}}(\mu)= \mathscr{E}_v(\mu)$ for all
    Borel probability measures $\mu$ on $\R^d$, and thence
    $\mathscr{C}_v(F)= \mathrm{cap}_{\bm{\Psi}}(F)$. This and
    Theorem \ref{th:main} together imply Theorem \ref{th:main:v}
    in the case that $\widetilde\X$ has a continuous one-potential
    density that is finite on $\R^d\setminus\{0\}$.

    Next, consider the remaining case that $1\le |T|\le N-1$.
    For all Borel sets $A\subseteq\R^d$ define
    \begin{equation}
        V_{T^c}(A) :=\E\left[\int_{\R^{N-|T|}}
        \1_A\left( \widetilde{\X}_{T^c}(\bm{s})\right)
        e^{-[\bm{s}]}\, d\bm{s}\right].
    \end{equation}
    This is the one-potential measure for the stationary field
    based on the additive L\'evy process $\X_{T^c}$.
    The corresponding object for $\X_T$ can be defined likewise, viz.,
    \begin{equation}
        V_T(A) :=\E\left[\int_{\R^{|T|}}
        \1_A\left( \widetilde{\X}_{T}(\bm{t})\right)
        e^{-[\bm{t}]}\, d\bm{t}\right].
    \end{equation}

    We can observe that
    \begin{equation}\begin{split}
        V_{T^c}(A) &:=\E\left[\int_{\R^N_+}
            \1_A\left( \sum_{j\in T^c} \e_j X_j(s_j)\right)
            e^{-[\bm{s}]}\, d\bm{s}\right],\\
        V_{T}(A) &:=\E\left[\int_{\R^N_+}
            \1_A\left( \sum_{j\in T} \e_j X_j(s_j)\right)
            e^{-[\bm{s}]}\, d\bm{s}\right],
    \end{split}\end{equation}
    where $\e_1,\ldots,\e_N$ are i.i.d.\ random functions
    that are totally independent of $X_1,\ldots,X_N$ and
    $\widetilde{X}_1,\ldots,\widetilde{X}_N$, and
    $\P\{\e_1=\pm 1\}=1/2$.

    Similarly, we can write for all Borel measurable functions
    $f:\R^d\to\R_+$ and $x\in\R^d$,
    \begin{equation}
        (Rf)(x) =  \E\left[\int_{\R^N_+} f\left(x+ \sum_{j=1}^N \e_j
        X_j(t_j) \right) e^{-[\bm{t}]}\, d\bm{t}\right].
    \end{equation}
    Note that the integral is now over $\R^N_+$
    [and not $\R^N$]. Moreover, we can write
    \begin{equation}\begin{split}
        (Rf)(x) &= \iint f(x+y +z) V_T(dy)\, V_{T^c}(dz)\\
        &= \int f(x+y)\, \left( V_T*{V}_{T^c}\right)(dx),
    \end{split}\end{equation}
    The condition that $\widetilde\X_T$ has a one-potential density
    $v_T$ is equivalent to the statement that $V_T(dy)=v_T(y)\, dy$,
    whence it follows that $(Rf)(x)=\int f(x+y)v(y)\, dy$, where
    \begin{equation}\label{eq:v:vs:vT}
        v(y) := \int v_T(y-z) V_{T^c}(dz).
    \end{equation}
    This proves the assertion that $\widetilde\X$ has a one-potential
    density $v$. Note that $v$ is of the form
    $\kappa*\nu$, where $\kappa:=v_T$ and $\nu:={V}_{T^c}$. Thus, another
    appeal to Theorem \ref{th:ac:fourier} shows that
    condition (2) there is satisfied [confer
    also with Remark \ref{rem:ac:fourier}], and thence
   it follows that $\mathscr{C}_v(F)= \mathrm{cap}_{\bm{\Psi}}(F)$. This and
    Theorem \ref{th:main} together complete the proof.
\end{proof}

\subsection{A relation to an intersection problem}

Define for all Borel sets $F\subseteq\R^d$,
\begin{equation}\begin{split}
    \mathrm{H}_F &:= \left\{x\in\R^d:\
        \P\left\{ \X \left( (0\,,\infty)^N\right) \cap (x\oplus F)
        \neq\varnothing\right\}>0\right\}\\
    \widetilde{\mathrm{H}}_F &:= \left\{x\in\R^d:\
        \P\left\{ \widetilde{\X} \left( \R_{\ne}^N\right) \cap (x\oplus F)
        \neq\varnothing\right\}>0\right\},
\end{split}
\end{equation}
where $\R_{\ne}^N := \{\bm{t}\in \R^N: \ t_j \ne
0 \ \hbox{ for every }\, j = 1, \ldots, N\}$. Clearly, $\mathrm{H}_F \subseteq
\widetilde{\mathrm{H}}_F$.

The following improves on Proposition 6.2 of \fullocite{KXZ:03}.

\begin{proposition}\label{pr:H_F}
     Consider the following statements:
    \begin{enumerate}
            \item[(0)] $\E[\lambda_d(\X(\R^N_+)\oplus F)]>0$;
        \item[(1)] $\mathrm{cap}_{\bm{\Psi}}(F)>0$;
        \item[(2)] $\mathscr{C}_v(F)>0$;
        \item[(3)] $\lambda_d(\mathrm{H}_F)>0$;
        \item[(4)] $\mathrm{H}_F\neq\varnothing$;
        \item[(5)] $\widetilde{\mathrm{H}}_F=\R^d$; and
        \item[(6)] $\mathrm{H}_F=\R^d$.
    \end{enumerate}
    Then:
    \begin{enumerate}
        \item[(a)] It is always the case that $(0)\Leftrightarrow(1)$.
        \item[(b)] If $\X$ has a one-potential density, then
            $(1)\Leftrightarrow(3)\Leftrightarrow(4)$.
        \item[(c)] If $\X$ has an a.e.-positive
            one-potential density, then
            $(1)\Leftrightarrow(3)\Leftrightarrow(4)\Leftrightarrow(6)$.
        \item[(d)] If there exists a non-empty $T\subseteq\{1\,,\ldots\,, N\}$
            such that the sub-process
            $\widetilde{\X}_T$ has a one-potential density $v_T$, and $v_T$ is
            continuous on $\R^d$ and finite on $\R^d\setminus\{0\}$. Then,
            $(1)\Leftrightarrow(2) \Leftrightarrow (3) \Leftrightarrow (4)$.
        \item[(e)] Suppose the conditions of (e) are met, and let
            $v$ denote the one-potential density of $\widetilde{\X}$. If
            $v>0$ a.e., then
            $(1)\Leftrightarrow(2) \Leftrightarrow (3) \Leftrightarrow (4)
            \Leftrightarrow (5)$.
     \end{enumerate}
\end{proposition}

\begin{remark}
      Proposition
      \ref{pr:H_F} is deceptively subtle. For example,
    The preceding proposition might fail
    to hold if $\mathrm{H}_F$ were replaced by the related
    set
    \begin{equation}
        \mathrm{H}_F^* := \left\{x\in\R^d:\
        \P\left\{ \X \left( \R^N_+\right) \cap (x\oplus F)
        \neq\varnothing\right\}>0\right\}.
    \end{equation}
    Indeed, because $\X(\bm{0})=0$, it follows
    that $\mathrm{H}_F^*=-F$.\qed
\end{remark}

Now we prove Proposition \ref{pr:H_F}.

\begin{proof}[Proof of Proposition \ref{pr:H_F}]
    Part (a) is merely a restatement of Theorem \ref{th:main}.

    We first observe the
    following consequence of the Fubini-Tonelli theorem:
    \begin{equation}\label{eq:fubini2}
        \E\left[\lambda_d\left(\X
        \left( (0\,,\infty)^N\right)\ominus F\right)\right]
        =\int_{\R^d} \P\left\{ \X\left( (0\,,\infty)^N\right)
        \cap(x\oplus F)\neq\varnothing
        \right\}\, dx.
    \end{equation}
    See \eqref{eq:fubini} for a similar computation.
    If $\mathrm{cap}_{\bm{\Psi}}(F)>0$, then the left-hand side of
    \eqref{eq:fubini2} is positive by a simple
    variant of Theorem \ref{th:main}. This proves that
    $(1)\Rightarrow(3)$. Conversely, if $\lambda_d(\mathrm{H}_F)>0$,
    then the right-hand side of \eqref{eq:fubini2}---and hence
    the left-hand side---are positive. Another appeal to
    Theorem \ref{th:main} shows us that $(1)\Leftrightarrow(3)$.

    Now we finish the proof of (b).
    Clearly, (3) implies (4). In order to prove that $(4)\Rightarrow(3)$,
    we define
    \begin{equation}
        Q_{\bm{s}} := (s_1\,,\infty)\times\cdots\times
        (s_N\,,\infty)\qquad\text{for all $\bm{s}\in\R^N$}.
    \end{equation}
    We note that for all $\bm{s}\in\R^N$ and $x\in\R^d$,
    \begin{equation}\label{eq:wt:Qs}\begin{split}
        &\P\left\{ \widetilde{\X}(Q_{\bm{s}})\cap (x\oplus F)\neq
            \varnothing\right\}\\
        &\hskip1.5in = \int_{\R^d}
            \P\left\{ \X\left( (0\,,\infty)^N\right) \cap
            \left((x-y)\oplus F\right)\neq\varnothing\right\}
            \P\left\{\widetilde{\X}(\bm{s})\in dy\right\}.
    \end{split}\end{equation}
    This uses only the fact that $\{\widetilde{\X}(\bm{t})
    -\widetilde{\X}(\bm{s})\}_{\bm{t}\in Q_{\bm{s}}}$
    is independent of $\widetilde{\X}(\bm{s})$, and has the
    same law as $\X$. In particular,
    \begin{equation}\label{eq:4.55}\begin{split}
        &\int_{\R^N} \P\left\{\widetilde{\X}(Q_{\bm{s}})
            \cap(x\oplus F)\neq\varnothing\right\} e^{-[\bm{s}]}\,
            d\bm{s}\\
        &\hskip1.8in=2^N\int_{\R^d} \P\left\{ \X\left((0\,,\infty)^N\right)
            \cap \left( (x-y)\oplus F\right) \neq\varnothing
            \right\} v(y)\, dy,
    \end{split}\end{equation}
    where $v$ denotes the one-potential density of $\widetilde{\X}$.
    [It exists thanks to Remark \ref{rem:pot:density}.]

    If (4) holds, then there exists $\bm{s}\in\R^N_+$
    such that $\P\{{\X}(Q_{\bm{s}})
    \cap(x\oplus F)\neq\varnothing\}>0$.
    Consequently,
    $\P \{\widetilde{\X}(Q_{\bm{s}}) \cap(x\oplus F)\neq\varnothing \}>0$.
    As we decrease $\bm{s}$, coordinatewise, the preceding
    probability increases. Therefore, the left-most
    term in \eqref{eq:4.55} is (strictly) positive, and therefore so is
    the right-most term in \eqref{eq:4.55}. We can conclude that
    $\P \{ \X ((0\,,\infty)^N )
    \cap  ( w\oplus F ) \neq\varnothing
   \}>0$ for all $w$ in a non-null set
    $[\lambda_d]$.
    This proves $(3)\Rightarrow(4)$, whence (b).

    In order to prove (c), let $g$ denote the one-potential
    density of $\X$. Then,
    we argue as we did to prove \eqref{eq:4.55},
    and deduce that
    \begin{equation}\label{eq:4.56}\begin{split}
        &\int_{\R^N_+} \P\left\{ \X(Q_{\bm{s}}) \cap(x\oplus F)\neq
            \varnothing \right\} e^{-[\bm{s}]}\, d\bm{s}\\
        &\hskip2in = \int_{\R^d} \P\left\{
            \X\left((0\,,\infty)^N\right) \cap\left((x-y)\oplus F\right)
            \neq\varnothing\right\} g(y)\, dy.
    \end{split}\end{equation}
    If $\mathrm{H}_F\neq\R^d$, then the left-most term must
    be zero for some $x\in\R^d$. Because $g>0$ a.e., this proves that
    $\lambda_d(\mathrm{H}_F)=0$.

    The equivalence of $(1)\Leftrightarrow(2) $ in  (d) is contained
    in Theorems \ref{th:main} and \ref{th:main:v}. The rest of Part (d)
    is proved similarly as in the proof of (b).

     To prove (e): We note that
     ``$\widetilde{\mathrm{H}}_F=\R^d$ iff (4)'' [under the condition
    that $v>0$ a.e.] has a very similar proof
    to ``$\mathrm{H}_F=\R^d$ iff (4)'' [under the condition
    that $g>0$ a.e.], but uses \eqref{eq:4.55} instead of
    \eqref{eq:4.56}. We omit the details and conclude our proof.
\end{proof}

\section{\bf Proof of Theorem \ref{th:FS}}\label{sec:Pf:cor:FS}

We continue to let $X_1,\ldots,X_N$ denote $N$ independent
L\'evy processes on $\R^d$ with respective exponents
$\Psi_1,\ldots,\Psi_N$.
There is a large literature
that is devoted to the ``$N$-parameter L\'evy process''
$\otimes_{j=1}^N X_j$ defined by
\begin{equation}\label{eq:Npar:PLP}
    \left(\otimes_{j=1}^N X_j\right)(\bm{t})
    := \left( \begin{matrix}
        X_1(t_1)\\
        \vdots\\
        X_N(t_N)
    \end{matrix}\right)\qquad\text{for all $\bm{t}\in\R^N_+$}.
\end{equation}
Note that:
\begin{enumerate}
    \item[(i)] The state space of $\otimes_{j=1}^NX_j$ is $(\R^d)^N$; and
    \item[(ii)] In the special case that
        $N=2$, \eqref{eq:Npar:PLP} reduces to \eqref{eq:2par:PLP}.
\end{enumerate}
In this section we describe how
this theory, and much more, is contained within the
theory of additive L\'evy processes of this paper.

Let us begin by making the observation that
product L\'evy processes are in fact degenerate
additive L\'evy processes. Indeed, for all
$\bm{t}\in\R^N_+$ we can write
\begin{equation}
    \left( \otimes_{j=1}^N X_j \right) (\bm{t})
    = \bm{A}^1 X_1(t_1) + \cdots +
    \bm{A}^N X_N(t_N),
\end{equation}
where each $X_j(t_j)$ is viewed as a column vector
with dimension $d$, and
each $\bm{A}^j$ is the $Nd\times d$ matrix
\begin{equation}
    \bm{A}^j = \left(\begin{matrix}
        \bm{0}_{(j-1)d\times d}\\
        \bm{I}_{d\times d}\\
        \bm{0}_{(N-j)d\times d}
    \end{matrix}\right).
\end{equation}
Here:
(i) $\bm{0}_{(j-1)d\times d}$ is a $(j-1)d\times d$ matrix of zeros; and
(ii) $\bm{I}_{d\times d}$ is the identity matrix with $d^2$ entries.
We emphasize that
the matrix $\bm{I}_{d\times d}$ appears after
$j-1$ square matrices of size
$d^2$ whose entries are all zeroes.

It is also easy to describe the law of the process $\otimes_{j=1}^NX_j$
via the following characteristic-function relation:
\begin{equation}
    \E\left[\exp\left\{ i\sum_{j=1}^N \xi^j\cdot X_j(t_j)
    \right\}\right] = \exp\left(-\sum_{j=1}^N t_j\Psi_j(\xi^j)\right),
\end{equation}
valid for all $\xi^1,\ldots,\xi^N\in\R^d$ and
$t_1,\ldots,t_N\ge 0$. Therefore, the exponent of the $N$-parameter
additive L\'evy process $\otimes_{j=1}^N X_j$ is
\begin{equation}
    \bm{\Psi}\left( \xi^1,\ldots,
    \xi^N\right) = \left(\begin{matrix}
        \Psi_1(\xi^1)\\
        \vdots\\
        \Psi_N(\xi^N)
    \end{matrix}\right)\qquad
    \text{for all }\quad\xi^1,\ldots,
    \xi^N\in\R^d.
\end{equation}

Therefore, Theorem \ref{th:main} applies, without
further thought, to yield the following.

\begin{corollary}\label{co:product:main}
    Choose and fix a Borel set $F\subseteq(\R^d)^N$.
    Then $(\otimes_{j=1}^N X_j)(\R^N_+)\oplus F$
    has positive Lebesgue measure with positive probability
    if and only if there exists a compact-support
    Borel probability measure
    $\mu$ on $F$ such that
    \begin{equation}
        \int_{(\R^d)^N} \left| \hat{\mu}\left( \xi^1,\ldots,
        \xi^N\right) \right|^2 \prod_{j=1}^N
        \Re\left( \frac{1}{1+\Psi_j(\xi^j)} \right)\, d\xi<\infty.
    \end{equation}
\end{corollary}

Next let us suppose
that $X_1,\ldots,X_N$ have one-potential
densities $u_1,\ldots,u_N$, respectively. According to
Definition \ref{def:one:pot},
\begin{equation}\label{eq:1one:pot}
    \E\left[\int_0^\infty f(X_j(s)) e^{-s}\, ds\right]
    =\int_{\R^d} f(a) u_j(a)\, da,
\end{equation}
for all $j=1,\ldots,N$, and all Borel measurable
functions $f:\R^d\to\R_+$. [This agrees with the usual
nomenclature of probabilistic potential theory.]

\begin{lemma}\label{lem:product:one:pot}
    If $X_1,\ldots,X_N$ have respective one-potential
    densities $u_1,\ldots,u_N$, then
    $\otimes_{j=1}^N X_j$ has the one-potential density
    $u( x^1,\ldots,x^N)
    := \prod_{j=1}^N
    u_j (x^j )$ for all $x^1,\ldots,x^N\in\R^d$.
    Also, $\widetilde{\otimes_{j=1}^N X_j}$
    has the one-potential density
    \begin{equation}\label{eq:product:v}
        v \left( x^1,\ldots,x^N \right)
        := \prod_{j=1}^N \left(
        \frac{u_j\left(x^j\right)+u_j\left(-x^j\right)}{2}\right)
        \qquad\text{for all $x^1,\ldots,x^N\in\R^d$}.
    \end{equation}
    Finally, if $u_j(0)>0$ for all $j$, then $u$ and
    $v$ are strictly positive everywhere.
\end{lemma}

\begin{proof}
    Let us prove the first assertion about the one-potential
    density of $\otimes_{j=1}^N X_j$ only. The second assertion
    is proved very similarly.

    We seek to establish that for all Borel measureable
    functions $f:(\R^d)^N\to\R_+$,
    \begin{equation}
        \E\left[ \int_{\R^N_+} f\left(
        \left(\otimes_{j=1}^N X_j \right) (\bm{t})  \right)
        e^{-[\bm{t}]}\,d\bm{t} \right]
        =\int_{(\R^d)^N} f\left( x^1,\ldots,x^N\right)
        \prod_{j=1}^N u_j\left(x^j\right)\, dx.
    \end{equation}
    A density argument reduces the problem to the case
    that $f(x^1,\ldots,x^N)$ has the special form
    $\prod_{j=1}^N f_j(x^j)$, where $f_j:\R^d\to\R_+$
    is Borel measurable. But in this case, the claim
    follows immediately from \eqref{eq:1one:pot}
    and the independence of $X_1,\ldots,X_N$.

    In order to complete the proof, consider the case that
    $u_j(0)>0$. Lemma 3.2 of \ocite{Evans:87b} asserts that
    $u_j(z)>0$ for all $z\in\R^d$, whence follows the lemma.
\end{proof}

The preceding lemma and Proposition \ref{pr:H_F}
together imply,
without any further effort, the following two theorems.

\begin{corollary}\label{co:FS-1}
    Let $X_1,\ldots,X_N$ be independent L\'evy
    processes on $\R^d$, and assume that each $X_j$
    has a one-potential density $u_j$ such that $u_j(0)>0$.
    Then, for all Borel sets $F\subseteq(\R^d)^N$,
    \begin{equation}\label{eq:FS:hit}
        \P\left\{ \left(\otimes_{j=1}^N X_j\right)
        \left( (0\,,\infty)^N \right) \cap F\neq\varnothing\right\}>0
    \end{equation}
    if and only if there exists a compact-support
    Borel probability measure
    $\mu$ on $F$ such that
    \begin{equation}
        \int_{(\R^d)^N} \left| \hat{\mu}\left( \xi^1,\ldots,
        \xi^N\right) \right|^2 \prod_{j=1}^N
        \Re\left( \frac{1}{1+\Psi_j(\xi^j)} \right)\,
        d\xi<\infty.
    \end{equation}
\end{corollary}

\begin{corollary}[Fitzsimmons and Salisbury, 1989]\label{co:FS}
    Suppose, in addition to the hypotheses of Corollory
    \ref{co:FS-1}, that each $u_j:\R^d\to\bar\R_+$
    is continuous on $\R^d$, and finite on $\R^d\setminus\{0\}$.
    Then, for all Borel sets $F\subseteq(\R^d)^N$,
    \eqref{eq:FS:hit} holds
    if and only if there exists a compact-support
    Borel probability measure
    $\mu$ on $F$ such that
    \begin{equation}
        \iint \prod_{j=1}^N \left(
        \frac{u_j\left(x^j-y^j\right)
        +u_j\left(y^j-x^j\right)}{2}\right)\,
        \mu\left(dx^1\cdots dx^N\right)\,
        \mu\left(dy^1\cdots dy^N\right) <\infty.
    \end{equation}
\end{corollary}

We can now prove Theorem \ref{th:FS} of
the Introduction.

\begin{proof}[Proof of Theorem \ref{th:FS}]
    We observe that:
    \begin{enumerate}
        \item[(i)] $\cap_{j=1}^N X_j((0\,,\infty))$ can intersect
            a set $F\subseteq\R^d$ if and only if
            $(\otimes_{j=1}^N X_j)((0\,,\infty)^N)$ can intersect
            $G:=\{x\otimes\cdots\otimes x:\,x\in F\}\subseteq(\R^d)^N$; and
        \item[(ii)] any compact-support Borel probability
            measure $\sigma $ on $G$ manifestly has the form
            $\sigma(dx^1\cdots\,dx^N) = \mu(dx^1) \prod_{j=2}^N
            \delta_{x^1}(dx^j)$,
            where
            $\mu$ is a compact-support Borel probability measure
            on $F$, and hence
            $\hat{\sigma} ( \xi^1,\ldots,\xi^N  )$ is
            equal to $\hat{\mu} (\xi^1+\cdots+\xi^N  )$
            for all $\xi^1,\ldots,\xi^N\in\R^d$.
    \end{enumerate}
    Therefore, Theorem \ref{th:FS} is a ready consequence of
    Corollaries \ref{co:FS-1} and \ref{co:FS}.
\end{proof}

\section{\bf Proof of Theorem \ref{th:bertoin}}

First, we elaborate on the connection between
Theorem \ref{th:bertoin} and Bertoin's conjecture
\cite{bertoin:st-flour}*{p.\ 49}
that was mentioned briefly in the Introduction.
Recall that a real-valued L\'evy process
is a \emph{subordinator} if its sample functions
are monotone a.s.\
\cites{bertoin:book,bertoin:st-flour,fristedt,sato}.

\begin{remark}\label{rem:bertoin}
    We consider the special case that $S_1$ and $S_2$
    are two (increasing) subordinators on $\R_+$
    and $F:=\{0\}$, and define
    two independent L\'evy processes by $X_1:=S_1$
    and $X_2:=-S_2$. Evidently,
    $X_1(t_1)+X_2(t_2)=0$ for some $t_1,t_2>0$ if and only if
    $S_1(t_1)=S_2(t_2)$ for some $t_1,t_2>0$.

    Let $\Sigma_j$ denote the one-potential measure of $S_j$,
    and suppose $\Sigma_1(dx)=u_1(x)\, dx$, where $u_1$ is
    continuous on $\R$ and strictly positive on $(0\,,\infty)$.
    Let $U_j$ denote the one-potential measure of $X_j$. Then,
    $U_1=\Sigma_1$ and $U_2=\breve{\Sigma}_2$, which we recall
    is the same as $\Sigma_2(-\bullet)$. It follows then that
    $U_1(dx)=u_1(x)\, dx$, and $(u_1*U_2)(x)=\int_0^\infty
    u_1(x+y)\,\Sigma_2(dy)$ is strictly positive a.e.
    Therefore, 
    we may apply Theorem \ref{th:bertoin}
    with $F:=\{0\}$ to deduce that
    \begin{equation}\label{eq:S_1:S_2}
        \P\left\{ S_1(t_1)=S_2(t_2)\quad
        \text{for some $t_1,t_2>0$}\right\}>0
        \quad\Leftrightarrow\quad
        Q(0)<\infty.
    \end{equation}

      Because $U_2$ does not charge $(-\infty\,,0)$ in
      the present setting,
        \begin{equation}
            Q(0) := \int_0^\infty \left[ \frac{u_1(y)+
            u_1(-y)}{2}\right]\, U_2(dy)=
            \frac12\int_0^\infty u_1(y)\, U_2(dy),
      \end{equation}
    since $u_1(y)=0$ for all $y<0$. [In fact,
    Lemma 3.2 of \ocite{Evans:87b} tells us
    that $u_1(y)=0$ for all $y\le 0$.]
    Therefore, we conclude from
    \eqref{eq:S_1:S_2} that
    \begin{equation}\label{eq:bertoin:conj0}
        \P\left\{ S_1(t_1)=S_2(t_2)\quad
        \text{for some $t_1,t_2>0$}\right\} = 0
        \quad\Leftrightarrow\quad
        \int_0^\infty u_1(y)\, U_2(dy)=\infty.
    \end{equation}

    Define the zero-potential measures $U^0_j$ as
    \begin{equation}
    	U^0_j(A) := \E\left[\int_0^\infty \1_A(S_j(t))\, dt\right],
    \end{equation}
    for $j=1,2$ and Borel sets $A\subseteq\R_+$. Suppose $U^0_1
    (dx)=u^0(x)\, dx$, where $u^0$ is positive and continuous
    on $(0\,,\infty)$. Then, it is possible to adapt our methods,
    without any difficulties, to deduce also that
    \begin{equation}\label{eq:bertoin:conj}
        \P\left\{ S_1(t_1)=S_2(t_2)\quad
        \text{for some $t_1,t_2>0$}\right\} = 0
        \quad\Leftrightarrow\quad
        \int_0^\infty u^0_1(y)\, U^0_2(dy)=\infty.
    \end{equation}
    The end of the proof of Theorem \ref{th:dimH} contains
    a discussion which describes how similar changes can be
    made to adapt the proofs from statements
    about one-potentials $U_j$ to those about zero-potentials
    $U^0_j$. We omit the details, as they are not enlightening.

    We mention the adaptations to zero-potentials of \eqref{eq:bertoin:conj}
    for historical interest: \eqref{eq:bertoin:conj}
    was conjectured by Bertoin, under precisely the
    stated conditions of this remark. Bertoin's conjecture
    was motivated in part by the fact
    \cite{bertoin:1999} that, under the very
    same conditions as above,
    \begin{equation}\label{eq:bertoin:99}
        \P\left\{ S_1(t_1)=S_2(t_2)\quad
        \text{for some $t_1,t_2>0$}\right\}=0
        \quad\Leftrightarrow\quad
        \sup_{z\in\R}\int_0^\infty u^0_1(y+z)\, U^0_2(dy)=\infty.
    \end{equation}

    It is possible to deduce \eqref{eq:bertoin:99}, and the same
    statement without the zero superscripts, from
    the present harmonic-analytic
    methods as well; see Lemma \ref{lem:sup:zero}. The
    said extension goes well beyond the theory of subordinators,
    and is a sort of
    ``low intensity maximum principle'' \cite{Salisbury:92}.
    But we will not describe
    the details further, since we find the forms
    of \eqref{eq:bertoin:conj0} and
    \eqref{eq:bertoin:conj} simpler to use, as well as
    easier to conceptualize.
    \qed
\end{remark}

Next we prove Theorem \ref{th:bertoin} without further ado.

\begin{proof}[Proof of Theorem \ref{th:bertoin}]
     A direct computation reveals that the one-potential density
     of the two-parameter additive L\'evy process
     $\X:=X_1\oplus X_2$ is $u_1*U_2$. Therefore,
    the equivalence of \eqref{Eq:Inverse} and
    \eqref{Eq:Inverse2} follows from Part (c) in
    Proposition \ref{pr:H_F}. Next we note that
    the one-potential density of $\widetilde{X}_1$
    is described by
    \begin{equation}
        v_1(x) := \frac{u_1(x)+u_1(-x)}{2}\qquad
        \text{for all $x\in\R^d$}.
    \end{equation}
    This function is
    positive a.e.\ and continuous away from the origin. We can use
    \eqref{eq:v:vs:vT} and verify directly that $\widetilde{\X}$  has a
    one-potential density $Q$ given by \eqref{Def:Q}. Hence the
    last statement follows from Part (e) of Proposition \ref{pr:H_F}.
\end{proof}

\section{\bf Intersections of L\'evy processes}

The goal of this section is to prove Theorem \ref{th:dimH}. We first
return briefly to Theorem \ref{th:FS} and discuss how it implies a necessary
and sufficient condition for the existence of path-intersections for $N$
independent L\'evy processes. After proving Theorem \ref{th:dimH}, we
conclude this section by presenting a nontrivial, though simple, example.

\subsection{Existence of intersections}

First we develop some general results on equilibrium measure that
we believe might be of independent interest.

Choose and fix a compact set $F\subseteq \R^d$, and recall
that $\mathrm{cap}_{\bm{\Psi}}(F)$ is the reciprocal of
$\inf I_{\bm{\Psi}}(\mu)$, where the infimum is taken
over all  Borel probability measures $\mu$ on $F$.
It is not hard to see that when $\mathrm{cap}_{\bm{\Psi}}(F)>0$,
this infimum is in fact
achieved. Indeed, for all $\e>0$ we can find
a Borel probability measure $\mu_\e$ on $F$ such that
$I_{\bm{\Psi}}(\mu_\e) \le (1+\e)/\mathrm{cap}_{\bm{\Psi}}(F)$.
We can extract any subsequential weak limit $\mu$
of $\mu_\e$'s. Evidently, $\mu$ is a Borel probability
measure on $F$, and $I_{\bm{\Psi}}(\mu)\le1/\mathrm{cap}_{\bm{\Psi}}(F)$
by Fatou's lemma. Because of the defining property of capacity,
$I_{\bm{\Psi}}(\mu)\ge1/\mathrm{cap}_{\bm{\Psi}}(F)$ also. Therefore,
$\mathrm{cap}_{\bm{\Psi}}(F)$ is in fact the reciprocal of
the energy of $\mu$.

Any Borel probability measure $\mu$ on $F$ that has the preceding
property is called an \emph{equilibrium measure} on $F$.
We now prove that there is only one equilibrium measure
on a compact $F$.

\begin{proposition}\label{pr:one:equil}
    If $F\subseteq \R^d$ is compact and has positive
    capacity $\mathrm{cap}_{\bm{\Psi}}(F)$, then there exists
    a unique Borel probability measure $e_F$ on $F$
    such that $\mathrm{cap}_{\bm{\Psi}}(F) =1/I_{\bm{\Psi}}(e_F)$.
\end{proposition}

\begin{proof}
    For all finite signed probability measures $\mu$ and $\nu$
    on $\R^d$ define
    \begin{equation}
        I_{\bm{\Psi}}(\mu\,,\nu) := \frac{1}{(2\pi)^d}
        \int_{\R^d} \left(\frac{\overline{\hat{\mu}(\xi)}\,
        \hat{\nu}(\xi)+\overline{\hat{\nu}(\xi)}\,
        \hat{\mu}(\xi)}{2}\right)\, K_{\bm{\Psi}}(\xi)\, d\xi.
    \end{equation}
    This is well-defined, for example, if
    $I_{\bm{\Psi}}(|\mu|)+I_{\bm{\Psi}}(|\nu|)<\infty$, where $|\mu|$
    is the total variation of $\mu$. Indeed, we
    have the following Cauchy--Schwarz inequality:
    $| I_{\bm{\Psi}}(\mu\,,\nu) |^2 \le
    I_{\bm{\Psi}}(|\mu|)\cdot I_{\bm{\Psi}}(|\nu|)$.
    If $\sigma $ is a finite [nonnegative] Borel measure on $\R^d$,
    then $I_{\bm{\Psi}}(\sigma\,, \sigma)$ agrees with
    $I_{\bm{\Psi}}(\sigma)$, and this is positive as long as $\sigma $
    is not the zero measure. However, we may note that slightly
    more general fact that if $\sigma$ is a non-zero
    finite \emph{signed} measure, then  we still have
    \begin{equation}\label{eq:EM:pos}
        I_{\bm{\Psi}}(\sigma\,,\sigma)>0.
    \end{equation}
    This follows from the fact that
    $K_{\bm{\Psi}}(\xi)>0$ for all
    $\xi\in\R^d$.

    Let $c:=1/\mathrm{cap}_{\bm{\Psi}}(F)$,
    and suppose $\mu$ and $\nu$
    were two distinct equilibrium measures on $F$.
    That is, $\mu\neq\nu$ but
    $I_{\bm{\Psi}}(\mu)= I_{\bm{\Psi}}(\nu)=1/c$.
    In accord with \eqref{eq:EM:pos},
    \begin{equation}
        0 < I_{\bm{\Psi}}\left( \frac{\mu-\nu}{2}
            ~,~\frac{\mu-\nu}{2}\right)
        = \frac{c-I_{\bm{\Psi}}(\mu\,,\nu)}{2}.
    \end{equation}
    Consequently, $I_{\bm{\Psi}}(\mu\,,\nu)<c$,
    and hence
    \begin{equation}\begin{split}
        I_{\bm{\Psi}}\left( \frac{\mu+\nu}{2}\right)
            &= \frac14 I_{\bm{\Psi}}(\mu) +
            \frac14 I_{\bm{\Psi}}(\nu) +
            \frac12 I_{\bm{\Psi}}(\mu\,,\nu)\\
        &= \frac{c+I_{\bm{\Psi}}(\mu\,,\nu)}{2}\\
        &<c.
    \end{split}\end{equation}
    Because $\frac12(\mu+\nu)$ is a Borel probability measure
    on $F$, this is contradicts the fact that $c$ is
    the smallest possible energy on $F$.
\end{proof}

Next we note with the following computation of the equilibrium
measure in a specific class of examples. The following result is
related quite closely to the celebrated local ergodic theorem of
\ocite{Csiszar}. [We hope to elaborate on this connection
elsewhere.]
See also Proposition A3 of
\fullocite{KXZ:03} for a related result.

\begin{proposition}\label{pr:flat}
    Suppose $F$ is a fixed compact subset of
    $\R^d$ with a nonvoid interior. If
    $\mathrm{cap}_{\bm{\Psi}}(F)>0$, then
    $e_F$ is the normalized Lebesgue measure on $F$.
\end{proposition}

\begin{proof}
    Our strategy is to prove that $e_F$ is translation invariant.

    Let $\mathcal{R}(r)$ denote the collection of all
    closed ``upright'' \emph{cubes} of the form
    \begin{equation}
        I:=[s_1\,,s_1+r]\times\cdots\times[s_d\,,s_d+r]
        \subseteq F,
    \end{equation}
    such that all of the $s_i$'s
    are rational numbers and $r>0$.
    Each $\mathscr{R}(r)$ is
    a countable collection, and hence we can (and will)
    enumerate its elements as $I_1(r),I_2(r),\ldots\,.$

    For all $i,j\ge 1$ and $r>0$
    we choose and fix a one-to-one
    onto piecewise-linear map $\theta_{i,j,r}:F\to F$
    that has the following
    properties:
    \begin{itemize}
        \item If $a\not\in I_i(r)\cup J_j(r)$, then
            $\theta_{i,j,r}(a)=a$;
        \item $\theta_{i,j,r}$ maps $I_i(r)$ onto $J_i(r)$
            bijectively; and
        \item $\theta_{i,j,r}$ maps $J_i(r)$ onto $I_i(r)$ bijectively.
    \end{itemize}
    To be concrete,
    let us write $I_i(r)= I_j(r)+b$, where
    $b\in\R^d$. Then we define
    \begin{equation}
        \theta_{i,j,r}(a) =\begin{cases}
            a-b&\text{if $a\in I_i(r)$},\\
            a +b&\text{if $a\in I_j(r)$},\\
            a&\text{if $a\not\in I_i(r)\cup J_i(r)$}.
        \end{cases}
    \end{equation}
It can be verified that $\theta_{i,j,r} \circ \theta_{i,j,r} = id$,
the identity map.

    For all integers $n\ge 1$ and $r>0$ consider
    \begin{equation}
        \rho_{n,r} := \frac{1}{n^2} \mathop{\sum\sum}_{1\le i,j\le n}
        \left( e_F\circ \theta_{i,j,r}^{-1} \right).
    \end{equation}
    Obviously, each $\rho_{n,r}$ is a probability measure on
    $F$, and
    \begin{equation}\label{eq:flat}
        \rho_{n,r}(I_i(r))=\rho_{n,r}
        (I_j(r))\qquad\text{for all $1\le i,j\le n$}.
    \end{equation}
    Furthermore, a direct computation reveals that for all $i,j\ge 1$
    and $r>0$,
    \begin{equation}
        \left| \widehat{e_F\circ\theta_{i,j,r}^{-1}} \right|
        =\left| \widehat{e_F}\right|\qquad\text{pointwise}.
    \end{equation}
    Therefore, we can apply Minkowski's inequality, to the
    norm $\mu\mapsto \sqrt{I_{\bm{\Psi}}(\mu)}$, to find that
    \begin{equation}\begin{split}
        \sqrt{I_{\bm{\Psi}}(\rho_{n,r})} &\le \frac{1}{n^2}
            \mathop{\sum\sum}_{1\le i,j\le n}
            \sqrt{I_{\bm{\Psi}}\left( e_F\circ \theta_{i,j,r}^{-1} \right)}\\
        &=\sqrt{I_{\bm{\Psi}}(e_F)}.
    \end{split}\end{equation}
    By the uniqueness of equilibrium measure
    (Proposition \ref{pr:one:equil}), $e_F=\rho_{n,r}$
    for all $n\ge 1$ and rationals $r>0$. This and \eqref{eq:flat}
    together prove that $e_F(I)=e_F(J)$
    for all $r>0$
    and all $I,J\in \mathscr{R}(r)$. A monotone-class
    argument reveals that for all
    Borel sets $I\subset F$ and $b\in\R^d$,
    $e_F(I)=e_F(b+I)$, provided that $I,b+I\subseteq F$.
    Because the Lebesgue measure is characterized by
    its translation invariance, this implies the proposition.
\end{proof}

Lemma \ref{lem:ac:fourier}
and Proposition \ref{pr:flat}, and Theorem
\ref{th:FS} together imply the following
variant of a theorem of \ocite{FitzsimmonsSalisbury}.

\begin{corollary}\label{cor:FS--1}
    Let $X_1,\ldots,X_N$ be independent L\'evy
    processes on $\R^d$, and assume that each $X_j$
    has a one-potential density $u_j$ that is
    continuous on $\R^d$, positive at zero,
    and finite on $\R^d\setminus\{0\}$. Then,
    \begin{equation}\label{eq:FS--1}
        \P\left\{ X_1(t_1)=\cdots= X_N(t_N)
        \text{ for some $t_1,\ldots,t_N>0$}\right\}>0
    \end{equation}
    if and only if
    \begin{equation}
        \prod_{j=1}^N \left(
        \frac{u_j (\bullet )
        +u_j (-\bullet)}{2}\right)
        \in L^1_{\text{\it loc}}(\R^d).
    \end{equation}
\end{corollary}

\begin{proof}
    Clearly, \eqref{eq:FS--1} holds if
    and only if there exists $n>0$ such that
    \begin{equation}\label{eq:FS--2}
        \P\left\{ X_1(t_1)=\cdots= X_N(t_N)\in[-n\,,n]^d
        \text{ for some $t_1,\ldots,t_N>0$}\right\}>0.
    \end{equation}
    Theorem \ref{th:FS} implies that \eqref{eq:FS--2}
    holds if and only if
    $\mathrm{cap}_{\bm{\Psi}} \big([-n\,, n]^d\big) > 0$.
    Lemma \ref{lem:ac:fourier} and Proposition \ref{pr:flat}
    together prove the result.
\end{proof}

\subsection{Proof of Theorem \ref{th:dimH}}

We begin by proving the first part; thus, we assume
only that the $u_j$'s exist and are a.e.-positive.

Let $\mathfrak{S}$ denote an independent $M$-parameter
additive stable process of index $\alpha\in(0\,,2)$;
see \eqref{def:frakS}. Next, we consider the $(N+M)$-parameter
process $\mathfrak{Y}:= \otimes_{j=1}^N (X_j-\mathfrak{S})$; i.e.,
\begin{equation}
    \mathfrak{Y}(\bm{s}\otimes\bm{t}) := \left(\begin{matrix}
        X_1(s_1)-\mathfrak{S}(\bm{t})\\
        \vdots\\
        X_N(s_N)-\mathfrak{S}(\bm{t})
    \end{matrix}
    \right)\quad\text{for all}\quad\bm{s}\in\R^N_+,\, \bm{t}\in\R^M_+.
\end{equation}
It is not hard to adapt the discussion of the first few paragraphs
in \S\ref{sec:Pf:cor:FS} to the present situation and deduce that
$\mathfrak{Y}$ is an $(N+M)$-parameter
additive L\'evy process, with values in $(\R^d)^N$, and that
for all $\bm{s}\in\R^N_+$, $\bm{t}\in\R^M_+$, and
$\xi:=\xi^1\otimes\cdots\otimes\xi^N\in(\R^d)^N$,
\begin{equation}
    \E\exp\left( i\xi\cdot \mathfrak{Y}(\bm{s}\otimes
    \bm{t}) \right) = \exp\left(-\sum_{k=1}^N s_k\Psi_k(\xi^k)
    -\sum_{l=1}^Mt_l\left\|\xi^1+\cdots+\xi^N\right\|^\alpha    \right).
\end{equation}
We can conclude readily from this that the
characteristic exponent of $\mathfrak{Y}$ is
defined by
\begin{equation}
    \bm{\Theta}(\xi) := \left( \Psi_1(\xi^1)\,,\ldots,
    \Psi_N(\xi^N)\,,\underbrace{\left\|
    \sum_{j=1}^N\xi^j\right\|^\alpha,\cdots\,,\left\|
    \sum_{j=1}^N\xi^j\right\|^\alpha}_{\text{$M$ times}}\right),
\end{equation}
for all $\xi:=\xi^1\otimes\cdots\otimes\xi^N\in(\R^d)^N$.

It follows readily from this and Lemma \ref{lem:product:one:pot}
that $\mathfrak{Y}$ and $\widetilde{\mathfrak{Y}}$
both have positive one-potential densities. Moreover, a
direct computation involving the inversion formula reveals that the
potential density of $\widetilde{\mathfrak{Y}}$ is defined by
\begin{equation}\label{eq:cor1.5:1}
    v(x^1\otimes\cdots\otimes x^N) = \int_{\R^d}
    \prod_{j=1}^N \left(
    \frac{u_j(x^j-y)+u_j(y-x^j)}{2}\right)
    w(y)\, dy,
\end{equation}
for all $x:=(x^1\otimes\cdots\otimes x^N)\in(\R^d)^N$.
Here, $w$ denotes the one-potential density of $\mathfrak{S}$. That is,
\begin{equation}
    w(y) := \int_{\R^M_+} p_{\bm t}(y) e^{-[\bm{t}]}\, d\bm{t},
\end{equation}
where $p_{\bm t}$ denotes the density of $\mathfrak{S}(\bm{t})$.
That is,
$p_{\bm t}(y) := (2\pi)^{-d}\int_{\R^d}
\exp(-iy\cdot z - [\bm{t}]\|z\|^\alpha)\, dz$
for all $\bm{t}\in\R^M_+$ and $y\in\R^d$.

Since the $u_j$'s are everywhere positive (Lemma \ref{lem:product:one:pot}),
Proposition \ref{pr:H_F} tell us that
$0\in\mathfrak{Y}(\R^N_+\times\R^M_+)$ with positive probability
if and only if $\text{cap}_{\bm{\Theta}}(\{0\})>0$. Thus, the
preceding positive capacity condition is equivalent to the
integrability of the function $K_{\bm{\Theta}}$. That is,
\begin{equation}\label{cond:cap1}
\begin{split}
    &0\in \mathfrak{Y}(\R^N_+\times\R^M_+)\quad\text{with positive probability }\ \
    \\
    &\qquad \Leftrightarrow \quad \int_{(\R^d)^N} \prod_{j=1}^N\Re\left( \frac{1}{1+\Psi_j(\xi^j)}\right)
    \, \frac{ d\xi}{1+\|\xi^1+\cdots+\xi^N\|^{\alpha M}}<\infty.
    \end{split}
\end{equation}
On the other hand, it is manifestly the case that
$\mathfrak{Y}(\R^N_+\times\R^M_+)$ contains the origin if
and only if the intersection of
$\cap_{j=1}^N X_j(\R_+)$ and $\mathfrak{S}(\R^M_+)$ is nonempty.
Thanks to Theorem 4.4.1 of \ocite{Kh:book}*{p.\ 428},
for all Borel sets $F\subseteq \R^d$,
$\P\{\mathfrak{S}(\R^M_+)\cap F\neq\varnothing\}>0$
if and only $\mathcal{C}_{d-M\alpha}(F)>0$,
provided that we also assume that $d>M\alpha$.
Suppose, then, that $d>M\alpha$. We can apply
the preceding, conditionally on $X_1,\ldots,X_N$, and deduce that
\begin{equation}\label{last1}\begin{split}
    &0\in \mathfrak{Y}(\R^N_+\times\R^M_+)\quad
    	\text{with positive probability }\ \\
    & \qquad \Leftrightarrow \quad
    	\mathcal{C}_{d-M\alpha}\left( \bigcap_{j=1}^N X_j(\R_+)\right)>0
    	\quad\text{with positive probability}.
\end{split}\end{equation}
We compare the preceding display to \eqref{cond:cap1},
and choose $M$ and $\alpha\in(0\,,(d/M)\wedge 2)$, such that
$d-M\alpha$ is any predescribed number $s\in(0\,,d)$.
This yields the following: For all $s\in(0\,,d)$,
\begin{equation}\label{last2}\begin{split}
    &\mathcal{C}_s\left( \bigcap_{j=1}^N X_j(\R_+)\right)>0
        \quad\text{with positive probability }\ \ \\
    &\qquad \Leftrightarrow \quad
    \int_{(\R^d)^N} \prod_{j=1}^N\Re\left( \frac{1}{1+\Psi_j(\xi^j)}\right)
        \, \frac{ d\xi}{1+\|\xi^1+\cdots+\xi^N\|^{d-s}}<\infty.
\end{split}\end{equation}
It follows from \eqref{last2}, Frostman's theorem \cite{Kh:book}*{p.\ 521} and
an argument similar to the proof of Theorem 3.2 in Khoshnevisan, Shieh
and Xiao (2007) that the first identity of \eqref{co:dimH:1} holds almost
surely on $\big\{\cap_{j=1}^N X_j(\R_+) \ne \varnothing\big\}$.

In order to obtain \eqref{co:dimH:2}, we assume also
that the $u_j$'s are continuous on $\R^d$ and finite on $\R^d\setminus
\{0\}$. Thanks to Proposition \ref{pr:H_F},
$\mathfrak{Y}(\R^N_+\times\R^M_+)$ contains zero with positive probability
if and only if
$v(0)<\infty$. Thus, \eqref{eq:cor1.5:1} and \eqref{last1} imply that
\begin{equation}\label{goal:UUU}\begin{split}
    &\mathcal{C}_{d-M\alpha}\left( \bigcap_{j=1}^N X_j(\R_+)\right)>0
        \ \ \text{ with positive probability } \\
    &\qquad \Leftrightarrow \quad \int_{\R^d}
        \prod_{j=1}^N \left(
        \frac{u_j(y)+u_j(-y)}{2}\right)
        w(y)\, dy<\infty.
\end{split}\end{equation}
If we could replace $w(y)$ by $\|y\|^{-d+\alpha M}$, then we
could finish the proof by choosing $M$ and $\alpha$ suitably,
and then appealing to the Frostman theorem and the argument
in Khoshnevisan, Shieh and Xiao (2007). [This is how
we completed the proof of the first part of the proof as
well.] Thus, our goal is to derive
\eqref{goal:UUU}.

Unfortunately, it is not possible to simply replace
$w(y)$ by $\|y\|^{-d+\alpha M}$ by simple real-variable
arguments.
Nonetheless, we recall the following
fact from \ocite{Kh:book}*{Exercise 4.1.4, p.\ 423}:
There exist
$c,C\in(0\,,\infty)$ such that
$c\|y\|^{-d+M\alpha}\le w(y) \le
C\|y\|^{-d+M\alpha}$
for all $y\in(-1\,,1)^d$. Moreover, the upper bound
holds for all $y\in\R^d$ ({\it loc.\ cit.},
Eq.\ (2), p.\ 423);
the lower bound [provably] does not.
It follows then that
\begin{equation}\begin{split}
    &\int_{\R^d}
        \prod_{j=1}^N \left(
        \frac{u_j(y)+u_j(-y)}{2}\right)
        \, \frac{dy}{\|y\|^{d-M\alpha}}<\infty\\
    &\qquad\Rightarrow\quad \mathcal{C}_{d-M\alpha}\left( \bigcap_{j=1}^N X_j(\R_+)\right)>0
        \ \ \ \hbox{with positive probability}\\
    & \qquad\Rightarrow\quad
        \int_{(-1,1)^d}
        \prod_{j=1}^N \left(
        \frac{u_j(y)+u_j(-y)}{2}\right)
        \, \frac{dy}{\|y\|^{d-M\alpha}}<\infty.
\end{split}\end{equation}
If we could replace $(-1\,,1)^d$ by $\R^d$ in the last
display, then our proof follows the outline mentioned earlier.
Thus, we merely point out how to derive \eqref{goal:UUU},
with $w(y)$ replaced by $\|y\|^{-d+M\alpha}$,
and omit the remainder of the argument.

The proof hinges on a few modifications to the entire
theory outlined here. We describe them [very] briefly,
since it is easy---though tedious---to check that the present changes go
through unhindered.

Define the operator $R^{1\otimes 0}$ via
\begin{equation}
    \left( R^{1\otimes 0} f\right)(x) :=
    \frac{1}{2^N} \E\left[\int_{\R^N\times\R^M}
    f\left( x+ \widetilde{\mathfrak{Y}}(\bm{s}\otimes\bm{t})
    \right) e^{-[\bm{s}]}\, d\bm{s}\, d\bm{t}\right]
    \quad\text{for all $x\in(\R^d)^N$}.
\end{equation}
If, we replace $\exp(-[\bm{s}])$ by $\exp(-[\bm{s}]-[
\bm{t}])$, then $R^{1\otimes 0} f$ turns into $Rf$.

As is, the operator $R^{1\otimes 0} f$ fails to map $L^p((\R^d)^N)$ into
$L^p((\R^d)^N)$. But this is a minor technical nuisance,
since one can check directly that
\begin{equation}
    \left( R^{1\otimes 0} f\right)(x) = \int_{\R^d} f(x+y)
    v^{1\otimes 0}(y)\, dy,
\end{equation}
where $v^{1\otimes 0}$ is defined exactly as $v$ was, but with
$w(y)$ replaced by a certain constant times $\|y\|^{-d+M\alpha}$.
And this shows fairly readily that if $f$ is a
compactly-supported function in $L^1((\R^d)^N)$, then
$R^{1\otimes 0} f\in L^1_{\text{\it loc}}((\R^d)^N)$. To complete
our proof, we need to redevelop the potential theory of
the additive L\'evy process $\mathfrak{Y}$, but this time in terms
of $R^{1\otimes 0} f$ and $v^{1\otimes 0}$.

The fact that $R^{1\otimes 0}:L^1_c((\R^d)^N)
\to L^1_{\text{\it loc}}((\R^d)^N)$ provides sufficient
regularity to allow us to push the
Fourier analysis of the present paper through without change.
And the end result is that
$\mathfrak{Y}(\R^N_+\times\R^M_+)$ contains zero if and only
if $v^{1\otimes 0}(0)$ is finite. Now we can complete the proof, but
with $v^{1\otimes 0}$ in place of $v$ everywhere.
\qed

\subsection{An example}

Suppose $X_1,\ldots,X_N$ are independent isotropic stable
processes in $\R^d$ with respective Fourier transforms
$\E\exp(i\xi\cdot X_j(t)) = \exp (-c_jt\|\xi\|^{\alpha_j} )$
for all $t\ge 0$, $\xi\in\R^d$, and $1\le j\le N$.
Here, $c_1,\ldots,c_N$ are constants, and $0< \alpha_1,\ldots,
\alpha_N< 2\wedge d$ are the indices of stability. We are primarily
interested in the case that $N\ge 2$, but the following remarks
apply to the case $N=1$ equally well.

In this section we work out some of intersection properties of
$X_1,\ldots,X_N$. It is possible to construct much
more sophisticated examples. We study the present setting because
it provides us with the simplest nontrivial example of its type.

It is known that each $X_j$ has a continuous positive one-potential
density $u_j$, and there exist $c,C\in(0\,,\infty)$
such that $c\|x\|^{-d+\alpha_j}\le
u_j(x) \le C\|x\|^{-d+\alpha_j}$
for all $x\in(-1\,,1)^d$ and $1\le j\le N$.
[These assertions follow, for example, from Corollary
3.2.1 on  page 379, and Lemma 3.4.1 on  page 383 of
\ocite{Kh:book}.]
Consequently, Theorem \ref{th:FS} immediately implies that for all
Borel sets $F\subseteq\R^d$,
\begin{equation}
    \P\left\{ \bigcap_{j=1}^N X_j(\R_+)\cap F\neq\varnothing\right\}
    >0 \quad\Leftrightarrow\quad
    \mathcal{C}_{Nd-\sum_{j=1}^N\alpha_j}(F)>0.
\end{equation}
In the case that $N=2$, a slightly more general form of this
was found in \ocite{Kh:book}*{Theorem 4.4.1, p.\ 428} by
using other methods.
We may apply the preceding with $F=\R^d$, and appeal
to Taylor's theorem (\emph{loc.\ cit.}, Corollary 2.3.1, p.\ 525)
to find that
\begin{equation}
\begin{split}
    \P\left\{ \bigcap_{k=1}^N X_k(\R_+)\neq\varnothing  \right\}
    >0 \quad\Leftrightarrow\quad (N-1)d<\sum_{j=1}^N \alpha_j.
    \end{split}
\end{equation}
Theorem \ref{th:dimH}, and a direct computation in polar coordinates,
together show that the slack in the preceding inequality determines
the Hausdorff dimension of the set $\cap_{k=1}^NX_k(\R_+)$ of
intersection points. That is, almost surely on
$\big\{\cap_{k=1}^N X_k(\R_+) \ne \varnothing\big\}$,
\begin{equation}
    \dimh \bigcap_{k=1}^N X_k(\R_+) = \left[\sum_{k=1}^N\alpha_k-(N-1)d
    \right]_+.
\end{equation}
This formula continues to hold in case
some, or even all, of the $\alpha_j$'s are
equal to or exceed $d$. We omit the details.

\subsection{Remarks on multiple points}

Next we mention how the preceding
fits in together with the well-known conjecture of
\ocite{HendricksTaylor} that was solved in Fitzsimmons
and Salisbury \ycite{FitzsimmonsSalisbury}, and also
make a few related remarks.

\begin{remark}\label{rem:FS}
    Recall that a [single] L\'evy process $X$ with values in
    $\R^d$ has $N$-multiple points if and only if
    there exist times $0<t_1<\ldots<t_N<\infty$ such that
    $X(t_1)=X(t_2)=\cdots=X(t_N)$. [We are ruling out
    the possibility that $t_1=0$ merely to avoid degeneracies.]

    By localization and
    the Markov property, $X$ has $N$-multiple points almost surely
    if and only if there are times $t_1,\ldots,t_N\in(0\,,\infty)$
    such that $X_1(t_1)=\cdots=X_N(t_N)$ with positive probability,
    where $X_1,\ldots,X_N$ are $N$ i.i.d.\ copies of $X$.

    Suppose that $X$ has a
    a continuous and positive [equivalently, positive-at-zero]
    one-potential density $u$ that is finite on $\R^d\setminus\{0\}$.
    Then according to Corollary \ref{cor:FS--1}, $X$ has $N$-multiple points
    if and only if $u(\bullet)+u(-\bullet)\in L^N_{\text{\it
    loc}}(\R^d)$.
    An application of H\"older's inequality reveals then that
    $X$ has $N$-multiple points
    if and only if $u\in L^N_{\text{\it loc}}(\R^d)$. This is
    more or less the well-known condition of Hendricks and Taylor
    \ycite{HendricksTaylor}.

    More generally, the following can be deduced with no extra
    effort: Under the preceding conditions, given a nonrandom
    Borel set $F\subseteq\R^d$,
    \begin{equation}
        \P\left\{ \text{there exist $0<t_1<\cdots<t_N$ such that }
        X(t_1)=\cdots=X(t_N)\in F\right\}>0
    \end{equation}
        if and only if there exists a compact-support Borel probability
        measure $\mu$ on $F$ such that
    \begin{equation}
        \iint \big[ u(x-y) \big]^N\,
        \mu(dx)\,\mu(dy)<\infty.
    \end{equation}
    See, for example, Theorem 5.1 of \ocite{FitzsimmonsSalisbury}.
    \qed
\end{remark}
\begin{remark}
    We mention the following formula for the Hausdorff
    dimension of the $N$-multiple points
    of a L\'evy process $X$: Under
    the conditions stated in Remark \ref{rem:FS},
    \begin{equation}\label{Mn}
        \dimh M_N =\sup\left\{s\in(0\,,d):\
        \int_{\R^d} \frac{[u(z)]^N}{\|z\|^s}\, dz<\infty
        \right\}\quad\text{a.s.,}
    \end{equation}
    where $M_N$ denotes the collection of all $N$-multiple
    points. That is, $M_N$ is the set of all $x\in\R^d$
    for which the cardinality of $X^{-1}(\{x\})$ is at least
    $N$. This formula appears to be new. \ocite{Hawkes:78b}*{Theorem 2}
    contains a similar formula---with $\int_{\R^d}$ replaced by
    $\int_{(-1,1)^d}$---which is shown to be valid for all isotropic
    [spherically symmetric] L\'evy processes that have measurable
    transition densities.

    In order to prove \eqref{Mn}, we appeal to Theorem \ref{th:dimH}
    and the Markov property to first demonstrate that the dimension
    formula in \eqref{Mn} is valid almost surely
    on $\{M_N \ne \varnothing\}$.
    Then the conclusion follows from
    the fact that the event $\{M_N \ne \varnothing\}$ satisfies a
    zero-one law; that zero-one law is itself proved by adapting
    the argument of \ocite{Orey}*{p.\ 124} or
    \ocite{Evans:87b}*{pp.\
    365--366}. We omit the details as the method is nowadays considered
    standard.
    \qed
\end{remark}
\begin{remark}\label{rem:allude}
    This is a natural place to complete a computation that we
    alluded to earlier in Open Problem \ref{pbm:4}. Namely,
    we wish to prove that under the preceding conditions
    on the L\'evy process $X$,
    \begin{equation}\label{M2}
        \dimh M_2=\sup\left\{s\in(0\,,d):\
        \int_{(-1,1)^d} \frac{[u(z)]^2}{\|z\|^s}\, dz<\infty
        \right\}\quad\text{a.s.,}
    \end{equation}
    By the Markov property, it suffices to derive
    \eqref{M2} with $M_2$ replaced by
    $X_1(\R_+)\cap X_2(\R_+)$, where $X_1$ and $X_2$ are
    independent copies of $X$.

    We have seen already that
    if $u\not\in L^2_{\text{\it loc}}(\R^d)$, then
    $X_1(\R_+)\cap X_2(\R_+)=\varnothing$, and there is
    nothing left to prove. Thus, we may consider only the
    case that $u\in L^2_{\text{\it loc}}(\R^d)$. In this
    case, $X_1(\R_+)\cap X_2(\R_+)\neq\varnothing$, which
    as we have seen is equivalent  to the condition that
    $(X_1\ominus X_2)(\R^2_+)$ contains the origin. This
    and Theorem \ref{th:main} together prove that
    $K_\Psi\in L^2(\R^d)$,
    where $K_\Psi(\xi)=\Re(1+\Psi(\xi))^{-1}$
    and $\E\exp(i\xi\cdot X(t))=\exp(-t\Psi(\xi))$
    in the present case. Now Lemma \ref{lem:vhat:kappa}
    shows that $\hat{v}$---and hence $v$---is square integrable, where
    $v:=\frac12(u(\bullet)+u(-\bullet))$. That is, $u\in
    L^2(\R^d)$. From here, it is a simple matter to check
    that \eqref{Mn} with $N=2$ implies \eqref{M2}.
    \qed
\end{remark}
\section{\bf Zero-one laws}

We conclude this paper by deriving two zero-one laws: One for the
Lebesgue measure of the range of additive L\'evy processes;
and another for capacities of the range
of additive L\'evy processes.

The following was proved first in \fullocite{KXZ:03} under a mild
technical condition. Here we remove the technical condition
(1.3) of that paper, and derive this result as an elementary
consequence of Theorem \ref{th:main}.

\begin{proposition}\label{Co:Image}
    Let $\X$ be a general $N$-parameter
    additive L\'evy process in $\R^d$ with
    exponent $\bm{\Psi}$. Then,
    \begin{equation}\label{eq:cor:KXZ:03:1.1}
        \E\left[\lambda_d\left(\X(\R^N_+)\right)\right]>0
        \quad\text{if and only if}\quad
        \int_{\R^d} \prod_{j=1}^N \Re\left(
        \frac{1}{1+\Psi_j(\xi)}\right)\, d\xi<\infty.
    \end{equation}
    Suppose, in addition, that $\widetilde{\X}$
    has a one-potential density that is continuous away
    from the origin, and $\X$ has an a.e.-positive one-potential
    density. Then, with probability one,
    $\lambda_d ( \X(\R^N_+) )$ is zero or infinity.
\end{proposition}

\begin{proof}
    According to Theorem \ref{th:main},
    $\E [ \lambda_d ( \X(\R^N_+) ) ]$ is positive if
    and only if $\mathrm{cap}_{\bm{\Psi}}(\{0\})>0$.
    But the only probability measure on $\{0\}$ is
    $\delta_0$, and hence $\mathrm{cap}_{\bm{\Psi}}(\{0\})>0$
    iff $K_{\bm{\Psi}}\in L^1(\R^d)$.
    This is the integrability condition of \eqref{eq:cor:KXZ:03:1.1},
    and implies the assertion made in \eqref{eq:cor:KXZ:03:1.1}.

    Now the proof of Proposition 6.5 of \ocite{KXZ:03}
    goes through unhindered to conclude the remainder of the proof.
    We include it here for the sake of completeness.
    Throughout the rest of this proof, we assume that $\X$ has
    an a.e.-positive one-potential density.

    Our task is to prove that if $\E[\lambda_d(\X(\R^N_+))]$
    is finite, then it is zero. Note that for all $n>0$,
    \begin{equation}\begin{split}
        &\E\left[\lambda_d\left( \X(\R^N_+) \right)\right] \\
        &\ge
            \E\left[ \lambda_d\left( \X([0\,,n]^N) \right)\right]
            + \E\left[ \lambda_d\left( \X([n\,,\infty)^N) \right)\right]
            - \E\left[\lambda_d\left( \X([0\,,n]^N)
            \cap \X'(\R^N_+)  \right)\right],
    \end{split}\end{equation}
    where $\X'$ is an independent copy of $\X$. Therefore,
    if $\E[\lambda_d(\X(\R^N_+))]$
    is finite, then
    \begin{equation}
        \E\left[\lambda_d\left( \X([0\,,n]^N\right)\right]
        \le \E\left[\lambda_d\left( \X([0\,,n]^N)
        \cap \X'(\R^N_+)  \right)\right].
    \end{equation}
    The inequality is, in fact, an equality.
    Let $n\uparrow\infty$ to deduce that if $\E[\lambda_d(\X(\R^N_+))]$
    is finite, then
    $\E [\lambda_d ( \X(\R^N_+) ) ]
    = \E [\lambda_d ( \X(\R^N_+)
    \cap \X'(\R^N_+)   ) ]$.
    Consider the function
    \begin{equation}
        \phi(a) := \P\left\{ a\in \X(\R^N_+)\right\}\qquad
        \text{for all $a\in\R^d$}.
    \end{equation}
    Then, we have just proved that
    $\int_{\R^d} \phi(a)\, da = \int_{\R^d} |
    \phi(a) |^2\, da$.
    Since $0\le \phi(a)(1-\phi(a))\le 1$ for all $a\in\R^d$,
    it follows that $\phi\in\{0\,,1\}$
    almost everywhere. Consequently, if $\E[\lambda_d
    (\X(\R^N_+))]$ is finite, then
    \begin{equation}\label{eq:Exp:phi}
        \E\left[\lambda_d \left( \X(\R^N_+)\right) \right]
        = \lambda_d\left( \phi^{-1} (\{1\})\right).
    \end{equation}
    According to Proposition \ref{pr:H_F}, either
    $\phi(a)=0$ for all $a$, or $\phi(a)=1$ for all $a$.
    Since $\E[\lambda_d(\X(\R^N_+))]<\infty$,
    this and \eqref{eq:Exp:phi} together prove that
    $\phi(a)=0$ for all $a$, and hence
    $\E[\lambda_d(\X(\R^N_+))]=0$.
\end{proof}

For all $s>0$ define $\mathcal{C}_s(A)$ to be
the $s$-dimensional \emph{Bessel--Riesz capacity} of the Borel set
$A\subseteq\R^d$. That is,
\begin{equation}
    \mathcal{C}_s(A) := \left[ \inf_{\mu\in\mathcal{P}_c(A)}
    \iint \frac{\mu(dx)\,\mu(dy)}{\|x-y\|^s}\right]^{-1},
\end{equation}
where $\inf\varnothing:=\infty$, $1/\infty:=0$,
and we recall that
$\mathcal{P}_c(A)$ denotes the collection of all
compact-support Borel probability measures on $A$.
Thus, $\mathcal{C}_s=\mathcal{C}_{\kappa_{d-s}}$, where
$\kappa_\alpha(x):=\|x\|^{-d+\alpha}$ denotes
the $(d-\alpha)$-dimensional
Riesz kernel. The goal of this subsection is to derive a
zero-one law for the capacity of the range of an arbitrary
additive L\'evy process $\X$.

\begin{proposition}\label{Prop:zero-one}
    If $\X$ denotes an $N$-parameter additive
    L\'evy process in $\R^d$, then for all $\beta\in(0\,,d)$
    fixed, the chances are
    either zero or one that
    $\mathscr{C}_\beta  (\X(\R_+^N)  )$ is strictly positive.
\end{proposition}

The following formula for the Hausdorff dimension of
$\X(\R_+^N)$ is an immediate application of Proposition \ref{Prop:zero-one}
and the methods of \ocite{KX04a}*{Theorem 4.1}: If
$\X$ is an additive L\'evy process with values in $\R^d$ with
characteristic exponent $(\Psi_1\,,\ldots, \Psi_N)$, then almost
surely,
\begin{equation}\label{Eq:rangedim}
	\dimh \left( \X(\R^N_+) \right)  = \sup\left\{ \beta\in(0\,,d):
	\int_{\R^d} \prod_{j=1}^N \Re \left( {1\over 1+ \Psi_j
	(\xi)}\right)\,  {d\xi \over
	\|\xi\|^{d -\beta}} < +\infty
	\right\}.
\end{equation}
Here  $\sup\varnothing := 0$.

\begin{proof}[Proof of Proposition \ref{Prop:zero-one}]
    We choose and fix a $\beta \in (0\,, d)$, and assume that
    \begin{equation}
        \P \{ \mathscr{C}_\beta(\X(\R_+^N)) >0
        \} > 0,
    \end{equation}
    for there is nothing to prove otherwise.

    Let us choose an integer $M \ge 1$ and a real
    number $\alpha \in (0\,, 2]$
    such that $\beta = d - M \alpha$. After enlarging the probability
    space, if need be, we may introduce $M$ i.i.d.\ isotropic
    stable processes $S_1,\ldots,S_M$---independent also
    of all $X_j$'s---such that each $S_j$ has
    stability index $\alpha$. In this way, we can consider also the
    $M$-parameter additive L\'evy process
    \begin{equation}\label{def:frakS}
        \mathfrak{S}(\bm{u}) := S_1(u_1)+\cdots+S_M(u_M)
        \quad\text{for all}\quad\bm{u}:=(u_1\,,\ldots,u_M)\in\R^M_+.
    \end{equation}
    defined on $\R^d$.

    According to Theorem 7.2 of \fullocite{KXZ:03},
    \begin{equation}\label{Eq:Cap-step1}\begin{split}
        \P\left\{ \mathscr{C}_\beta \left(\X(\R_+^N) \right) >0
        \right\} > 0 \ \Leftrightarrow\
        \E\left[ \lambda_d\big(\X(\R_+^N) \oplus
        \mathfrak{S} (\R_+^M)\big)\right] > 0.
    \end{split}\end{equation}
    Because $\X \oplus \mathfrak{S}$ is itself
    an $(N+M)$-parameter additive L\'evy process,
    Proposition \ref{Co:Image} implies that the
    right-hand side of \eqref{Eq:Cap-step1}
    is equivalent to
    \begin{equation}\label{Eq:Cap-step2}
        \int_{\R^d} \left(\frac 1 {1 + \|\xi\|^{\alpha}}
        \right)^M\, \prod_{j=1}^N \Re\left(
        \frac{1}{1+\Psi_j(\xi)}\right)\, d\xi<\infty.
    \end{equation}
    It remains to prove $\P \{\mathscr{C}_\beta  (\X(\R_+^N)  ) > 0 \}=1$.

    The proof is similar to that of Proposition 2.8 in 
    \ocite{KXZ:05}. Define the random probability 
    measure $m$ on $\X(\R_+^N)$ by
    \begin{equation}
        \int_{\R^d} f(x)\, m(dx) = \int_{\R_+^N}
        f(\X(\bm{t}))\, e^{-[\bm{t}]}\, d \bm{t},
    \end{equation}
    where $f: \R^d \to \R_+$ denotes a Borel measurable function.
    It follows from Lemma \ref{lem:ac:fourier}
    and the Fubini-Tonelli theorem that
    \begin{equation}\label{Eq:Cap-step3}
        \E \left[ \iint\frac {m(dx)\, m(dy)}{\|x - y\|^\beta}\right]
        = \frac 1 {(2 \pi)^d} \int_{\R^d} \E
        \left(\left\|\hat{m}(\xi)\right\|^2\right)\,
        \frac{d\xi} {\|\xi\|^{d - \beta}}.
    \end{equation}
    We may observe that
    \begin{equation}\label{Eq:Cap-step4}\begin{split}
        \E\left(\left\|\hat{m}(\xi)\right\|^2\right)
            &= \int_{\R_+^N} \int_{\R_+^N}
            \E\left[ e^{- i \xi\cdot \left(
            \X(\bm{t}) - \X(\bm{s}) \right)}\right]\,
            e^{-[\bm{s}] - [\bm{t}]}\, d \bm{s}\, d \bm{t}\\
        & = \prod_{j=1}^N \left[\int_0^\infty \int_0^\infty e^{-s_j - t_j -
            |s_j - t_j |\Psi_j({\rm sgn}(s_j-t_j)\xi)}\, ds_j\, dt_j\right].
    \end{split}\end{equation}
    A direct computation shows that the preceding is equal to
    $K_{\bm{\Psi}}(\xi)$;
    see also the proof of Lemma \ref{lem:technical}.
    Thanks to this and \eqref{Eq:Cap-step2}, the
    final integral in \eqref{Eq:Cap-step3} is finite, whence it follows
    that $\mathscr{C}_\beta(\X(\R_+^N)) >0$ almost surely.
\end{proof}

\bigskip
\noindent\textbf{Acknowledgements.}
A few years ago, Professor Jean Bertoin suggested to us a
problem that is addressed by
Theorem \ref{th:bertoin} of
the present paper. We thank him wholeheartedly.

We have been writing several versions of this paper since April
2005. During the writing of one of this draft we have received
many preprints by Dr.\ Ming Yang who, among other things,
has independently discovered
Propositions \ref{Co:Image} and \ref{Prop:zero-one}
and equation \eqref{Eq:rangedim}
of the present paper. His method combines
the theory of \fullocite{KXZ:03} with a clever symmetrization idea,
and is worthy of further investigation. We wish to thank Dr.\ Yang
for communicating his work with us.


\begin{bibdiv}
\begin{biblist}
\bib{Aizenman}{article}{
   author={Aizenman, Michael},
   title={The intersection of Brownian paths as a case study of a
   renormalization group method for quantum field theory},
   journal={Comm. Math. Phys.},
   volume={97{\it (1-2)}},
   date={1985},
   pages={91--110},
}
\bib{AlbeverioZhou}{article}{
   author={Albeverio, Sergio},
   author={Zhou, Xian Yin},
   title={Intersections of random walks and Wiener sausages in four
   dimensions},
   journal={Acta Appl.\ Math.},
   volume={45{\it (2)}},
   date={1996},
   pages={195--237},
}
\bib{bergforst}{book}{
    author={Berg, Christian},
    author={Forst, Gunnar},
     title={Potential Theory on Locally Compact Abelian Groups},
 publisher={Springer-Verlag},
     place={New York},
      date={1975},
}
\bib{bertoin:st-flour}{article}{
    author={Bertoin, Jean},
     title={Subordinators: Examples and Applications},
 booktitle={Lectures on Probability Theory and Statistics (Saint-Flour,
            1997)},
    series={Lecture Notes in Math.},
    volume={1717},
     pages={1\ndash 91},
 publisher={Springer},
     place={Berlin},
      date={1999a},
}
\bib{bertoin:1999}{article}{
    author={Bertoin, Jean},
     title={Intersection of independent regenerative sets},
   journal={Probab. Theory Related Fields},
    volume={114{\it (1)}},
      date={1999b},
     pages={97\ndash 121},
}
\bib{bertoin:book}{book}{
    author={Bertoin, Jean},
     title={L\'evy Processes},
     publisher={Cambridge University Press},
     place={Cambridge},
      date={1996},
}
\bib{BG}{book}{
   author={Blumenthal, R. M.},
   author={Getoor, R. K.},
   title={Markov Processes and Potential Theory},
   publisher={Academic Press},
   place={New York},
   date={1968},
}
\bib{Csiszar}{article}{
    author={Csisz{\'a}r, Imre},
     title={A note on limiting distributions on topological groups},
  language={English, with Russian summary},
   journal={Magyar Tud. Akad. Mat. Kutat\'o Int. K\"ozl.},
    volume={9},
      date={1965},
     pages={595\ndash 599 (1965)},
}
\bib{DM}{book}{
   author={Dellacherie, Claude},
   author={Meyer, Paul-Andr{\'e}},
   title={Probabilities and Potential},
   volume={29},
   publisher={North-Holland Publishing Co.},
   place={Amsterdam},
   date={1978},
}
\bib{Doob}{book}{
   author={Doob, Joseph L.},
   title={Classical Potential Theory and Its Probabilistic Counterpart},
   note={Reprint of the 1984 edition},
   publisher={Springer-Verlag},
   place={Berlin},
   date={2001},
}
\bib{dyn1}{article}{
    author={Dynkin, E. B.},
     title={Self-intersection local times, occupation fields, and stochastic
            integrals},
   journal={Adv. in Math.},
    volume={65{\it (3)}},
      date={1987},
     pages={254\ndash 271},
}
\bib{dyn2}{article}{
    author={Dynkin, E. B.},
     title={Generalized random fields related to self-intersections of the
            Brownian motion},
   journal={Proc. Nat. Acad. Sci. U.S.A.},
    volume={83{\it (11)}},
      date={1986},
     pages={3575\ndash 3576},
}
\bib{dyn3}{article}{
    author={Dynkin, E. B.},
     title={Random fields associated with multiple points of the Brownian
            motion},
   journal={J. Funct. Anal.},
    volume={62{\it (3)}},
      date={1985},
     pages={397\ndash 434},
}
\bib{dyn4}{article}{
    author={Dynkin, E. B.},
     title={Local times and quantum fields},
 booktitle={Seminar on stochastic processes, 1983 (Gainesville, Fla., 1983)},
    series={Progr.\ Probab.\ Statist.},
    volume={7},
     pages={69\ndash 83},
 publisher={Birkh\"auser Boston},
     place={Boston, MA},
      date={1984},
}
\bib{dyn5}{article}{
    author={Dynkin, E. B.},
     title={Polynomials of the occupation field and related random fields},
   journal={J. Funct. Anal.},
    volume={58{\it (1)}},
      date={1984},
     pages={20\ndash 52},
      issn={0022-1236},
}
\bib{dyn6}{article}{
    author={Dynkin, E. B.},
     title={Gaussian and non-Gaussian random fields associated with Markov
            processes},
   journal={J. Funct. Anal.},
    volume={55{\it (3)}},
      date={1984},
     pages={344\ndash 376},
}
\bib{dyn7}{article}{
    author={Dynkin, E. B.},
     title={Gaussian random fields and Gaussian evolutions},
 booktitle={Theory and application of random fields (Bangalore, 1982)},
    series={Lecture Notes in Control and Inform.\ Sci.},
    volume={49},
     pages={28\ndash 39},
 publisher={Springer},
     place={Berlin},
      date={1983},
}
\bib{dyn8}{article}{
    author={Dynkin, E. B.},
     title={Markov processes as a tool in field theory},
   journal={J. Funct. Anal.},
    volume={50{\it (2)}},
      date={1983},
     pages={167\ndash 187},
}
\bib{dyn9}{article}{
    author={Dynkin, E. B.},
     title={Markov processes, random fields and Dirichlet spaces},
      note={New stochastic methods in physics},
   journal={Phys. Rep.},
    volume={77{\it (3)}},
      date={1981},
pages={239\ndash 247},
      issn={0370-1573},
}
\bib{dyn10}{article}{
    author={Dynkin, E. B.},
     title={Markov processes and random fields},
   journal={Bull. Amer. Math. Soc. (N.S.)},
    volume={3{\it (3)}},
      date={1980},
     pages={975\ndash 999},
}
\bib{DET:54}{article}{
   author={Dvoretzky, A.},
   author={Erd{\H{o}}s, P.},
   author={Kakutani, S.},
   title={Multiple points of paths of Brownian motion in the plane},
   journal={Bull. Res. Council Israel},
   volume={3},
   date={1954},
   pages={364--371},
}
\bib{DET:50}{article}{
   author={Dvoretzky, A.},
   author={Erd{\H{o}}s, P.},
   author={Kakutani, S.},
   title={Double points of paths of Brownian motion in $n$-space},
   journal={Acta Sci. Math. Szeged},
   volume={12},
   date={1950},
   note={Leopoldo Fejer et Frederico Riesz LXX annos natis dedicatus, Pars
   B},
   pages={75--81},
}
\bib{DvoretzkyErdosKakutaniTaylor}{article}{
    author={Dvoretzky, A.},
    author={Erd\H{o}s, P.},
    author={Kakutani, S.},
    author={Taylor, S. J.},
     title={Triple points of Brownian paths in $3$-space},
   journal={Proc. Cambridge Philos. Soc.},
    volume={53},
      date={1957},
     pages={856\ndash 862},
}
\bib{Evans:87a}{article}{
    author={Evans, Steven N.},
     title={Potential theory for a family of several Markov processes},
  language={English, with French summary},
   journal={Ann. Inst. H. Poincar\'e Probab. Statist.},
    volume={23{\it (3)}},
      date={1987a},
     pages={499\ndash 530},
}
\bib{Evans:87b}{article}{
    author={Evans, Steven N.},
     title={Multiple points in the sample paths of a L\'evy process},
   journal={Probab. Theory Related Fields},
    volume={76{\it (3)}},
      date={1987b},
     pages={359\ndash 367},
}
\bib{farkasjacobschilling}{article}{
   author={Farkas, Walter},
   author={Jacob, Niels},
   author={Schilling, Ren{\'e} L.},
   title={Function spaces related to continuous negative definite functions:
   $\psi$-Bessel potential spaces},
   journal={Dissertationes Math. (Rozprawy Mat.)},
   volume={393},
   date={2001},
   pages={62},
}
\bib{farkasleopold}{article}{
   author={Farkas, Walter},
   author={Leopold, Hans-Gerd},
   title={Characterisations of function spaces of generalised smoothness},
   journal={Ann. Mat. Pura Appl. (4)},
   volume={185{\it (1)}},
   date={2006},
   pages={1--62},
}
\bib{FelderFrohlich}{article}{
   author={Felder, G.},
   author={Fr{\"o}hlich, J.},
   title={Intersection properties of simple random walks: a renormalization
   group approach},
   journal={Comm. Math. Phys.},
   volume={97{\it (1-2)}},
   date={1985},
   pages={111--124},
}
\bib{FitzsimmonsSalisbury}{article}{
    author={Fitzsimmons, P. J.},
    author={Salisbury, Thomas S.},
     title={Capacity and energy for multiparameter Markov processes},
  language={English, with French summary},
   journal={Ann. Inst. H. Poincar\'e Probab. Statist.},
    volume={25{\it (3)}},
      date={1989},
     pages={325\ndash 350},
}
\bib{fristedt}{article}{
    author={Fristedt, Bert},
     title={Sample {F}unctions of {S}tochastic {P}rocesses with
     {S}tationary, {I}ndependent {I}ncrements},
 booktitle={Advances in probability and related topics, Vol. 3},
     pages={241\ndash 396},
 publisher={Dekker},
     place={New York},
      date={1974},
}
\bib{FOT}{book}{
   author={Fukushima, Masatoshi},
   author={{\=O}shima, Y{\=o}ichi},
   author={Takeda, Masayoshi},
   title={Dirichlet Forms and Symmetric Markov Processes},
   volume={19},
   publisher={Walter de Gruyter \& Co.},
   place={Berlin},
   date={1994},
}
\bib{Getoor:book}{book}{
    author={Getoor, R. K.},
     title={Excessive Measures},
 publisher={Birkh\"auser Boston Inc.},
     place={Boston, MA},
      date={1990},
}
\bib{Hawkes:84}{article}{
   author={Hawkes, John},
   title={Some geometric aspects of potential theory},
   conference={
      title={Stochastic analysis and applications},
      address={Swansea},
      date={1983},
   },
   book={
      series={Lecture Notes in Math.},
      volume={1095},
      publisher={Springer},
      place={Berlin},
   },
   date={1984},
   pages={130--154},
}
\bib{Hawkes:79}{article}{
    author={Hawkes, John},
     title={Potential theory of L\'evy processes},
   journal={Proc. London Math. Soc. (3)},
    volume={38{\it (2)}},
      date={1979},
     pages={335\ndash 352},
}
\bib{Hawkes:78a}{article}{
    author={Hawkes, John},
     title={Image and intersection sets for subordinators},
   journal={J. London Math. Soc. (2)},
    volume={17{\it (3)}},
      date={1978a},
     pages={567\ndash 576},
}
\bib{Hawkes:78b}{article}{
    author={Hawkes, John},
     title={Multiple points for symmetric L\'evy processes},
   journal={Math. Proc. Cambridge Philos. Soc.},
    volume={83{\it (1)}},
      date={1978b},
     pages={83\ndash 90},
}
\bib{Hawkes:77}{article}{
    author={Hawkes, John},
     title={Local properties of some Gaussian processes},
   journal={Z. Wahrscheinlichkeitstheorie und Verw. Gebiete},
    volume={40{\it (4)}},
      date={1977},
     pages={309\ndash 315},
}
\bib{Hawkes:76:77}{article}{
   author={Hawkes, John},
   title={Intersections of Markov random sets},
   journal={Z. Wahrscheinlichkeitstheorie und Verw. Gebiete},
   volume={37{\it (3)}},
   date={1976/77},
   pages={243--251},
}
\bib{Hendricks:99}{article}{
    author={Hendricks, W. J.},
     title={Multiple points for transient symmetric L\'evy processes in
            ${\bf R}\sp{d}$},
   journal={Z. Wahrsch. Verw. Gebiete},
    volume={49{\it (1)}},
      date={1979},
     pages={13\ndash 21},
      issn={0044-3719},
}
\bib{Hendricks:74}{article}{
    author={Hendricks, W. J.},
     title={Multiple points for a process in $R\sp{2}$ with stable
            components},
   journal={Z. Wahrscheinlichkeitstheorie und Verw. Gebiete},
    volume={28},
      date={1973/74},
     pages={113\ndash 128},
}
\bib{HendricksTaylor}{article}{
    author={Hendricks, W. J.},
    author={Taylor, S. J.},
    title={Concerning some problems about polar sets for processes
        with stationary independent increments},
    status={preprint},
    year={1979}
}
\bib{Hirsch}{article}{
    author={Hirsch, Francis},
     title={Potential theory related to some multiparameter processes},
   journal={Potential Anal.},
    volume={4{\it (3)}},
      date={1995},
     pages={245\ndash 267},
}
\bib{HS:99}{article}{
    author={Hirsch, Francis},
    author={Song, Shiqi},
     title={Multiparameter Markov processes and capacity},
 booktitle={Seminar on Stochastic Analysis, Random Fields and Applications
            (Ascona, 1996)},
    series={Progr.\ Probab.},
    volume={45},
     pages={189\ndash 200},
 publisher={Birkh\"auser},
     place={Basel},
      date={1999},
}
\bib{HS:96}{article}{
    author={Hirsch, Francis},
    author={Song, Shiqi},
     title={Inequalities for Bochner's subordinates of two-parameter
            symmetric Markov processes},
  language={English, with English and French summaries},
   journal={Ann. Inst. H. Poincar\'e Probab. Statist.},
    volume={32{\it (5)}},
      date={1996},
     pages={589\ndash 600},
}
\bib{HS:95a}{article}{
    author={Hirsch, Francis},
    author={Song, Shiqi},
     title={Markov properties of multiparameter processes and capacities},
   journal={Probab. Theory Related Fields},
    volume={103{\it (1)}},
      date={1995a},
     pages={45\ndash 71},
}
\bib{HS:95b}{article}{
    author={Hirsch, Francis},
    author={Song, Shiqi},
     title={Symmetric Skorohod topology on $n$-variable functions and
            hierarchical Markov properties of $n$-parameter processes},
   journal={Probab. Theory Related Fields},
    volume={103{\it (1)}},
      date={1995b},
     pages={25\ndash 43},
}
\bib{HS:95c}{article}{
    author={Hirsch, Francis},
    author={Song, Shiqi},
     title={Une in\'egalit\'e maximale pour certains processus de Markov \`a
            plusieurs param\`etres. II},
  language={French},
   journal={C. R. Acad. Sci. Paris S\'er. I Math.},
    volume={320{\it (7)}},
      date={1995c},
     pages={867\ndash 870},
}
\bib{HS:95d}{article}{
    author={Hirsch, Francis},
    author={Song, Shiqi},
     title={Une in\'egalit\'e maximale pour certains processus de Markov \`a
            plusieurs param\`etres. I},
  language={French},
   journal={C. R. Acad. Sci. Paris S\'er. I Math.},
    volume={320{\it (6)}},
      date={1995d},
     pages={719\ndash 722},
}
\bib{HS:94}{article}{
    author={Hirsch, Francis},
    author={Song, Shiqi},
     title={Propri\'et\'es de Markov des processus \`a plusieurs
            param\`etres et capacit\'es},
  language={French},
   journal={C. R. Acad. Sci. Paris S\'er. I Math.},
    volume={319{\it (5)}},
      date={1994},
     pages={483\ndash 488},
}
\bib{hunt:III}{article}{
   author={Hunt, G. A.},
   title={Markoff processes and potentials. III},
   journal={Illinois J. Math.},
   volume={2},
   date={1958},
   pages={151--213},
}
\bib{hunt:II}{article}{
    author={Hunt, G. A.},
     title={Markoff processes and potentials. I, II},
   journal={Illinois J. Math.},
    volume={1},
      date={1957},
     pages={316\ndash 369},
}
\bib{hunt:I}{article}{
    author={Hunt, G. A.},
     title={Markoff processes and potentials. I, II},
   journal={Illinois J. Math.},
    volume={1},
      date={1957},
     pages={44\ndash 93},
}
\bib{hunt:56}{article}{
   author={Hunt, G. A.},
   title={Markoff processes and potentials},
   journal={Proc. Nat. Acad. Sci. U.S.A.},
   volume={42},
   date={1956},
   pages={414--418},
}
\bib{jacob3}{book}{
   author={Jacob, Niels},
   title={Pseudo Differential Operators and Markov Processes. Vol. III:
    Markov Processes and Applications},
   publisher={Imperial College Press},
   place={London},
   date={2005},
}
\bib{jacob2}{book}{
   author={Jacob, Niels},
   title={Pseudo Differential Operators \& Markov Processes. Vol. II:
    Generators and Their Potential Theory},
   publisher={Imperial College Press},
   place={London},
   date={2002},
}
\bib{jacob1}{book}{
    author={Jacob, Niels},
     title={Pseudo-Differential Operators and Markov Processes.
     Vol. I: Fourier Analysis and Semigroups},
 publisher={Imperial College Press},
     place={London},
      date={2001},
}
\bib{jacobschilling}{article}{
   author={Jacob, Niels},
   author={Schilling, Ren{\'e} L.},
   title={Function spaces as Dirichlet spaces (about a paper by W.
   Maz\cprime ya and J. Nagel). Comment on: ``On equivalent standardization
   of anisotropic functional spaces $H\sp \mu({\bf R}\sp n)$'' (German)
   [Beitr\"age Anal. No. 12 (1978), 7--17; MR0507094]},
   journal={Z. Anal. Anwendungen},
   volume={24{\it (1)}},
   date={2005},
   pages={3--28},
}
\bib{Kahane:SRSF}{book}{
    author={Kahane, Jean-Pierre},
     title={Some Random Series of Functions},
   edition={2},
 publisher={Cambridge University Press},
     place={Cambridge},
      date={1985},
}
\bib{Kakutani:44a}{article}{
   author={Kakutani, Shizuo},
   title={On Brownian motions in $n$-space},
   journal={Proc. Imp. Acad. Tokyo},
   volume={20},
   date={1944},
   pages={648--652},
}
\bib{Kakutani:44b}{article}{
   author={Kakutani, Shizuo},
   title={Two-dimensional Brownian motion and harmonic functions},
   journal={Proc. Imp. Acad. Tokyo},
   volume={20},
   date={1944},
   pages={706--714},
}
\bib{Kesten}{book}{
    author={Kesten, Harry},
     title={Hitting Probabilities of Single Points for Processes with
            Stationary Independent Increments},
    series={Memoirs of the American Mathematical Society, No. 93},
 publisher={American Mathematical Society},
     place={Providence, R.I.},
      date={1969},
}
\bib{Kh:2003}{article}{
   author={Khoshnevisan, Davar},
   title={Intersections of Brownian motions},
   journal={Expo. Math.},
   volume={21{\it (2)}},
   date={2003},
   pages={97--114},
   issn={0723-0869},
}
\bib{Kh:1999}{article}{
    author={Khoshnevisan, Davar},
     title={Brownian sheet images and Bessel-Riesz capacity},
   journal={Trans. Amer. Math. Soc.},
    volume={351{\it (7)}},
      date={1999},
     pages={2607\ndash 2622},
}
\bib{Kh:book}{book}{
    author={Khoshnevisan, Davar},
     title={Multiparameter Processes: An Introduction to Random Fields},
 publisher={Springer-Verlag},
     place={New York},
      date={2002},
}

\bib{KSX07}{article}{
       author = {Khoshnevisan, Davar},
       author = {Shieh, Narn-Rueih},
       author = {Xiao, Yimin},
        title = {Hausdorff dimension of the contours of symmetric additive L\'evy processes},
      journal = {Probab. Th. Rel. Fields},
        pages = {to appear},
         year = {2007},
    }
\bib{KX04a}{book}{
       author = {Khoshnevisan, Davar},
       author = {Xiao, Yimin},
       title = {Additive L\'evy processes: capacity and Hausdorff dimension},
       series = {In: Proc. of Inter. Conf. on
			Fractal Geometry and Stochastics III., Progress in Probability},
       volume = {57, pp. 62--100},
       publisher={Birkh\"auser},
         year = {2004},
    }
\bib{KXZ:05}{article}{
    author={Khoshnevisan, Davar},
    author={Xiao, Yimin},
     title={L\'evy processes: capacity and
Hausdorff dimension},
   journal={Ann. Probab.},
    volume={33{\it (3)}},
      date={2005},
     pages={841\ndash 878},
}

\bib{KXZ:03}{article}{
    author={Khoshnevisan, Davar},
    author={Xiao, Yimin},
    author={Zhong, Yuquan},
     title={Measuring the range of an additive L\'evy process},
   journal={Ann. Probab.},
    volume={31{\it (2)}},
      date={2003},
     pages={1097\ndash 1141},
}
\bib{Lawler:89}{article}{
   author={Lawler, Gregory F.},
   title={Intersections of random walks with random sets},
   journal={Israel J. Math.},
   volume={65{\it (2)}},
   date={1989},
   pages={113--132},
}
\bib{Lawler:85}{article}{
   author={Lawler, Gregory F.},
   title={Intersections of random walks in four dimensions. II},
   journal={Comm. Math. Phys.},
   volume={97{\it (4)}},
   date={1985},
   pages={583--594},
}
\bib{Lawler:82}{article}{
   author={Lawler, Gregory F.},
   title={The probability of intersection of independent random walks in
   four dimensions},
   journal={Comm. Math. Phys.},
   volume={86},
   date={1982},
   number={4},
   pages={539--554},
}
\bib{LeGall92}{article}{
   author={Le Gall, Jean-Fran{\c{c}}ois},
   title={Some properties of planar Brownian motion},
   conference={
      title={\'Ecole d'\'Et\'e de Probabilit\'es de Saint-Flour XX---1990},
   },
   book={
      series={Lecture Notes in Math.},
      volume={1527},
      publisher={Springer},
      place={Berlin},
   },
   date={1992},
   pages={111--235},
}
\bib{LeGall87}{article}{
    author={Le Gall, Jean-Fran{\c{c}}ois},
    title={Le comportement du mouvement brownien entre les
    deux instants o\`u il passe par un point double},
   journal={J. Funct. Anal.},
    volume={71{\it (2)}},
      date={1987},
     pages={246\ndash 262},
}
\bib{LeGallShiehRosen}{article}{
    author={Le Gall, Jean-Fran{\c{c}}ois},
    author={Rosen, Jay S.},
    author={Shieh, Narn-Rueih},
     title={Multiple points of L\'evy processes},
   journal={Ann. Probab.},
    volume={17{\it (2)}},
      date={1989},
     pages={503\ndash 515},
}
\bib{Levy}{article}{
   author={L{\'e}vy, Paul},
   title={Le mouvement brownien plan},
   language={French},
   journal={Amer. J. Math.},
   volume={62},
   date={1940},
   pages={487--550},
}
\bib{MarcusRosen:b}{article}{
    author={Marcus, Michael B.},
    author={Rosen, Jay},
     title={Multiple Wick product chaos processes},
   journal={J. Theoret. Probab.},
    volume={12{\it (2)}},
      date={1999b},
     pages={489\ndash 522},
}
\bib{MarcusRosen:a}{article}{
    author={Marcus, Michael B.},
    author={Rosen, Jay},
     title={Renormalized self-intersection local times and Wick power chaos
            processes},
   journal={Mem. Amer. Math. Soc.},
    volume={142 {\it (675)}},
      date={1999a},
}
\bib{masjanagel}{article}{
   author={Masja, Wladimir},
   author={Nagel, J{\"u}rgen},
   title={\"Uber \"aquivalente Normierung der anisotropen Funktionalr\"aume
   $H\sp{\mu }({\bf R}\sp{n})$},
   language={German},
   journal={Beitr\"age Anal.},
   number={12},
   date={1978},
   pages={7--17},
}

\bib{Mattila}{book}{
     author={Mattila, Pertti}
     title ={Geometry of Sets and Measures in Euclidean Spaces},
     publisher={Cambridge University Press},
     place   = {Cambridge}
     date = {1995}
}
\bib{Orey}{incollection}{
   author={Orey, Steven},
   title={Polar sets for processes with stationary independent increments},
   booktitle={Markov Processes and Potential Theory (Proc. Sympos. Math. Res.
      Center, Madison, Wis., 1967)},
   publisher={Wiley},
   place={New York},
   date={1967},
   pages={117--126},
}
\bib{PPS}{article}{
   author={Pemantle, Robin},
   author={Peres, Yuval},
   author={Shapiro, Jonathan W.},
   title={The trace of spatial Brownian motion is capacity-equivalent to the
     unit square},
   journal={Probab. Theory Related Fields},
   volume={106{\it (3)}},
   date={1996},
   pages={379--399},
}
\bib{Peres:99}{article}{
    author={Peres, Yuval},
     title={Probability on Trees: An Introductory Climb},
 booktitle={Lectures on probability theory and statistics (Saint-Flour,
            1997)},
    series={Lecture Notes in Math.},
    volume={1717},
     pages={193\ndash 280},
 publisher={Springer},
     place={Berlin},
      date={1999},
}
\bib{Peres:96a}{article}{
    author={Peres, Yuval},
     title={Remarks on intersection-equivalence and capacity-equivalence},
  language={English, with English and French summaries},
   journal={Ann. Inst. H. Poincar\'e Phys. Th\'eor.},
    volume={64{\it (3)}},
      date={1996a},
     pages={339\ndash 347},
}
\bib{Peres:96b}{article}{
    author={Peres, Yuval},
     title={Intersection-equivalence of Brownian paths and certain branching
            processes},
   journal={Comm. Math. Phys.},
    volume={177{\it (2)}},
      date={1996b},
     pages={417\ndash 434},
}
\bib{Ren}{article}{
   author={Ren, Jia Gang},
   title={Topologie $p$-fine sur l'espace de Wiener et th\'eor\`eme des
   fonctions implicites},
   language={French, with English summary},
   journal={Bull. Sci. Math.},
   volume={114{\it (2)}},
   date={1990},
   pages={99--114},
}
\bib{rockner}{article}{
   author={R{\"o}ckner, Michael},
   title={General Theory of Dirichlet Forms and Applications},
   conference={
      title={In: Dirichlet forms},
      address={Varenna},
      date={1992},
   },
   book={
      series={Lecture Notes in Math.},
      volume={1563},
      publisher={Springer},
      place={Berlin},
   },
   date={1993},
   pages={129--193},
}
\bib{rogers}{article}{
    author={Rogers, L. C. G.},
     title={Multiple points of Markov processes in a complete metric space},
 booktitle={S\'eminaire de Probabilit\'es, XXIII},
    series={Lecture Notes in Math.},
    volume={1372},
     pages={186\ndash 197},
 publisher={Springer},
     place={Berlin},
      date={1989},
}
\bib{rosen:1984}{article}{
    author={Rosen, Jay},
     title={Self-intersections of random fields},
   journal={Ann. Probab.},
    volume={12{\it (1)}},
      date={1984},
     pages={108\ndash 119},
}
\bib{rosen:1983}{article}{
    author={Rosen, Jay},
     title={A local time approach to the self-intersections of Brownian
            paths in space},
   journal={Comm. Math. Phys.},
    volume={88{\it (3)}},
      date={1983},
     pages={327\ndash 338},
}
\bib{Salisbury:96}{article}{
   author={Salisbury, Thomas S.},
   title={Energy, and intersections of Markov chains},
   conference={
      title={Random discrete structures},
      address={Minneapolis, MN},
      date={1993},
   },
   book={
      series={IMA Vol.\ Math.\ Appl.},
      volume={76},
      publisher={Springer},
      place={New York},
   },
   date={1996},
   pages={213--225},
}
\bib{Salisbury:92}{article}{
   author={Salisbury, Thomas S.},
   title={A low intensity maximum principle for bi-Brownian motion},
   journal={Illinois J. Math.},
   volume={36{\it (1)}},
   date={1992},
   pages={1--14},
}
\bib{Salisbury:87}{article}{
   author={Salisbury, Thomas S.},
   title={Brownian bitransforms},
   conference={
      title={Seminar on Stochastic Processes, 1987},
      address={Princeton, NJ},
      date={1987},
   },
   book={
      series={Progr.\ Probab.\ Statist.},
      volume={15},
      publisher={Birkh\"auser Boston},
      place={Boston, MA},
   },
   date={1988},
   pages={249--263},
}
\bib{sato}{book}{
    author = {Sato, Ken-iti},
     title = {L\'evy Processes and Infinitely Divisible Distributions},
      note = {Translated from the 1990 Japanese original,
              Revised by the author},
 publisher = {Cambridge University Press},
   address = {Cambridge},
      year = {1999},
}
\bib{Schoenberg:38}{article}{
    author={Schoenberg, I. J.},
     title={Metric spaces and positive definite functions},
   journal={Trans.\@ Amer.\@ Math.\@ Soc.},
    volume={44},
      date={1938},
     pages={522\ndash 536},
}
\bib{slobodecki}{article}{
   author={Slobodecki{\u\i}, L. N.},
   title={S. L. Sobolev's spaces of fractional order and their application
    to boundary problems for partial differential equations},
   language={Russian},
   journal={Dokl. Akad. Nauk SSSR (N.S.)},
   volume={118},
   date={1958},
   pages={243--246},
}
\bib{Tongring}{article}{
   author={Tongring, Nils},
   title={Which sets contain multiple points of Brownian motion?},
   journal={Math.\ Proc. Cambridge Philos. Soc.},
   volume={103{\it (1)}},
   date={1988},
   pages={181--187},
}

\bib{Ville}{article}{
   author={Ville, Jean},
   title={Sur un probl\`eme de g\'eom\'etrie sugg\'er\'e par l'\'etude du
   mouvement brownien},
   language={French},
   journal={C. R. Acad. Sci Paris},
   volume={215},
   date={1942},
   pages={51--52},
}
\bib{Walsh}{article}{
   author={Walsh, John B.},
   title={Martingales with a Multidimensional Parameter and Stochastic
   Integrals in the Plane},
   conference={
      title={Lectures in probability and statistics},
      address={Santiago de Chile},
      date={1986},
   },
   book={
      series={Lecture Notes in Math.},
      volume={1215},
      publisher={Springer},
      place={Berlin},
   },
   date={1986},
   pages={329--491},
}
\bib{Westwater:I}{article}{
   author={Westwater, J.},
   title={On Edwards' model for long polymer chains},
   journal={Comm. Math. Phys.},
   volume={72{\it (2)}},
   date={1980},
   pages={131--174},
}
\bib{Westwater:II}{article}{
   author={Westwater, John},
   title={On Edwards' model for polymer chains. II. The self-consistent
   potential},
   journal={Comm. Math. Phys.},
   volume={79{\it (1)}},
   date={1981},
   pages={53--73},
}
\bib{Westwater:III}{article}{
   author={Westwater, John},
   title={On Edwards' model for polymer chains. III. Borel summability},
   journal={Comm. Math. Phys.},
   volume={84{\it (4)}},
   date={1982},
   pages={459--470},
}
\bib{Wolpert}{article}{
   author={Wolpert, Robert L.},
   title={Local time and a particle picture for Euclidean field theory},
   journal={J. Funct. Anal.},
   volume={30{\it (3)}},
   date={1978},
   pages={341--357},
}
%
\end{biblist}
\end{bibdiv}
\end{document}